\documentclass[a4paper, 11pt]{article}

\usepackage{mwe}
\usepackage[markcase=noupper
]{scrlayer-scrpage}
\ohead{}
\usepackage[totalwidth=17cm,totalheight=25cm]{geometry}

\usepackage{makeidx}

\usepackage{psfrag}
\usepackage{epsfig}

\usepackage{amsmath,amssymb,amscd,amsthm}
\usepackage{amsfonts}
\usepackage{subfigure}
\usepackage{latexsym}
\usepackage{graphics}
\usepackage{picins}
\usepackage{colortbl}
\usepackage{color}
\usepackage{ae}
\usepackage{enumitem}
\usepackage[all]{xy}
\usepackage{tikz}
\usepackage{cancel}
\usepackage{bbm}

\usepackage[T1]{fontenc}
\usepackage[latin1]{inputenc}
\usepackage{ae}

\usepackage{lmodern}
\usepackage[english,russian]{babel}
\usepackage[T2A]{fontenc}
\usepackage[utf8]{inputenx}

\usepackage[font=small]{caption}

\usepackage{relsize}

\numberwithin{equation}{section}

\newcommand{\arc}{K}
\newcommand{\param}{z}

\renewcommand{\C}{\mathbb C}
\newcommand{\R}{\mathbb R}

\newcommand{\curv}{\mathrm{K}}

\newcommand{\N}{{\mathbb N}}

\newcommand{\grad}{\operatorname{grad}}

\newcommand{\res}{Reshetnyak }

\newcommand{\normal}{\nu}

\newcommand{\const}{-\frac{1}{2\pi}}

\newcommand{\sing}{\mathcal{I}_\mes}

\newcommand{\I}{\operatorname{i}}
\newcommand{\D}{\operatorname{d}}
\newcommand{\E}{\operatorname{e}}

\newcommand{\mes}{\omega}

\usepackage{color}

\newcommand{\dc}{$\delta$-convex }

\newtheorem{definition}{Definition}[section]
\newtheorem{lemma}[definition]{Lemma}
\newtheorem{theorem}[definition]{Theorem}
\newtheorem{corollary}[definition]{Corollary}
\newtheorem{proposition}[definition]{Proposition}
\newtheorem{example}[definition]{Example}
\newtheorem{remark}[definition]{Remark}
 
 \newtheorem{question}[definition]{Question}

\usepackage[framemethod=tikz,footnoteinside=false,nobreak=true]{mdframed}

\theoremstyle{plain}
\newmdtheoremenv[
  linecolor=cyan,
  roundcorner=5pt,
  linewidth=1pt
]{theo}[definition]{Theorem}
\theoremstyle{nonumberplain}
\newmdtheoremenv[
  linecolor=cyan,
  roundcorner=5pt,
  linewidth=1pt
]{statment}{Statement}

 \makeindex

\begin{document} 
\selectlanguage{english}
\setcounter{secnumdepth}{3}
\setcounter{tocdepth}{4}

\title{An introduction to Reshetnyak's theory of subharmonic distances}
\author{Fran\c{c}ois Fillastre}
\date{Version 3 \\ \today} 

\maketitle

\abstract{The aim of the present text is to provide some basics around Reshetnyak's theory of subharmonic distances, together with an overview of the main results. This text is intended to be a complement to the English translation of  \cite{R1954}, \cite{Res1959}, \cite{R60I}, \cite{R60II}, \cite{R61}, \cite{R61b}, \cite{R62}, \cite{R63}, \cite{R63III} and \cite{huber}. They all will be published in a same volume \cite{livre}. 
While subharmonic distances are often confused with distances of bounded curvature, we present them as a complete autonomous theory. In turn, there is no specific prerequisite for the present text.}

\tableofcontents

\section{Introduction}

Given a set of admissible curves in a plane domain $M$, and 
 a function $\lambda$ such that its square root is integrable along any admissible curve, 
 an intrinsic pseudo-distance can be defined over $M$. Indeed, 
the integration of  the square root of $\lambda$ over an admissible curve defines the length of the curve, and the infimum of these lengths between two points is defined as the distance between the points. If $\lambda$ is $C^\infty$, one obtains a Riemannian distance, whose metric tensor is conformal to the Euclidean one. Another basic famous examples are flat metrics with conical singularities, for which $\lambda$ is a sum of $|\cdot - z_0|^\gamma$, with $z_0$ any point in the plane, and $\gamma$ any real number. What should be the good regularity of $\lambda$ to develop a theory of ``Riemannian geometry with low regularity'' on surfaces, containing the Riemannian case and the case of flat metrics? This has been answered in a series of articles by Yu. G. Reshetnyak in the $60'$s: the logarithm of $\lambda$ is a difference of subharmonic functions, and admissible curves are  rectifiable arcs (for the Euclidean distance). In this case, the pseudo-distance is a distance, called a \emph{subharmonic distance}. 
 
 A subharmonic distance comes with a natural notion of curvature measure, which, in the Riemannian case, is the sectional curvature times the area measure. Actually, the $\lambda$ term in a subharmonic distance is better defined from the potential of a measure plus a harmonic mapping. 
  A central result in the theory relates the converge of the measures to  the (uniform) convergence of the distances. One implication of this theorem is  that, on a given domain, the set of subharmonic distances is exactly the set of the limits (for the uniform convergence) of  sequences of Riemannian distances, with uniformly bounded total variation of the curvature measure, preserving the topology.\footnote{Relaxing the curvature bound assumption would be too general, see e.g., \cite{cassorla}. See \cite{burago-limit} if the condition for the limit metric to have the topology of a surface is relaxed.} 
  
Y. Reshetnyak and A. Huber proved simultaneously that a distance preserving mapping between domains with  subharmonic distances is a conformal mapping, giving a bridge with the theory of Riemann surfaces. Also, using his convergence theorem, \res proved that  subharmonic distances and two-dimensional manifolds of bounded curvatures,  in the sense of A. D. Alexandrov, are (locally) the same thing, providing an analytical approach to a classical object from the theory of metric geometry of surfaces. 

Section~\ref{sec:curve} contains a very general presentation of the two extreme cases of subharmonic metrics on a plane domain: conformal Riemannian metrics and flat cones. It also contains an introduction to  the theory of curves of bounded rotation, that are the most helpful class of curves of the theory. Section~\ref{sec:sh} contains  basics of potential theory, that are needed in the sequel. It also contains a presentation of Reshetnyak's result about localization of $\delta$-subharmonic functions, that is one of the main technical tool used in the sequel. 

  Section~\ref{sec:distance} contains basic facts about metric spaces, which are mixed with results from Section~\ref{sec:sh} about subharmonic mappings, in order to define subharmonic metrics on plane domains.  Then, we present the heart of \res theory, that are convergence theorems. At the end of the section, we present a theorem about contraction onto a cone, that allows to deduce that subharmonic distances are two-dimensional manifolds of bounded curvature, in the sense of Alexandrov.
  
  In Section~\ref{sec:conforme}, we first present the striking result mentioned above, that says that a distance preserving mapping between two plane domains, endowed with subharmonic distance, is a conformal mapping. Moreover, from a formula of Huber, 
   subharmonic metrics can be defined over Riemann surface. Indeed, as it was noted by M. Troyanov \cite{troyanov-alexandrov}, there is a bijection between the set of Radon measure over a compact Riemann surface satisfying Gauss--Bonnet Formula, and subharmonic distances with a given total area on the surface.

In the  sections mentioned above, all the arguments should be complete, by referring to the original sources (with the exception of Theorem~\ref{thm uber chgt coord}, see Remark~\ref{rem:huber rem aires}). In Section~\ref{sec-other}, we present some related results and questions, mainly around relations between subharmonic distances and two-dimensional manifolds of bounded curvature. 

The present text is not a survey about two-dimensional manifolds of bounded curvature, nor Alexandrov surfaces.\footnote{The term ``Alexandrov surface''\index{Alexandrov surface} is equivocal in the literature. Sometimes, it means a two-dimensional manifold of bounded curvature, sometimes if means a metric surface with a curvature bound in the sense of Alexandrov, and sometimes a metric surface with a lower curvature bound. In turn, we won't use further this terminology.} Concerning surfaces with curvature bound, they are covered in many textbooks \cite{bbi}, \cite{BH}, \cite{alexander-kapovich-petrunin}, \cite{alexandrov-book}. Concerning two-dimensional manifolds of bounded curvature, we only know the original old-fashioned \cite{AZ}, together with \cite{AZ2}. The text \cite{R93} is a survey about both two-dimensional manifolds of bounded curvature and subharmonic distances, but it also contains  some  proofs, some of which we will refer to in the present text. The last pages of \cite{R93} presents the state of the art about two-dimensional manifolds of bounded curvature at the time. Let us mention another survey  related to the content of the present article: \cite{R01} ---concerning the specific question of curves of bounded rotation, see \cite{res2005}. And of course, the survey \cite{troyanov-alexandrov}, which renewed attraction to this theory.

The present text is also not a survey about the literature around subharmonic distances. For recent results about this theory, let us only mention \cite{richard}, \cite{debin}. For results using 
subharmonic distances, let us mention \cite{burago-buyalo}, \cite{tro-principe}, \cite{kokarev}, \cite{lytchak-wenger}, \cite{AB}, \cite{bartolucci}, \cite{Veronelli}.

\subsection*{Terminology}\label{sec:terminology}

In this section, we clarify some points about terminology and choices of translation.

A first comment is that, for a metric space $(X,d)$, the function $d:X\times X\to \R$ is always called a \emph{distance} in the present text, the term \emph{metric} is used for $(0,2)$-tensors. We apologize that we are not following the terminology used in the articles of Reshetnyak. In those articles, what we are calling a  distance is called a metric, and what we are calling a metric is called a \emph{linear element}\index{linear element}, as it was usual at the time. 
The term \emph{distance} is used by \res to speak about the value of $d$ between two given points.

We have kept the term \emph{polyhedron} from Reshetnyak's articles, even if nowadays, the objects called polyhedra here are more often called \emph{polyhedral surfaces}\index{polyhedral surface}, or more precisely in the present context, \emph{flat surfaces (with conical singularities)}\index{flat surfaces (with conical singularities)}.

A very confusing point is about rotation and turn of curves, which are defined in Section~\ref{sec:curve}.
There are many different names for the turn and the rotation in the literature (and also depending on the translation). We followed the terminology in Reshetnyak's articles. It is also the one used in \cite{R93}. It is summarized in Table~\ref{table turn}. In the case of a Riemannian metric, the turn defined in \eqref{eq:def turn r} is usually called 
\emph{integral curvature}\index{integral curvature}, \emph{total curvature}\index{total curvature} or  \emph{integral geodesic curvature}\index{integral geodesic curvature}, and this terminology may be used for curves in more general metric spaces.  The turn was also called \emph{swerve}\index{swerve} by H. Busemann, see the footnote page 26 in \cite{pogorelov}.

\begin{table}
\begin{center}
\begin{tabular}{ c | c }
 General metric & Euclidean metric \\ 
 \hline
turn & rotation \\  
$\kappa_\lambda$, $\kappa_l$, $\kappa_r$ & $\kappa$ \\
 \selectlanguage{russian}{поворот} &  \selectlanguage{russian}{вращение} \\
povorot  &   vrashchenie
\end{tabular}
\selectlanguage{english}
\end{center}
\caption{Turn and rotation.}\label{table turn}
\end{table}

For a measure $\mes$, points $z$ with $\mes(\{z\})\geq 2\pi$ will play a particular role. They can be called \emph{cusp points}\index{cusp}, \emph{peak points}\index{peak point}, \emph{return points}\index{return points}, \emph{spike points}\index{spike point}. However, as no special terminology is used in Reshetnyak's articles, we also won't name those points.

\begin{theo}
A statement inside a framed box has an equivalent version in  an article of Yu. Reshetnyak in \cite{livre}.
\end{theo}

  \paragraph{Acknowledgment}  The author is very grateful to Alexander Lytchak for his comments, which significantly enhanced the interest and the presentation of the present text.

\section{Curves in the plane}\label{sec:curve}

The first part of this section  is a kind of lumber room. It takes the opportunity of introducing the concepts of rotation and turn of curves,  to  give some flavor about  Reshetnyak's construction. For example, we encounter the two main examples of subharmonic distances, Riemannian conformal metrics and flat cones.

The second part of this section is intended to be a comprehensive introduction to plane curves of bounded rotation. They are one of the main object of the theory, and, despite the fact that most of all the needed results are contained in \cite{AR}, it seemed more convenient to give a detailed presentation.

\subsection{Rotation, turn and Laplacian}\label{sec regular arcs}

\subsubsection{Rotation of arcs in the plane}\label{sec:rot arc pla}

Let us recall some general definitions. Let $X$ be a Hausdorff topological space. A \emph{parameterized path}\index{parameterized path} in $X$ is a continuous mapping $\param:[a,b]\to X$, $a<b$.
We will restrict ourselves to the simpler case
of  \emph{parameterized arcs}\index{parameterized arc}, where $\param$ is asked to be injective.
Sometimes we will consider that $\param(a)=\param(b)$,  in this case we will say that the arc parameterizes a  \emph{simple closed curve}\index{simple closed curve}. The restriction to arcs will facilitate the following notion of equivalence, and is not so restrictive, as any path contains an arc connecting its endpoints \cite[Lemma 3.1]{Falconer1}. 
Two parameterized arcs $\param_1:[a,b]\to X$ and $\param_2:[c,d]\to X$ are said to be \emph{equivalent}, if there
exists a function $\psi:[c,d]\to [a,b]$ that is strictly monotonic and surjective (hence a homeomorphism), and
that satisfies $\param_2=\param_1\circ \psi$. This is an equivalence relation. A class of parameterized arcs will be called an \emph{arc}. We will confuse this equivalence class with the subset of the plane which is the  image of a parameterization. 
For an arc $\arc$, $\arc^\circ$\index{$\arc^\circ$} is $\arc$ minus its extremities.
 A choice of an element in the equivalence class is called a \emph{parameterization of the arc}.
Note that our arcs are oriented. 
Let $\param:[a,b]\to X$ be a parameterized arc.
As $\param([a,b])$ is compact hence closed, it is easy to see that the mapping $\param$ is indeed closed, hence a homeomorphism onto its image. 

A sequence of arcs $(\arc_n)_n$  \emph{converges}\index{convergence (arcs)} to an arc $\arc$, if there are parameterizations of these arcs on the same segment $[a,b]$,
 which uniformly converge as functions.

 We now look at arcs in the plane. For latter convenience, we will look at the complex plane $\C$.
  Let $L$ be a \emph{broken line}\index{broken line} in the plane, that is an arc which is a union of segments. Its (Euclidean) length is the sum of the lengths of the segments composing $L$, i.e., the sum of the distances between its vertices. 
A broken line $L$ is said to be \emph{inscribed}\index{inscribed} in an arc $K$, if the vertices of $L$ are in the interior of $K$,  the extremities of $K$ and $L$ coincide, and the vertices of $L$ are ordered for a parameterization of $K$. The \emph{(Euclidean) length}\index{length (Euclidean)} 
$s(\arc)$\index{$s(\arc)$} of $\arc$ is then the supremum of the lengths of all the broken lines inscribed in $\arc$. It is well-defined in $\R\cup \{+\infty\}$. Then $\arc$ is called \emph{rectifiable}\index{rectifiable} if its length is finite.
Let us recall some relevant properties of rectifiable arcs:
\begin{itemize}
\item the length $s$ is lower lower-semicontinuous on the set of arcs, i.e., if $\arc_n\to  \arc$, then  
$$s(\arc)\leq \liminf_n s(\arc_n)~; $$
\item if $\arc_n\to \arc$ and $s(\arc_n)$ are uniformly bounded from above, then $\arc$ is rectifiable; 
\item from any set of arcs with length uniformly bounded from above, and contained in a compact set, one can extract a converging subsequence.
\end{itemize}

If $\param:[a,b]\to\C$ is a parameterization of an arc $\arc$, let us denote by $\arc\vert_{ t_1,t_2}$ the arc parameterized by the restriction of $\param$ to $[t_1,t_2]$, and denote $ \arc\vert_{ t}:=\arc\vert_{ a,t}$. A parameterized arc  is \emph{parameterized by arc length}\index{parameterization by arc length}  if, for any $a\leq t_{i-1}\leq t_i\leq b$, we have $t_i-t_{i-1}=s(\arc|_{t_{i-1},t_i})$. Any rectifiable arc 
has a parameterization by arc length, defined on the interval $[0,s(\arc)]$, which is $1$-Lipschitz. By an affine change of variable, one obtains from a rectifiable arc $\arc$ a parameterization on $[0,1]$, which is $s(\arc)$-Lipschitz. Such parameterization is said \emph{proportional to the arc length}\index{parameterization proportional to the arc length}.

A parameterized arc $\param:[a,b]\to\C$ is \emph{regular} if it is  $C^2$ and  $\param'(t)\not= 0$ for all $t$. The \emph{tantrix}\index{tantrix} $T$ (tantrix is a contraction of tangent indicatrix) is  defined by 
$$T(t)=\frac{z'(t)}{|z'(t)|}~.$$

The \emph{curvature}\index{curvature (of an arc)} of a regular arc
is the function 
\begin{align}\label{eq:rotation coord cplxe}
 k(\param)  &= \left\langle \frac{T'}{|\param'|}, \nu\right\rangle =\frac{1}{|\param'|^3}\langle \param'',\nu \rangle \nonumber \\ 
 &= \frac{1}{|\param'|^3} \det(\param',\param'')=
 \frac{\param'_1\param''_2 - \param'_2\param''_1}{|\param'|^3}= \frac{\operatorname{Re}(\I\param'\bar{\param}'')}{|\param'|^3}~, 
\end{align}
where we denoted  by $\nu$\index{$\nu$} the oriented unit normal to $\param$, i.e.,
\begin{equation}\label{eq:normale}
\nu=\frac{1}{|\param'|}\param'^{\bot}
\end{equation}
where \index{$\param'^{\bot}$}
$$\param'^{\bot}=\binom{-\param_2'}{\param_1'}~,$$
and by $\langle\cdot,\cdot\rangle$\index{$\langle\cdot,\cdot\rangle$} the usual scalar product.

  If the arc is parameterized  by  arc length, i.e., if  $\langle \param',\param'\rangle=1$, then $\langle \param'',\param'\rangle=0$ and $z''$ and $\nu$ are colinear with the same norm, and we have
$k(\param)=\pm |\param''|$, depending on the direction of $\param''$.

A \emph{regular arc}\index{regular arc} is an arc whose parameterization by arc length is  a regular parameterization.
The \emph{rotation}\index{rotation (of a curve)} of a regular arc $\arc$ is 
\begin{equation}\label{eq:rotation}
\begin{split}
\ \kappa(K) & = \int_a^b k(\param)(t)|\param'(t)|\D t~, \\
\end{split}
\end{equation} 
as the definition above is independent of the choice of a regular parameterization.

In the case of a simple closed curve, it is easy to see that $ \kappa(K)=\pm 2\pi$, see e.g. \cite{docarmo}.

If $\arc_1$ and $\arc_2$ are two regular arcs, non-overlapping but with common extremities, then using the orientation, 
one can say when $\arc_2$ is \emph{on the left}  of $\arc_1$, see Figure~\ref{fig:left} for an heuristic illustration ($\arc_1$ is the plain line and $\arc_2$ the dotted one), 
and Section~\ref{left right turn sec} for more details.  Now, let us change the orientation of $\arc_2$ such that 
both arc have the same starting points and the same endpoints, and let 
$\alpha$ and $\beta$  be the inner angles of the open set enclosed by the curves at the extremities. 
We have    \cite{R63III}:
\begin{equation}\label{eq:def gen rot}\kappa(\arc_1)=\kappa(\arc_2)+\alpha+\beta~. \end{equation}

The formula above motivates the following definition.

\begin{definition}\label{def: tilde delta}
An arc is said to be of \emph{class $\tilde{\Delta}$}\index{$\tilde{\Delta}$ (class of arcs)} if it has a parameterization $z:[a,b]\to \C$ with
$z'_r(a)$ and $z'_l(b)$ that exist and are non-zero, where $z'_r$ and $z'_l$ are the right derivative and the left derivative respectively.
\end{definition}
In short, an arc is in $\tilde{\Delta}$ if it has a semi-tangent at its extremities. Not any rectifiable arc is in the class $\tilde{\Delta}$, as shows the graph of the function $f(x)=x\sin(-\ln(x))$, which has no right derivative at $0$, but which is Lipschitz.

 The inner angle at the extremities of two arcs of the class $\tilde{\Delta}$ bounding a domain can be defined. In turn, for any arc $\arc_1\in \tilde{\Delta}$, one can define its rotation  $\kappa(\arc_1)$  by the formula \eqref{eq:def gen rot}, by considering any suitable regular arc $\arc_2$ on the left of 
  $\arc_1$ and with the same extremities, and noting  that this definition actually does not depend on the choice of $\arc_2$. It is the way the rotation for an arc of the class $\tilde{\Delta}$ is defined in \cite{R63III}.

\subsubsection{Turn of arcs, Laplacian and curvature for Riemannian metric}\label{sec:Laplacian and curvature}

 We now define a  notion analogous to the rotation, when a subset of the plane is endowed with a conformal metric. We will use the standard notation $|\D z|^2$ for the Euclidean metric. Indeed, if 
we define the one-forms $\D z=\D x+\I\D y$ and $\D\bar{z}=\D x-\I \D y$, then $|\D z|^2$ is a notation for 
$\D z\D \bar{z}$, itself representing the symmetrization of the tensor product between 
$\D z$ and $\D \bar{z}$:
\index{$\vert\D z\vert^2$}
$$|\D z|^2=\frac{1}{2}\left( \D z\otimes \D \bar{z} + \D \bar{z}\otimes \D z\right) $$
and  an immediate computation shows that
$$|\D z|^2=\frac{1}{2}\left(\D x\otimes \D x + \D x\otimes \D y +\D y\otimes \D x+\D y\otimes \D y\right)~, $$ 
 and this last term is the usual Euclidean metric.
 
Let $M$ be a bounded \emph{domain}\index{domain} (i.e., a connected open set) in the plane, and let $g$ be a \emph{Riemannian metric}\index{Riemannian metric} on $M$, that is, for all $z\in M$, the data of a scalar product $g(z)$ on $T_zM\simeq \R^2$, such that $z\mapsto g(z)$ is $C^3$. The Riemannian metric $g$ is said to be \emph{conformal}\index{conformal} if there exists a $C^3$ function
$\lambda:M\to \R$ (necessarily positive) such that 
 $$g=g_\lambda:=\lambda|\D z|^2~. $$
We restrict ourselves to the $C^3$ regularity for Riemannian metrics,  in order to have a $C^1$ curvature, that would be a relevant assumption. See Remark~\ref{rem:C1C2curvature}.

\begin{remark}\label{pullback metric}
As an example, let $U\subset \C$ be a domain, and let $f:U\to \C$ be complex differentiable, with non-zero derivative. As $$|f'|^2|X|^2=|\D f(X)|^2$$
it readily follows that the pull-back over $U$ by 
$f$ of the Euclidean metric is 
$$|f'|^2 |\D z|^2~. $$ 
\end{remark}

Let $X,Y$ be smooth vector fields on $M$. If $D$ is the usual connection of the Euclidean space (the differentiation), and $\nabla$ is the Levi--Civita connection of $g_\lambda$, the Koszul formula gives
$$\nabla_XY= D_XY + \frac{1}{2\lambda} (\D\lambda(X) Y + \D\lambda(Y)X - \operatorname{grad}(\lambda)\langle X,Y\rangle)~, $$
 where $\operatorname{grad}$ is the usual gradient.
Let 
$(e_1,e_2)$ be the usual basis. Let $\param:[a,b]\to\C$ be a regular parameterized arc. The covariant derivative of a smooth vector field $X$ along $\param$ is given by
\begin{equation*}
\begin{split}
\ \frac{D}{\D t} X& = X' + X^1\nabla_{c'}e_1 + X^2\nabla_{c'} e_2 \\
\ &=X'+\frac{1}{2\lambda}\left( X^1\binom{D\lambda(c')}{-D\lambda (c'^{\bot})} +X^2\binom{D\lambda (c'^{\bot})}{D\lambda(c')}\right) \\
\end{split}
\end{equation*}
and the \emph{geodesic curvature}\index{geodesic curvature}\index{$k_\lambda(\param)$} of $\param$ is, if 
$\nu_\lambda=\frac{\nu}{\sqrt{\lambda}}$\index{$\nu_\lambda$} is the unit (for $g_\lambda$) normal to $\param$ and $\nu$ is defined in  \eqref{eq:normale}, 
\begin{equation*}\label{eq:geod curv}
\begin{split}
\ k_\lambda(\param) & = \frac{1}{\sqrt{g_\lambda(\param',\param')}}g_\lambda\left( \frac{D}{\D t} \param', \nu_\lambda \right)= \frac{1}{\lambda^{1/2}}k(\param) - \frac{1}{2\lambda^{3/2}}\frac{\partial\lambda}{\partial \nu}~,
\end{split}
\end{equation*}
 where $\frac{\partial }{\partial \nu}$\index{ $\frac{\partial }{\partial \nu}$} is the derivative 
in the direction of $\nu$, i.e., 
$$\frac{\partial f}{\partial \nu}  =\langle \grad f, \nu\rangle~.$$

Let $\arc$ be the arc parameterized by $\param$.
The \emph{turn} $\kappa_\lambda(\arc)$\index{turn}\index{ $\kappa_\lambda$} of  $\arc$ is
\begin{equation}\label{eq:def turn r}
\begin{split}
 \kappa_\lambda(\arc)& =\int_a^b k_\lambda(\param(t))\sqrt{g_\lambda(\param'(t),\param'(t))}\D t = \kappa(\arc) - \frac{1}{2}\int_\arc \frac{\partial \ln \lambda}{\partial \nu}~,
\end{split}
\end{equation}
where we used the definition of the rotation $\kappa(\arc)$, see \eqref{eq:rotation}. Note that the rotation of an arc is the turn for the Euclidean metric.

\begin{remark}[Length structure]\label{remark:def distance riemn}{\rm 
From the conformal Riemannian metric $\lambda|\D z|^2$ on $M$, for a broken line $K\subset M$, we may define its length

\begin{equation}\label{def length}\tilde{s}_\lambda(K)=\int_K \sqrt{\lambda}~, \end{equation}
and then define a \emph{Riemannian distance}\index{Riemannian distance} $\rho_\lambda$ between two points, as the infimum of the lengths of the broken line contained in $M$ joining them. 

For any non-negative function $\lambda$ (allowing the $+\infty$ value) that is integrable over any broken line, 
\eqref{def length} defines a 
 \emph{pseudo-distance} $\rho_\lambda$, i.e. a non-negative symmetric function on $M\times M$ satisfying the triangle inequality. In general, $\rho_\lambda$ is not a distance, in the sense that this function may be zero when evaluated at two distinct points. 
 
In the case of a conformal Riemannian metric, or, more generally, 
if $\lambda$ is finite, positive and continuous, then $\rho_\lambda$ is a distance. Indeed, for any $\arc$  from  a point $z_0$ to the boundary of a small Euclidean circle centered at this point, $\tilde{s}_\lambda(K)$ is greater than a positive constant times the  Euclidean length, where the positive constant is the minimum of $\lambda$ on the closed disc. 
 It follows that  for any $z$ belonging to the circle,  $\rho_\lambda(z,z_0)$ is positive, and then 
  any broken line to another point outside this disc will have larger length. Along the same line, it is easy to see that the topology induced by such a distance is the usual topology of the plane.
  \res found lower regularity assumptions on $\lambda$ in order for $\rho_\lambda$ to be a distance, see Section~\ref{sec:intrinsic length}.  } 
\end{remark}

\begin{remark}[Shortest arcs and turn]\label{rem:shortest path and turn}{\rm As by definition, a  regular  parameterized arc $\arc$ is a \emph{geodesic}\index{geodesic}
if  $\frac{D}{\D t}\param'=0$, it follows that a geodesic has zero turn.
Recall also that for a domain with a distance (see Remark~\ref{remark:def distance riemn}), a \emph{shortest arc}\index{shortest arc} between two points is an arc 
whose length is equal to the distance between the points, supposed this latter is finite. A shortest arc may not exist, or, if it exists, may not be unique. It is classical that for a Riemannian distance, the arc length parameterization of a shortest arc is a geodesic (note that to be a shortest arc is a property of an arc, while to be a geodesic is a property of a parameterized arc). It follows that for  a Riemannian distance, a shortest arc has zero turn. One of the main result of \cite{R63III} is a generalization of this fact to a larger class of metric spaces, see Theorem~\ref{thm:shortest boundedturn}. More precision about shortest arcs is given in Section~\ref{sec:intrinsic length}. 

We will say that an arc is a \emph{geodesic broken line}\index{geodesic broken line} is it can be decomposed as a union of shortest arcs.}\end{remark}

Now let  $(\tilde{e}_1,\tilde{e}_2)=(e_1/\sqrt{\lambda}, e_2/\sqrt{\lambda})$
be an orthogonal basis for $\lambda|\D z|^2$.  
The \emph{sectional curvature} of $\lambda|\D z|^2$ is given by 
 \begin{equation}\label{eq:def sec curv}
\begin{split}
\ \curv_\lambda & = \lambda \langle \nabla_{\tilde{e}_2}(\nabla_{\tilde{e}_1}\tilde{e}_1) -\nabla_{\tilde{e}_1}(\nabla_{\tilde{e}_2}\tilde{e}_1)+\nabla_{[\tilde{e}_1,\tilde{e}_2]}\tilde{e}_1,\tilde{e}_2\rangle~, \\
\  & =-\frac{1}{2\lambda} \Delta \ln \lambda = -\frac{1}{\lambda} \Delta \ln \sqrt{\lambda}~,
\end{split}
\end{equation} 
where $\Delta$ is the usual Laplacian: $\Delta  = \frac{\partial^2 }{\partial x^2}+\frac{\partial^2}{\partial y^2} $ \index{$\Delta$ (Euclidean Laplacian)}.

 \begin{definition}
 For a conformal Riemannian metric $\lambda|\D z|^2$, the \emph{area measure}\index{area measure} $A_\lambda$\index{$A_\lambda$} is $\lambda \mathcal{L}$, where $\mathcal{L}$\index{$\mathcal{L}$} is the Lebesgue measure, and the \emph{curvature measure}\index{curvature measure}
 $\omega_\lambda$ is $\curv_\lambda A_\lambda$, i.e.
 \begin{equation}\label{def:mes riem}\mes_\lambda=\curv_\lambda \lambda \mathcal{L}=-\frac{1}{2} \Delta \ln \lambda \mathcal{L}~. \end{equation}
 \end{definition}
 
 Note that if $h$ is a constant function, or more generally a \emph{harmonic function}\index{harmonic function}, i.e., a $C^\infty$ function such that $\Delta h=0$, 
 if $\bar{\lambda}=\E^h\lambda$, then $$ \omega_{\bar{\lambda}}=\omega_\lambda \mbox{ but } \curv_{\bar{\lambda}}=\E^{-h}\curv_\lambda~.$$

 The main point of the theory is that, up to a harmonic function, $\lambda$ is determined by the curvature measure.  
 This fact is based on the famous  \emph{Green formula}\index{Green formula}, for $C^2$ functions on the plane
\begin{equation}\label{eq:green}\iint_{M} f\Delta g - g\Delta f  = \int_{\partial M} f\frac{\partial g}{\partial\nu} - g\frac{\partial f}{\partial\nu} \end{equation}
if $\partial M$ is a  piecewise $C^1$ simple closed curve, and $\nu$ is the unit outward normal to $M$.
It is worth noting that for  $z\in M$, the function  $\zeta\mapsto \ln|z-\zeta|$ is harmonic out of $\{z\}$.
This leads to the following classical result, see e.g. \cite[Lemma 3.4.1]{jost}.

\begin{lemma}\label{lem:poisson} Let $f$ be a $C^2$ function over the plane. Then for $z\in M$,
\begin{equation}\label{eq:poisson1}f(z)=\frac{1}{2\pi}\iint_M  \ln |z-\cdot| \Delta f +\frac{1}{2\pi} \int_{\partial M}  f\frac{\partial \ln|z- \cdot|}{\partial \nu}-\ln|z-\cdot| \frac{\partial f }{\partial  \nu}~.\end{equation}
\end{lemma}

Let  $\mes$ be a signed measure  with compact support in $M$, with $C^\infty$ Lebesgue density $k$.  The function $p(\mes)$ over $M$ is defined at a point $z$ by

\begin{equation}\label{eq:poisson lisse}p(z;\mes)=\frac{1}{2\pi}\iint \ln |z-\cdot| k=\frac{1}{2\pi}\iint \ln |z-\zeta| \D\mes(\zeta) \end{equation}
(if nothing is specified, the integration is made over the plane). By \eqref{eq:poisson1}, on $M$ we have
\begin{equation}\label{eq:lapl curv}k=\Delta p(\mes)~. \end{equation}

The function $p(\mes)$
 is the \emph{potential}\index{potential} of $\mes$. 
 Comparing \eqref{eq:def sec curv} and \eqref{eq:lapl curv}, it follows that if we define
 $$\lambda=\E^{-2p(\mes)}~, $$
  then the conformal Riemannian metric $\lambda|\D z|^2$
has sectional curvature $\curv=k/\lambda$, i.e., the curvature measure of $\lambda|\D z|^2$ is $\mes$.

\begin{remark}\label{rem:C1C2curvature}{\rm
Actually, if $k$ is $C^1$, then $p(\mes)$, and in turn $\lambda$, is $C^2$ \cite[Corollary 4.5.4]{AG}. We hence obtain a $C^2$ Riemannian metric with a $C^1$ curvature (in general, if the metric is $C^2$, the curvature is $C^0$).
 Actually, the class of $C^2$ Riemannian metrics on an open domain with a $C^1$ curvature is a relevant class for the problem of finding \emph{isothermal coordinates}\index{isothermal coordinates} \cite{wintner} ---i.e., finding a $C^1$ homeomorphism such that the pull-back of the metric is conformal.
  Weaker assumptions for the regularity of the Riemannian metric in order to be isometric (here an isometry is a $C^1$ mapping) to a conformal metric are given in \cite{chern-hartman-wintner}, \cite{chern}, \cite{hartman-wintner}.
See \cite[Proposition 4.1]{troyanov-alexandrov} for a proof of the existence of isothermal coordinates in the $C^\infty$ case. 
 
 However, the Riemannian metric must be more regular than only continuous. For a modern treatment, see \cite[Section 3.5]{jost}.  We will be concerned by conformal metrics with a different kind of regularity assumption: the logarithm of the factor $\lambda$ will be a difference of subharmonic functions. See Remark~\ref{remark:regularity} for the kind of regularity this implies.
}\end{remark}

If two $C^2$ functions $p_1$ and $p_2$ satisfy 
$k=\Delta p_i$,  then 
$p_1 - p_2$ is a harmonic function.
It follows that all the conformal Riemannian metrics $\lambda|\D z|^2$ with curvature measure  $\mes$ are given by 
\begin{equation}\label{eq:metrique conforme}\lambda =\E^{-2\left(p(\mes) + h\right)}~, \end{equation}
where $h$ is a harmonic function, and $p(\mes)$ is the potential of $\mes$. 

The turn  \eqref{eq:def turn r} of an arc $\arc$, for a Riemannian conformal metric with conformal factor given by \eqref{eq:metrique conforme}, is given by\footnote{There is a slight difference in notation with the corresponding formula in \cite{R60I}, because \res is considering $\lambda =\E^{-2p(\mes) + h}$.}

\begin{equation}\label{eq:turn riem harm}
\kappa_\lambda(\arc)=\kappa(\arc)+\int_\arc \frac{\partial  p(\mes)}{\partial \nu} +\int_K \frac{\partial h}{\partial \nu}~.
\end{equation}

We want to rewrite the expression above, in order to be able to define the turn for more general objects than a Riemannian conformal metric.  For this we need another notion, the notion of angle under which an arc is seen.

\subsubsection{Angle under which an arc is seen, left and right turn of arcs}\label{left right turn sec}

Let $\arc$ be an arc, and  $\zeta$ be a point in the plane that is not on $\arc$. 
 The \emph{angular function} of $\arc$ \index{angular function} is a continuous function 
$ \varphi(\arc,\zeta,\cdot) :[a,b]\to \R$ defined in the following way. If $\param:[a,b]\to\C$ is a parameterization of $\arc$,
\begin{equation}\label{eq:angular function} \varphi(\arc,\zeta,t)= \operatorname{arg} (\param(t)-\zeta)  - \operatorname{arg}(\param(a)-\zeta)~,\end{equation}
where $\operatorname{arg}(z)\in (-\pi, \pi]$\index{arg} is the principal value of the argument of $z$.

There is a priori an ambiguity in the choice of the argument.  Of course,  $\varphi(\arc,\zeta,a)=0$. 
For $t$ close to $a$, $\operatorname{arg}(z(t))\in(-\pi,\pi]$ is the principal value of the argument. Then, the argument is uniquely determined by the condition that 
$t\mapsto \varphi(\arc,\zeta,t)$ is continuous.

The \emph{angle under which $\arc$ is seen from the point $\zeta$}\index{angle under which an arc is seen from a point} is the function $\varphi(\arc): \C\setminus K \to \R $ defined by
$$\varphi(\arc,\zeta):=\varphi(\arc,\zeta,b)$$  and it clearly does not depend on the parameterization of $\arc$.
For example, if $\arc$ is a circle, then $\varphi(\arc,\zeta)=0$ if $\zeta$ is outside the disc enclosed by the circle, and $\varphi(\arc,\zeta)=\pm 2\pi$ otherwise, depending on the orientation of the circle.

\begin{lemma}\label{lem angle cas c1}
Let $\param$ be a regular parameterization of a regular arc $\arc$. If $\zeta\notin \arc$,
$$\varphi(\arc,\zeta) =-\int_a^b \frac{\partial \ln | \cdot -\zeta|}{\partial \nu}(\param(t)) |\param'(t)| \D t
=-\int_\arc  \frac{\partial \ln | \cdot -\zeta|}{\partial \nu}$$ 
and in particular, if $0\notin \arc$,
\begin{equation}\label{eq:angle c1 zero complexe}
\varphi(\arc,0)=\int_a^b \frac{\operatorname{Re}(\I
\param(t)\bar{\param}'(t))}{|\param(t)|^2}\D t~.
\end{equation}
\end{lemma}
 \begin{proof}
We have $\varphi(\arc,\zeta)=\int_a^b\varphi(\arc,\zeta,t)'\D t$, and locally, 
\begin{equation}\label{eq:der angle}\varphi(\arc,\zeta,t)'=\left(\arctan \left(\frac{\param_2(t)-\zeta_2}{\param_1(t)-\zeta_1}\right)\right)'
=-\frac{\langle \param(t)-\zeta, \param'^\bot(t)\rangle}{|\param(t)-\zeta|^2}~. \end{equation}

 On the other hand,
\begin{equation}\label{eq:grad ln}\operatorname{grad} (\ln |\cdot - \zeta|)(z)=\frac{z-\zeta}{|z-\zeta|^2} \end{equation}
hence
$$ \frac{\partial \ln |\cdot - \zeta|}{\partial \nu}(\param(t))= \frac{\langle \param(t)-\zeta , \nu\rangle}{|\param(t)-\zeta|^2}= \frac{\langle \param(t)-\zeta , \param'^\bot(t)\rangle}{|\param(t)-\zeta|^2|\param'(t)|}~.$$ 

Evidently, one immediately computes
$$- \frac{\partial \ln |\cdot|}{\partial \nu}(z(t))|\param'(t)|=\frac{\param_1(t)\param_2'(t)-\param_2(t)\param_1'(t)}{|\param(t)|^2}=\frac{\operatorname{Re}(\I
\param(t)\bar{\param}(t)')}{|\param(t)|^2}~.$$
~\end{proof}

There is another way to interpret  Lemma~\ref{lem angle cas c1}. Let $h:M\to \R$ be a harmonic function on $M$. A $C^\infty$ function $h^*:M\to \R$ is said to be \emph{the (harmonic) conjugate}\index{conjugate (harmonic)} of $h$ if 
\begin{equation}\label{eq:CR}
  \left\{
      \begin{aligned}
        \frac{\partial h}{\partial x}-\frac{\partial h^*}{\partial y}=0\\
   \frac{\partial h}{\partial y}+\frac{\partial h^*}{\partial x}=0\\
      \end{aligned}
    \right.
\end{equation}
i.e., $f=h+\I h^*$ is holomorphic, as \eqref{eq:CR} are the Cauchy--Riemann equations. The conjugate exists on a simply connected domain if the form $- \frac{\partial h}{\partial y}\D x +\frac{\partial h}{\partial x}\D y$ is closed, i.e. $\Delta h=0$, that is the case as $h$ is harmonic. From \eqref{eq:CR}, it follows that $h^*$ is also harmonic.
 Then we may compute that, for a parameterization by arc length,
$$\frac{\partial h}{\partial \normal}=\left( z_1'\frac{\partial h}{\partial y}-z'_2\frac{\partial h}{\partial x}\right)\circ z=-\D h^*(z)(z')=-(h^*\circ z)'~,$$
hence
\begin{equation}\label{eq:int conjugate}
\int_K \frac{\partial h}{\partial \normal}=h^*(z_1)-h^*(z_2)~,
\end{equation}
if $z_2$ is the endpoint of $\arc$, and $z_1$ its starting point.

The relation with Lemma~\ref{lem angle cas c1}
follows from the fact that on $D=\C\setminus\{\zeta +y, y<0 \}$, 
the function  $z\mapsto \ln|\zeta-z|$ is the real part of $z\mapsto \operatorname{Ln}(z-\zeta)$,
with $\operatorname{Ln}$ the principal determination of the complex logarithm, i.e.,
$$\operatorname{Ln}(z-\zeta)=\ln|z-\zeta| + \I \operatorname{arg}(z-\zeta)~, $$
so that on $D$, the harmonic conjugate of 
$\ln|\cdot-\zeta|$ is $\operatorname{arg}(\cdot-\zeta)$, and then Lemma~\ref{lem angle cas c1} is essentially \eqref{eq:int conjugate}.

From \eqref{eq:turn riem harm}  and Lemma~\ref{lem angle cas c1}, we obtain the following formula,
which  is basically the content of the  first section of \cite{R60I}. 
 \begin{lemma}\label{lem:smooth turn} 
Let $\mes$ be a smooth Lebesgue density with compact support, and let $h$ be a harmonic function over $M$.
For $\lambda$ given by \eqref{eq:metrique conforme},  
the turn of a regular arc $\arc \subset M$  can be computed as follows:
 \begin{equation}\label{eq:smooth turn} \kappa_\lambda(\arc)=\kappa(\arc)-\frac{1}{2\pi}\iint \varphi(\arc,\zeta) \D\omega(\zeta) +\int_\arc 
 \frac{\partial h}{\partial \normal}~.\end{equation}
 \end{lemma}

\begin{remark}{\rm
In the case of a Riemannian conformal metric, we said in Remark~\ref{rem:shortest path and turn} that $\kappa_\lambda(\arc)=0$ if $\arc$ is a shortest arc. Hence in this case:
\begin{equation}\label{eq: rot shortest riem}
\kappa(\arc)=  \frac{1}{2\pi}\iint \varphi(\arc,\zeta) \D\omega(\zeta) -\int_\arc 
 \frac{\partial h}{\partial \normal}~.
\end{equation}}
\end{remark}

Before going on, let us mention here the following result, that is used in the proof of Theorem~9.1 in \cite{R60II}.
\begin{lemma}
We have, for a rectifiable arc $\arc$,  
if $\zeta$ is at distance $\geq \rho>0$ from $\arc$, $$\bigvee_a^b \varphi (\arc,\zeta)\leq \frac{s(\arc)}{\rho}~.$$
\end{lemma}

Recall that the \emph{total variation}\index{total variation} of a function $f:[a,b]\to \R$ is
\begin{equation}\label{eq:bigvee}\bigvee_a^b f = \sup \sum_i |f(t_i)-f(t_{i+1})|~, \end{equation}
where the suprememum is taken over all partitions of $[a,b]$.\index{$\bigvee$}

\begin{proof}
From the properties of the principal value of the argument, it clearly follows that if $\arc_n\to \arc$ and $\arc_n$ and $\arc$ all have the same starting point, and $\zeta\not\in K_n$,  then $\varphi (\arc_n,\zeta)\to \varphi (\arc,\zeta)$.
So  it suffices to consider the case of a broken line, the  conclusion will hold  by approximation. 
Actually, due to the concatenation property of $\vee$ and $s$, it suffices to check the inequality for the segment between each vertex of the broken line. But in such a case, the conclusion follows from \eqref{eq:der angle}, Cauchy--Schwarz inequality, and the fact that 
in the smooth case, the total variation is the integral of the derivative.
\end{proof}

Lemma~\ref{lem:smooth turn} allows a generalization of the concept of turn of an arc in two  directions:
\begin{enumerate}
\item  the right-hand side of  \eqref{eq:smooth turn} can be defined from a finite signed Borelian measure with compact support (that has not necessarily smooth Lebesgue density) and a harmonic function;  
\item the right-hand side of  \eqref{eq:smooth turn} can be defined for a non-regular arc, as we saw that the rotation can be  defined for any arc of the class $\tilde{\Delta}$, and for the  term $\frac{\partial h}{\partial \normal}$, it suffices to consider the harmonic conjugate and use  \eqref{eq:int conjugate}.
\end{enumerate}

The point is that the angular function is not defined from a point on the arc, that may cause some problems 
to define the term $\iint \varphi(\arc,\zeta) \D\omega(\zeta)$ for a more general measure.
 This is why 
the notion of right and left turn enters the picture. 

First, if $a$ is an extremity  of $K$, 
if $K_n$ is a proper subarc of $K$ converging to $K$, then we define
$\varphi(K,a)=\lim_n\varphi(K_n,a)$.
To extend the definition of $\varphi(K,\cdot)$ to the interior $K^\circ$\index{$K^\circ$} of $K$, we need the notion of left and right side of a curve.

\begin{figure}
\begin{center}
\includegraphics{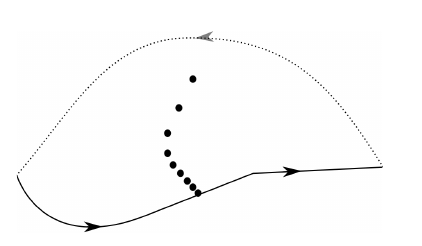}
\caption{The sequence of points is converging on the left to a point in the interior of the plain curve.}\label{fig:left}
\end{center}
\end{figure}

Let us consider a simple closed curve. By Jordan Theorem, it has an interior and an exterior. 
The angle under which the arc is seen from an exterior point is equal to zero. We say that the arc is \emph{positively oriented}\index{positively oriented (curve)} if the angle under which the arc is seen from an interior point is equal to $2\pi$. The only other possibility being that the angle is equal to $-2\pi$, in which case we say that the arc is \emph{negatively oriented}.

Now consider an arc $\arc$ whose extremities do not coincide. By joining it to another suitably oriented arc $\arc'$ in order to obtain a simple close curve $\arc\arc'$, as the arcs are oriented, we are able to say when the domain enclosed by $\arc\arc'$ is \emph{on the right} or \emph{on the left} \index{on the left} of $\arc$.
Using this, 
for a sequence of points in $\C\setminus \arc$ converging to a point
in $\arc^\circ$, we can say if the convergence is \emph{on the right} or \emph{on the left}\index{convergence on the left} of $\arc$.
See 
 Figure~\ref{fig:left} for an heuristic argument, and paragraph 2 in \cite{R63III} for more details.

It is then possible to extend the definition of the angle $\varphi(\arc,\cdot)$ to $K$ by continuity
on the left or on the right. This  leads to the definition of \emph{left angle under which the curve is seen}\index{left angle} $\varphi_l(\arc)$, as well as  \emph{right angle} $\varphi_r(\arc)$,  which  satisfy the following equalities:
\begin{equation}\label{eq:left right angles}\begin{cases}   \varphi_l(\arc)-\varphi_r(\arc)= 0 & \mbox{over }  \C\setminus \arc^\circ\\
   \varphi_l(K)-\varphi_r(\arc)= 2\pi &\mbox{over }   \arc^\circ
\end{cases}~.
\end{equation}
 Existence of the left and right angles and the equalities above are proved in Lemma~2 in \cite{R63III}. For example, if $\arc$ is a  directly oriented circle and $\zeta\in \arc$, then 
 $\varphi_l(\arc,\zeta)=2\pi $, and  $\varphi_r(\arc,\zeta)=0$.

We can now define the \emph{left turn}\index{left turn}
of an arc $\arc$ of the class $\tilde{\Delta}$, given  any signed Borelian measure $\mes$ with compact support, and any harmonic function $h$ defined on a domain containing $\arc$:
 \begin{equation}\label{eq:def turn}\kappa_l(K):=\kappa(K)-\frac{1}{2\pi}\iint_M \varphi_r(K,z)\D\mes(z)+(h^*(z_1)-h^*(z_2))~,\end{equation}
 if $z_1$ and $z_2$ are the extremities of $K$,
 and the \emph{right turn}\index{right turn}  \begin{equation}\label{eq:def turnb}\kappa_r(K):=-\kappa(K)+\frac{1}{2\pi}\iint_M\varphi_l(K,z)\D\mes(z)-(h^*(z_1)-h^*(z_2))~,\end{equation}
 and we have by \eqref{eq:left right angles}
 \begin{equation}\label{eq:left+right}
 \kappa_l(K)+\kappa_r(K)=\mes(K^\circ)~.
 \end{equation}

  Let $K$ be a regular closed curve, positively oritented, enclosing a domain $D$ homeomorphic to a disc. In this case,  we obtain a \emph{Gauss--Bonnet Formula}\index{Gauss--Bonnet Formula} almost by definition:
\begin{equation}\label{eq:GB}\kappa_l(K)+\mes(D)=2\pi~. \end{equation}

See \cite{R63III}. A more general version of \eqref{eq:GB} is given in Theorem~2 in \cite{R63III}.
 
In the next paragraph, we will look at those concepts for the simplest measures which are not smooth Lebesgue densities, that are weighted Dirac measures. The motivation for looking at the turn will be made clear in Section~\ref{sec:tot rot and bound rot}.

 \subsubsection{Flat cones and weak Laplacian}\label{sec:weak laplacian}\label{Sec: flat cones}
 
As a basic and fundamental example, let us look at the turn for arcs in a flat cone metric.
Flat cones are the simplest and the main  example of subharmonic metrics which are not Riemannian. Flat cones are an elementary example of \emph{simple singularities} on Riemannian surfaces, see \cite{tro-ouvert}.
Also,  one of the main tools in Reshetnyak's arguments, the canonical stretching of a distance (see Section~\ref{sec can stretch}), allows to reduce some problems to this  case.

 For $\zeta\in \C$ and $\mes_0\in \R$, let $$\mes=\mes_0\delta_{\zeta}$$ be the signed measure concentrated at $\zeta$ with value $\mes_0\in \R$.
Then  define
$$\lambda(z)=|z-\zeta|^{-\frac{\mes_0}{\pi}}=|z-\zeta|^{2\beta}=
|z-\zeta|^{\frac{\alpha}{\pi}-2}~,$$
where $$\beta=-\frac{\mes_0}{2\pi}$$ is the \emph{weight}\index{weight (cone)}, and $$\alpha=2\pi(\beta+1)=2\pi-\omega_0$$ is the
\emph{cone angle}\index{cone angle}.
To simplify the exposition, let us consider $\zeta=0$.  We will look at the plane endowed with the metric $\lambda|\D z|^2$. 
As $z\mapsto \ln |z|$ is harmonic on $\C\setminus\{0\}$, it follows
that $\lambda|\D z|^2$ is actually a Riemannian flat metric on  $\C\setminus\{0\}$. Let us see that it defines however a distance on the whole plane.

We define the length of a broken line $\arc$ in the same way as for a Riemannian metric (see Remark~\ref{remark:def distance riemn}). If $\param$ is an arc length  parameterization (for the Euclidean length), 
\begin{equation}\label{eq:length curve cone}\tilde{s}_{\lambda}(\arc)=\int_0^{s(\arc)} |\param(t)|^\beta \D t~,\end{equation}
where $s(\arc)$ is the Euclidean length of $\arc$.
For example, let $v$ be a unit vector and let consider the parameterized arc
 $\param:[0,r]\to \C$, $\param(t)=tv$. Hence  
\begin{equation}\label{eq:dist cone zero}\tilde{s}_\lambda(\param)=\int_0^r t^\beta \D t=
\left\{
\begin{array}{c c l c l c l c}
+\infty &\mbox{ if }& \beta\leq -1  &\mbox{ i.e. } & \mes_0 \geq 2\pi  &\mbox{ i.e. } & \alpha \leq 0\\
\frac{1}{\beta+1}r^{\beta+1} &\mbox{ if }&  \beta > -1 & \mbox{ i.e. }&  \mes_0< 2\pi & \mbox{ i.e. } & \alpha > 0
 \end{array}
\right.~.\end{equation}

Let us denote by 
$\rho_\lambda(a,b)$ the infimum of  $\tilde{s}_\lambda$ for all the broken lines between $a$ and $b$. We will denote by $C_r(\zeta)$\index{$C_r(\zeta)$} the circle of center $\zeta$ and (Euclidean) radius $r$.

\begin{lemma}
$\rho_\lambda$ is a distance over the plane.
\end{lemma}
\begin{proof}
From Remark~\ref{remark:def distance riemn}, 
$\rho_\lambda$ is a pseudo-distance over the plane. Actually the argument in the  remark also gives that it is a distance on $\C\setminus\{0\}$. \eqref{eq:dist cone zero}
says that the distance from $0$ to $C_r(0)$ is positive (or infinite) for any $r>0$. 
\end{proof}

The choice of the name cone angle for $\alpha$ may be explained as follows.  Let us define
\begin{equation}\label{eqValpha}V_\alpha:=\left\{ (r,\theta) \vert r\geq 0,\, 
\theta\in \R / \alpha \mathbb{Z} \right\} / (0,\theta)\thicksim (0,\theta')~. \end{equation} 
Then the pull-back of the metric $|z|^{2\beta}|\D z|^2$ over $\mathbb{C}\setminus \{0\}$ onto $V_\alpha \setminus \{(0,0)\}$ via the mapping
$$\left(r,\theta\right) \mapsto
 \left(\left(\beta+1\right)r\right)^{\frac{1}{\beta+1}} \left( \cos\left(\frac{\theta}{\beta+1}\right), \sin\left(\frac{\theta}{\beta+1}\right)\right)$$
is
$$\D r^2 + r^2\D\theta^2~. $$

It is easy to see that \eqref{eq:dist cone zero} actually gives the $\rho_\lambda$ distance from  the origin to $C_r(0)$, see \cite[Lemma 12.1]{R60II}. 
 If follows that 
$C_r(0)$ is at $\rho_\lambda$ distance one  from the origin if $\mes_0< 2\pi$
and $r=\left(1+\beta\right)^{\frac{1}{1+\beta}}$.
But 
\begin{equation}\label{eq:long cercle cone}\tilde{s}_\lambda(C_r(0))=\int_{C_r(0)} \sqrt{\lambda}= r\int_0^{2\pi} 
r^{\beta}\D t= 2\pi r^{1+\beta}=2\pi (1+\beta)=\alpha~,\end{equation}
hence the length for $\rho_\lambda$  of the  unit circle for $\rho_\lambda$ is $\alpha$. 
From this it is easy to recognize the distance over some usual objects, see Table~\ref{table cone}.

We are cheating a little in the sense that \eqref{eq:long cercle cone} is a priori not the intrinsic length of $C_r(0)$ for the distance $\rho_\lambda$ (see Section~\ref{sec:intrinsic length}). Indeed, the intrinsic length must be defined as  the infimum of the length for $\rho_\lambda$ of all the broken lines inscribed in the curve. But we will see in Section~\ref{sec:intrinsic length} that both definitions agree in our case.

\begin{table}
\begin{center}

\begin{tabular}{|c|c|c|}
  \hline
cone angle & curvature & weight  \\
  $\alpha$ & $\mes_0$ & $\beta=-\frac{\mes_0}{2\pi}=\frac{\alpha}{2\pi}-1$  \\
  \hline
  $0$ & $2\pi$ & $-1$   \\
  $]0,2\pi[$ & $]0,2\pi[$ & $]-1,0[$ \\
  $2\pi$ & $0$ & $0$  \\
  $>2\pi$ & $<0$ & $>0$\\
  \hline
\end{tabular}\caption{It is easy to see that the distance given by the first line is isometric to the induced distance  over a Euclidean cylinder of radius $2\pi$. For the second line, it is a convex cone in the Euclidean $3$-space. Of course, the third line corresponds to the Euclidean plane. It is also possible to see that the distance given by the  last line corresponds to the induced distance on a convex  cone in the Lorentzian Minkowski $3$-space, i.e., $\R^3$ endowed with the bilinear form $\D x^2+\D y^2-\D z^2$, see Example~\ref{example:cone4pi}.}\label{table cone}
\end{center}

\end{table}

It follows that the point $0$ is at finite distance from the other points if $\mes(\{0\})<2\pi$, and at infinite distance if  $\mes(\{0\})\geq 2\pi$. This will be a general phenomenon for the subharmonic distances, except for the case $=2\pi$, that may behave in  different manners, depending on $\mes$. This is the source of some technical difficulties.

\begin{remark}\label{rem lon cercle}{\rm Note that  $$\tilde{s}_\lambda(C_1(0))=2\pi$$
for \emph{any} value of $\alpha>0$.}
\end{remark}

Let us mention the following basic result, which is used for example in the proof of Lemma 16 in \cite{R63III}.

\begin{lemma}\label{lem isom cone} For $\beta>-1$ i.e., $\omega_0<2\pi$,
the sector of the plane between two rays making an angle $\theta$ for the distance  $\rho_\lambda$ is isometric to a Euclidean sector of angle $\theta\left(1-\frac{\mes_0}{2\pi} \right)$.
\end{lemma}
\begin{proof}
Let us consider the function
$$f:z\to (1+\beta)^{-1}z^{\beta+1}~,$$
defined for a suitable branch cut of the plane.
For an arc length parameterization $\param$ of a rectifiable arc $K$, we have
$$\tilde{s}_\lambda(f(\arc))=\int_0^{s(\arc)} |\param(t)|^{-\frac{\omega_0}{2\pi}}\D t~. $$
Comparing with \eqref{eq:length curve cone}, it follows that $f$  is an isometry from  
 the plane with the metric $ \lambda|\D z|^2$  to the Euclidean plane. 
 Without loss of generality we can consider that the sector is between vectors $1$ and $\E^{\I\theta}$, and 
 $f(\E^{\I\theta})=(1+\beta)^{-1}\E^{\I\theta(\beta+1)}$, so that the new sector has an angle $\theta(\beta+1)$.
\end{proof}

As we already noted, on $\mathbb{C}\setminus\{0\}$, a flat cone metric is a Riemannian metric. Considering the function $f$ from the proof of Lemma~\ref{lem isom cone}, together with Remark~\ref{pullback metric},
$f$ is a Riemannian isometry 
between the disc without the origin endowed with the 
metric $|z|^\beta|\D z|^2$ and the standard Euclidean metric.
So  if $\arc$ is a regular arc not passing through the origin, we have the following formula for the geodesic curvature of $\arc$, see Lemma~\ref{lem:smooth turn},
\begin{equation}\label{eq:turn flat metric zero}
\kappa_\lambda(\arc)=\kappa(\arc)+\beta \varphi(\arc,0)~.  
\end{equation}
An argument similar to the proof of Lemma~\ref{lem isom cone} gives  the following result, which is used in the beginning of the proof of Lemma 13 in \cite{R63III}.

\begin{lemma}
Let $\arc$ be a regular arc not passing through the origin. Then 
$$\kappa(\arc^{\beta+1})= \kappa_\lambda(\arc)~, $$
with $\lambda=|\cdot|^{2\beta}$, and $K^{\beta+1}$ is the image of the arc under $f(z)=z^{\beta+1}$.
\end{lemma}
\begin{proof}
It suffices to compute $\kappa(\arc^{\beta+1})$ with 
\eqref{eq:rotation coord cplxe}, and then use \eqref{eq:angle c1 zero complexe} and \eqref{eq:turn flat metric zero}.
\end{proof}

\begin{example}\label{example:cone4pi}{\rm
 It is easily checked that the Euclidean length of the curve $(r,\theta=\frac{1}{r}-\pi)$, $0<r<\frac{1}{\pi}$ in polar coordinates, and $0$ at $r=0$, is infinite. Its length in the flat cone metric defined by $\lambda(z)=|z|^2$ is finite, see the last section of \cite{R63}. The distance given by $\lambda$ is isometric to the induced distance on the cone in $\R^3$, endowed with the metric $\D x^2+\D y^2-\D z^2$, generated by the rotation preserving horizontal planes of the half-line of direction $(2,0,\sqrt{3})$.
}\end{example}

One of the main result of \cite{R63III} (see the end of Section~\ref{sec:tot rot and bound rot})
 applied to our flat cone situation, implies the well-known fact that a shortest arc for a flat cone metric with $0<\mes_0<2\pi$, i.e., with a cone angle  in $]0,2\pi[$, cannot contain the origin in its interior.

\begin{remark}\label{rem:regular sens de resh}{\rm
So far, we have considered signed measures with smooth Lebesgue density or weighted Dirac measure. It would then be natural to consider signed measures which are a sum of measures of those two kinds. Such measures are the main objects of the last part of  \cite{R63III}. Indeed, they are dense (for the weak convergence of measures) in a set of measures $A$. The Localization Theorem, see Section~\ref{sec:localization}, will say that any of the signed measures considered in Reshetnyak's theory
can be transformed in a suitable manner to 
a measure of the set $A$. 
 }\end{remark}

In Section~\ref{sec:Laplacian and curvature}, we saw the relation between the measure and the Laplacian in the Riemannian case. A similar analysis is possible here. 
The \emph{weak Laplacian} of a locally $L^1$ function $f:M\to \R$ over a domain $M$ of the plane  is the linear functional on $C_c^\infty(M)$, the space of $C^\infty$ functions with compact support included in $M$, defined by
$$\Delta f(\varphi)=\iint_M f \Delta\varphi~.$$ 

Here we are making an abuse of notation, as if $f\in C^2$, then the notation $\Delta f$ may stand for a continuous function or a linear functional on $C^\infty_c$. But if one recalls  the Green formula
\eqref{eq:green}, we have the following relation for $f\in C^2$:
 $$\Delta f(\varphi)=\iint_M \varphi \Delta f~,$$ 
 i.e., if $f$ is $C^2$, the function $\Delta f$  is equal as a distribution to the distribution $\Delta f$.

 Let us go back to the simplest non-smooth example. 

\begin{lemma}\label{lem:potential dirac} For $\zeta\in \C$, let $p(\delta_\zeta)$ be the function defined by 
$p(z;\delta_\zeta)=\frac{1}{2\pi}\ln|z-\zeta|$.  Then for $\varphi\in C_c^\infty$,
$$\Delta p(\delta_\zeta)(\varphi) =\varphi(\zeta)~.$$
\end{lemma}
\begin{proof}
It is a straightforward consequence of Lemma~\ref{lem:poisson}, by taking a domain containing the support of $\varphi$.
\end{proof}

Hence, as a distribution, the weak Laplacian of $p(\delta_\zeta)$ is the Dirac measure $\delta_\zeta$. But it is classical that if the weak Laplacian is a \emph{positive functional}, i.e., if for a $L^1$ function $f$ we have
$\Delta f(\varphi)\geq 0$ for any non-negative $\varphi\in C_c^\infty$, then it extends to a positive linear functional over $C_c^0$  \cite{ransford}, \cite{hayman}, \cite{rao}. In turn, by the Riesz Representation Theorem, to a positive measure. At the end of the day, we may identify $\Delta p(\delta_\zeta)$ and $\delta_\zeta$ as positive measures.

In turn, we obtain an analogue of \eqref{def:mes riem}, in term of measures, as if we consider the flat cone metric
$\lambda|\D z|^2$
with
$$\lambda(z)=|z-\zeta|^{-\frac{\mes_0}{\pi}}~, $$
then, in the sense of measures, by Lemma~\ref{lem:potential dirac}
\begin{equation}\label{cone measure weak lap flat}\mes_0\delta_\zeta=-\frac{1}{2}\Delta \ln \lambda~.\end{equation}

In the other sections, we will bring what we do for Riemannian metrics and flat metrics to its full generality. Before that, let us mention two natural generalizations of flat metrics: spherical  cone metrics
and hyperbolic cone metrics.

\begin{example}[Spherical cone-metrics]\label{rem:spherical cone metric}{\rm

Let us consider the conformal Riemannian metric 
\begin{equation}\label{eq:def spher hyp}\lambda(z)|\D z|^2= \frac{4}{\left(1+|z|^2\right)^2}|\D z|^2~, \end{equation}
for $z$ belonging to  the plane. 
A straightforward computation (see \eqref{eq:def sec curv}) shows that the curvature is constant equal to $1$. 
The plane endowed with the metric \eqref{eq:def spher hyp} is a conformal representation of a
\emph{spherical metric}\index{spherical metric}.  On may check that 
 \eqref{eq:def spher hyp} is actually given by the stereographic projection (from the south pole) of the induced metric on the unit sphere in the Euclidean $3$-space.

Over $\C$ minus the origin, 
let us consider the function  $f(z)=\frac{1}{\beta+1}z^{\beta+1}$, for a suitable branch cut of the plane and $\beta> -1$. The spherical metric is pull-backed by $f$ (see Remark~\ref{pullback metric}) to a conformal metric with conformal factor

$$\tilde{\lambda}(z)=\lambda(f(z))|f'(z)|^2=\frac{4|z|^{2\beta}}{\left(1+(\beta+1)^{-2}|z|^{2(\beta+1)}\right)^2}~. $$

The pull-back of the metric $\tilde{\lambda}|\D z|^2$ onto the cone $V_\alpha$ given by \eqref{eqValpha}, where $\alpha=2\pi(\beta+1)$, via the function
$$(r,\theta)\mapsto \left(\left((\beta+1)\tan(r/2)\right)^{\frac{1}{\beta+1}}
\left(\cos\left(\frac{\theta}{\beta+1}\right),\sin\left(\frac{\theta}{\beta+1}\right)\right)   \right)$$
is written
$$\D r^2+\sin(r)^2 \D\theta^2~. $$

Let us compute
$$-\frac{1}{2}\Delta \ln \tilde{\lambda}(z)=
-\frac{1}{2}\Delta \ln |z|^{2\beta} + \Delta \ln \left(1+(\beta+1)^{-2}|z|^{2(\beta+1)}\right)~.
 $$

Considering the equation above in the sense of distribution, 
by \eqref{cone measure weak lap flat}, we know that the term 
$-\frac{1}{2}\Delta \ln |z|^{2\beta}$ is the measure
$-2\pi\beta\delta_0$. Concerning the second term $\Delta \ln \left(1+(\beta+1)^{-2}|z|^{2(\beta+1))}\right)$, the function 
$$z\mapsto\ln \left(1+(\beta+1)^{-2}|\cdot|^{2(\beta+1)}\right)$$ has no singularity on $U$, and a straightforward computation shows that 
$$\Delta \ln \left(1+(\beta+1)^{-2}|z|^{2(\beta+1)}\right)=\tilde{\lambda}(z)~,$$
hence in the sense of distribution, $-\frac{1}{2}\Delta \ln \tilde{\lambda}(z)$ is equal to the measure

$$\omega_0\delta_0 + \tilde{\lambda}\mathcal{L} $$
where $\mathcal{L}$ is the Lebesgue measure over $U$ and
$\omega_0=2\pi-\alpha=-2\pi\beta$.
}\end{example}

\begin{example}[Hyperbolic cone-metrics]{\rm

Let us consider the conformal Riemannian metric 
\begin{equation}\label{eq:def met hyp}\lambda(z)|\D z|^2= \frac{4}{\left(1-|z|^2\right)^2}|\D z|^2~, \end{equation}
for $z$ belonging to a subdomain of the open unit disc. 
A straightforward computation (see \eqref{eq:def sec curv}) shows that the curvature is constant equal to $-1$, so \eqref{eq:def met hyp} is a conformal representation of a
\emph{hyperbolic metric}\index{hyperbolic metric}. 
Suppose that $U$ contains the origin. 
Over $U$ minus the origin, 
let us consider the function  $f(z)=\frac{1}{\beta+1}z^{\beta+1}$, for a suitable branch cut of the plane and $\beta> -1$. The hyperbolic metric is pull-backed by $f$ (see Remark~\ref{pullback metric}) to a conformal metric with conformal factor

$$\tilde{\lambda}(z)=\lambda(f(z))|f'(z)|^2=\frac{4|z|^{2\beta}}{\left(1-(\beta+1)^{-2}|z|^{2(\beta+1)}\right)^2}~. $$

The pull-back of the metric $\tilde{\lambda}|\D z|^2$ onto the cone $V_\alpha$ given by \eqref{eqValpha}, where $\alpha=2\pi(\beta+1)$, via the function
$$(r,\theta)\mapsto \left(\left((\beta+1)\tanh(r/2)\right)^{\frac{1}{\beta+1}}
\left(\cos\left(\frac{\theta}{\beta+1}\right),\sin\left(\frac{\theta}{\beta+1}\right)\right)   \right)$$
is written
$$\D r^2+\sinh(r)^2 \D\theta^2~. $$

Let us compute

$$-\frac{1}{2}\Delta \ln \tilde{\lambda}(z)=
-\frac{1}{2}\Delta \ln |z|^{2\beta} + \Delta \ln \left(1-(\beta+1)^{-2}|z|^{2\beta+1}\right)~.
 $$

Computations similar to the ones of the spherical case (Remark~\ref{rem:spherical cone metric}) show that, in the sense of distribution, $-\frac{1}{2}\Delta \ln \tilde{\lambda}(z)$ is equal to the measure

$$\omega_0\delta_0 - \tilde{\lambda}\mathcal{L} $$
where $\mathcal{L}$ is the Lebesgue measure over $U$ and
$\omega_0=2\pi-\alpha=-2\pi\beta$.

}\end{example}

\subsection{Arcs of bounded rotation}\label{sec abs rot}

The arcs of (absolute) bounded rotation are the basic objects of the whole theory. They form a proper subclass of the class of rectifiable arcs of the plane, and this subclass contains  smooth arcs and broken lines. 

The main property that distinguish them from the class of rectifiable arcs is the convergence of the length for a converging sequence of arcs having a uniform bound of the absolute rotation (Lemma~\ref{lem: BR longueur cv}). This is   an essential step in the proof of the Distances Convergence Theorem~\ref{thm: distances convergence theorem}.

Roughly speaking, if rectifiable arcs have a Lipschitz parameterization, arcs of bounded rotation have a \dc parameterization, see Section~\ref{sec arc cvexes}.

In all this section, we are considering the Euclidean plane. References for this section are \cite{AR}, \cite{milnor}, \cite{sullivan}, \cite{res2005}. Basically, we review the properties stated in $\S$3 of \cite{R60II}.

\subsubsection{Classical results}

Let $\arc$ be a broken line, parameterized by affine functions.  Its \emph{absolute rotation}\index{absolute rotation} $|\kappa|(\arc)$ is the non-negative real number that is the sum of the angles (in $[0,\pi]$) between its left and right derivative at its interior vertices.

\begin{definition}
The \emph{absolute rotation} $|\kappa|(\arc)\in\R\cup \{+\infty\}$ of an arc $\arc$ is the supremum of the absolute rotation of the broken lines inscribed in $\arc$. The arc is said to be of \emph{bounded rotation}\index{bounded rotation} if $|\kappa|(\arc)$ is finite. 
\end{definition}

It is clear that if we modify a broken line $\arc$ by adding a new vertex, we obtain a new broken line $\arc'$ such that $|\kappa|(\arc')
\geq |\kappa|(\arc)$, see Figure~\ref{fig:kappa increase}. From this, it is easy to see that for a broken line, the two definitions of absolute rotation we have given above (as a sum of angles on the one hand, and as a supremum of absolute rotations of inscribed broken lines on the other hand) coincide.

\begin{figure}
\begin{center}
\includegraphics{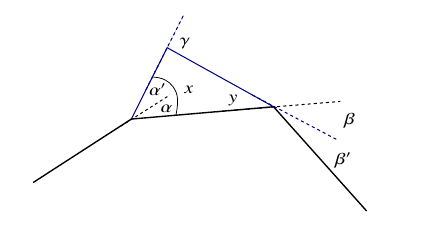}
\caption{$\alpha'+\gamma+\beta' =\alpha'+x+y+\beta'\geq \alpha+\beta $.}\label{fig:kappa increase}\normalsize
\end{center}
\end{figure}

\begin{lemma}\label{lem:liminf kappa}
Let $(\arc_n)_n$ be a sequence of arcs converging to $\arc$. Then  $ |\kappa|(\arc) \leq \liminf |\kappa|(\arc_n)$.
\end{lemma}
\begin{proof}
Let $L$ be a broken line inscribed in $\arc$, and let $L_n$ be broken lines inscribed in $\arc_n$ such that the associated sequences of vertices converge to the vertices of $L$. Clearly, $|\kappa|(L_n)\to |\kappa|(L)$ and 
$|\kappa|(L_n)\leq |\kappa|(\arc_n)$, so $|\kappa|(L) \leq \liminf |\kappa|(\arc_n) $, and as $L$ is arbitrary,  $  |\kappa|(\arc)=\sup |\kappa|(L)  \leq \liminf |\kappa|(\arc_n)$.
\end{proof}

\begin{lemma}\label{lem:conv kappa inscr}
Let $(\arc_n)_n$ be a sequence of broken lines inscribed in $\arc$ and converging to $\arc$. Then $s(\arc_n)\to s(\arc)$ and $|\kappa|(\arc_n)\to |\kappa|(\arc)$.
\end{lemma}

We skip the proof, as the one for the length is classical, and 
the one  for the absolute rotation is
analogous. This result gives for example that, 
as the absolute rotation of a closed convex broken line is $2\pi$, the absolute rotation of a circle is $2\pi$.

If $A$ is an interior vertex of a broken line $\arc$, and $\arc_1, \arc_2$ the two broken lines obtained from $\arc$ by splitting at $A$, then if $\alpha$ is the angle at $A$, by definition
$$|\kappa|(K)=|\kappa|(K_1)+|\kappa|(K_2)+\alpha~. $$

It follows that, for any decomposition  $a=t_0<t_1 <\cdots < t_N=b$ of a parametrization of an arc $\arc$, \begin{equation}\label{eq:sum rot subarc}\sum_i |\kappa|(K\vert_{ t_i,t_i+1})\leq |\kappa|(\arc)~.\end{equation}
A sharper version of \eqref{eq:sum rot subarc}  will be given  in Corollary~\ref{cor:BR split}.
From \eqref{eq:sum rot subarc} we deduce that  $t\mapsto |\kappa|(K\vert_t)$ is non-decreasing. This is used in the proof of
the following result, which is Theorem~5.8.1  in \cite{AR}.

\begin{proposition}[Alexandrov Inequality]\label{prop: alexandrov inequality}
Let $\arc$ be an arc with $|\kappa|(\arc)< \pi$ and extremities $z_1$ and $z_2$. Then
$$\cos \left(|\kappa|(\arc) / 2 \right) s(\arc)\leq \vert z_1-z_2\vert ~.$$

\end{proposition}
\begin{proof}
From Lemma~\ref{lem:conv kappa inscr}, it suffices to prove the formula for broken lines.
Let $\param:[0,s(\arc)]\to \C$ be the arc length parameterization of a  broken line. In particular, its left derivative has unit norm. Up to a rotation, suppose that the first segment of the broken line is directed by $(1,0)$. Let $e$ be a unit vector with principal argument  $|\kappa|(K)/2$. So for any $t\in [0,s(\arc)]$, 
the angle between $\param'_l(t)$ and $e$ is less than  $|\kappa|(K)/2$, i.e., $\langle \param'_l(t),e \rangle \geq \cos \left(|\kappa|(\arc)/2 \right)$, and
$$\vert z_1 - z_2\vert \geq \langle z_1 - z_2 , e \rangle = \int_0^{s(\arc)} \langle \param_l', e \rangle \geq s(\arc) \cos \left(|\kappa|(\arc)/2 \right)~. $$
~\end{proof}

That's not true that for any $t_0\in [a,b]$ there exists a sufficiently small $\delta $ such that $|\kappa|(\arc_{t_0-\delta,t_0+\delta})$ is arbitrarily small, as the example of an inner vertex of a broken line shows. But this property is true on a right neighborhood $[t_0, t_0+\delta]$ for $t_0\in [a,b)$ (of course, a similar property holds 
 on a left neighborhood $[t_0-\delta, t_0]$ for $t_0\in (a,b]$). Indeed, we have for $t_n\to t_0\in [a,b)$, $t_0<t_n<b$, from \eqref{eq:sum rot subarc},
$$|\kappa|(\arc\vert_{t_0,b}) - |\kappa|(\arc\vert_{t_n,b})\geq |\kappa|(\arc\vert_{ t_0,t_n }) $$
and from $|\kappa|(\arc\vert_{t_0,b}) \geq |\kappa|(\arc\vert_{t_n,b})$ and Lemma~\ref{lem:liminf kappa},
it follows that  $|\kappa|(\arc\vert_{t_n,b}) \to |\kappa|(\arc\vert_{t_0,b})$, and hence 
\begin{equation}\label{eq:kappa point=0}
|\kappa|(\arc\vert_{ t_0,t_n}) \xrightarrow[n\to \infty]{}  0~. 
\end{equation}

\begin{remark}\label{rem:equi decompositio BR}{\rm Let  $\epsilon >0$, 
and let us define $t_i\in [a,b)$ by 
$t_0=a$ and 
$|\kappa|(\arc\vert_{ t_{k-1},t_k})= \epsilon$, that is possible by \eqref{eq:kappa point=0}.  
If $ |\kappa|(\arc) \leq A<\infty $, the number $N$ of such $k$ is finite. 
Actually, we have
$$A\geq |\kappa|(\arc) \geq \sum_{k=1}^N |\kappa|(\arc\vert_{ t_{k-1},t_k})\geq 
\sum_{k=1}^{N-1} |\kappa|(\arc\vert_{ t_{k-1},t_k})=(N-1)\epsilon~, $$
so that $N\leq \frac{A}{\epsilon}+1$.

Hence, for any arc with  $|\kappa|(\arc) \leq A$, for any $\epsilon >0$  there exists $N\leq \frac{A}{\epsilon}+1$ and a decomposition  $a=t_0<t_1 <\cdots < t_N=b$  such that for any $i$,  $|\kappa|(\arc\vert_{ t_i,t_{i+1}})\leq \epsilon$.  The decomposition depends on the arc, but $N$ depends only on $A$ and $\epsilon$. 
}\end{remark}

In particular, any  arc of bounded rotation can be decomposed into a finite number of subarcs with absolute rotation less than $\pi$. Applying  Proposition~\ref{prop: alexandrov inequality} to these subarcs, the concatenation property of the length gives the following result.

\begin{corollary}\label{cor:BR->rec}
An arc of bounded rotation is rectifiable.
\end{corollary}

\begin{lemma}\label{lem: BR longueur cv}
Let $A\geq 0$ and let $(\arc_n)_{n\geq 1}$, $\arc_0$ be arcs such that $\arc_n\to\arc_0$ and for any $n\geq 1$, $|\kappa|(\arc_n)\leq A$. Then  $s(\arc_n)\to s(\arc)$ when $n\to \infty$.
\end{lemma}
\begin{proof}
As we already know that $\liminf s(\arc_n)\geq s(\arc)$, it suffices to show that $\limsup s(\arc_n)\leq s(\arc)$. Let us consider parameterizations $\param_n$ proportional to  arc length and   a subsequence  $(\arc_{n_i})_{n_i}$ such that $s(\arc_{n_i})$ converges to $\limsup s(\arc_{n})$.  Let $\epsilon >0$. By Remark~\ref{rem:equi decompositio BR}, we know that there exists   $N\leq \frac{A}{\epsilon}+1$ and a decomposition   $0=t_0^{n_i}<t_1^{n_i} <\cdots < t_N^{n_i}=1$ such that  $|\kappa|(\arc_{n_i}\vert_{ t^{n_i}_k,t^{n_i}_{k+1}})\leq \epsilon$. For a given $k$, we  extract from $(t^{n_i}_k)_{n_i}$ a  subsequence  converging to some number $t_k\in [0,1]$. Abusing notation, we will denote with the same indices $n_i$ this subsequence. 
Let $L_{n_i}$ (resp. $L$) be the broken line with vertices  at $t^{n_i}_k$ (resp. $t_k$), so it is inscribed in $\arc_{n_i}$ (resp. $\arc$). By construction, $$s(L_{n_i}) \xrightarrow[n_i\to \infty]{} s(L)\leq s(\arc)~.$$ From Proposition~\ref{prop: alexandrov inequality},
$$s(\arc_{n_i}\vert_{t^{n_i}_k,t^{n_i}_{k+1}}) \leq \frac{1}{\cos\epsilon} \vert \param_{n_i}(t^{n_i}_{k+1})- \param_{n_i}(t^{n_i}_{k}) \vert$$
so
$$s(\arc_{n_i}) \leq \frac{1}{\cos\epsilon} s(L_{n_i})$$
and passing to the limit,
$$\limsup s(\arc_{n}) =  \lim s(\arc_{n_i}) \leq  \frac{1}{\cos\epsilon} \lim  s(L_{n_i}) \leq   \frac{1}{\cos\epsilon} s(\arc)  $$
and the result follows because $\epsilon$ was arbitrary.
\end{proof}

For the proof of Lemma~\ref{lem: BR longueur cv}, we have copied the argument of Lemma~6.2 in \cite{R60II}. Note that the statement in this reference is a bit more general, but only the version of Lemma~\ref{lem: BR longueur cv} will be useful in the sequel.  Lemma~\ref{lem: BR longueur cv} also follows from the following more general result (see Theorem~5.6.2 in \cite{AR}): there exists a positive constant $C$ such that, for two arcs $\arc_1, \arc_2$ of bounded rotation,
$$\vert s(\arc_1)-s(\arc_2)\vert \leq C \rho(\arc_1,\arc_2)(|\kappa|(\arc_1)+|\kappa|(\arc_2)+\pi)$$
where $\rho$ is the distance between $\arc_1$ and $\arc_2$.

\begin{lemma}\label{lem:uniform cv subarcs}
Let $A\geq 0$ and let $(\arc_n)_{n\geq 1}$, $\arc_0$ be arcs such that $\arc_n\to\arc_0$ and for any $n$, $|\kappa|(\arc_n)\leq A$. Then there exists $N$ such that  for all $n$, $\arc_n$ can be decomposed into $N$ subarcs of absolute rotation less than $\pi$, and any sequence of subarcs converge to the corresponding subarc of $\arc_0$.
\end{lemma}
\begin{proof}
The equidecomposition of the $\arc_n$ is given by Remark~\ref{rem:equi decompositio BR}. It remains to check that the sequence of subarcs converge to a subarc of $\arc$. But this follows from the concatenation property of the length and Lemma~\ref{lem: BR longueur cv}.
\end{proof}

\begin{proposition}\label{prop: rot diam long}
For any arc $\arc$ of bounded rotation, 
\begin{equation}\label{aq: rot diam long}s(\arc) \leq \frac{\operatorname{diam}(K)}{2} \left( |\kappa|(\arc) +\pi\right)~.\end{equation}
\end{proposition}

For the proof of Proposition~\ref{prop: rot diam long}, we refer to Theorem 5.6.1 in \cite{AR}.
It uses projection techniques, similar to the ones in paragraph 8 in \cite{R63III}.
There is one case for which \eqref{aq: rot diam long} is straightforward: when the absolute rotation is sufficiently small. Indeed,  consider the function over $[0,\pi)$ defined by $f(x)=\cos(x/2)(x+\pi)/2$. As $f'(0)>0$ and $f(0)>1$, there exists 
$T\in (0,\pi)$ such that for $x\in (0,T)$, 
\begin{equation*}\label{eq: cos 1}1\leq  \cos(x/2)(x+\pi)/2~.\end{equation*}
 Then \eqref{aq: rot diam long}  follows from Proposition~\ref{prop: alexandrov inequality}

\begin{corollary}
For any sequence of arcs contained in a bounded domain, with absolute rotation uniformly bounded from above, one can extract a converging subsequence with converging length.
\end{corollary}

\begin{remark}[The condition $\Gamma(h)$]\label{rem Gammah}
{\rm
Let us present a related class of curves, that are introduced in \cite{R60II}.
Let $0<h<1$.
A rectifiable arc $\arc$ satisfies the condition $\Gamma(h)$\index{$\Gamma(h)$} if, for its arc length parameterization $\param$, for any $t_1,t_2\in[a,b]$,
$$|z(t_1)-z(t_2)|\geq h|t_1-t_2|~. $$
If the arc is rectifiable, we already know that $z:[0,s(K)]\to \arc\subset \C$ is a $1$-Lipschitz homeomorphism. The condition  $\Gamma(h)$ means that $z^{-1}$ is $h$-Lipschitz. In turn, the arc length parameterization is bi-Lipschitz. 
By definition, 
$$\operatorname{diam}(\arc)\geq |z(t_1)-z(t_2)|\geq h|t_1-t_2| $$
for any $t_1,t_2$.

Let us consider an arc of bounded rotation, decomposed  it into a finite number of subarcs of absolute rotation $<\pi$. For such a subarc $\arc_i$, Proposition~\ref{prop: alexandrov inequality} says that it satisfies the property $\Gamma(h)$ with
$$h=\cos\left(\frac{\kappa(\arc_i)}{2}\right)~.$$
}\end{remark}

\subsubsection{Arcs of bounded rotation and convex functions}\label{sec arc cvexes}

Let us first prove a regularity property of arcs of bounded rotation.

\begin{lemma}\label{lem: BR derivee}
Let $\param:[a,b]\to\C$ be the arc length parameterization of an arc of bounded rotation. Then $\param$ has a right derivative $\param'_r$  on $[a,b)$, which is a unit vector. 
\end{lemma}
Of course  a similar result holds for the left derivative on $(a,b]$. In particular, an arc of bounded rotation belongs to the class $\tilde{\Delta}$ introduced in Definition~\ref{def: tilde delta}.
\begin{proof}
Let $t_0\in [a,b)$, and $t_0<t_1<t_2<t_0+\delta$ for  a sufficiently small $\delta >0$. 
Let us consider the triangle $t_0t_1t_2$, $\alpha$ being the angle at $t_0$ and $\beta$ being the angle at  $t_1$ of this triangle. Up to change $t_1,t_2$, we can suppose that $\alpha+\beta<\pi$, otherwise the right of the arc at $\param(t_0)$ is a segment and we are done. Let $k=\pi-\beta$; $k$ is the absolute rotation of the broken line with vertices $t_0$, $t_1$ and $t_2$. By definition of the absolute rotation,
$k\leq |\kappa|(\arc\vert_{t_0,t_0+\delta})$ and hence $\alpha<  |\kappa|(\arc\vert_{t_0,t_0+\delta})$, and  by \eqref{eq:kappa point=0} we know that |$\kappa|(\arc\vert_{t_0,t_0+\delta})$ can be made arbitrary small. In turn, if $t_n\to t_0, t_n>t_0$, 
the sequence of unit vectors $\frac{\param(t_n)-\param(t_0)}{|\param(t_n)-\param(t_0)|}$ is a Cauchy sequence on the unit circle, hence converge. 

On  the other hand, as $\param$ is an arc length parameterization, from Proposition~\ref{prop: alexandrov inequality} we have
$$1\geq \frac{\vert \param(t_n)-\param(t_0)\vert}{t_n-t_0} \geq \cos \left(|\kappa|(\arc\vert_{t_0,t_n})/2 \right)~,$$
so if $t_n\to t_0$, by \eqref{eq:kappa point=0},  $\frac{\vert \param(t_n)-\param(t_0)\vert}{t_n-t_0}\to 1$.

At the end of the day, for $t_n>t_0$, $\frac{ \param(t_n)-\param(t_0)}{t_n-t_0}=\frac{  \param(t_n)-\param(t_0)}{\vert \param(t_n)-\param(t_0)|}\frac{\vert \param(t_n)-\param(t_0)\vert}{t_n-t_0}$ converges when $t_n\to t_0$, and it is also immediate that the limit has unit norm.
\end{proof}

For a rectifiable  arc $\arc$ such that its arc length parameterization $\param$ admits a right derivative, we denote by $T_r(\arc):[a,b)\to \mathbb{S}^1$ its \emph{right tantrix}\index{tantrix}, 
that is  the function $t\mapsto \param_r'(t)$.

\begin{lemma}\label{lem: BR BV cercle}
Let $\arc$ be 
a rectifiable arc such that $T_r(K)$ is defined. Then 
  for $0<\eta<s(K)$,
$$\bigvee_0^{s(K)-\eta} T_r(\arc)\leq |\kappa|(\arc)~. $$
\end{lemma}
Here, the total variation of $T_r(\arc)$ is taken for the intrinsic distance $d_S$ over the circle: the definition is formally the same as in \eqref{eq:bigvee}, with $d_S$ instead of $|\cdot|$.\index{$\bigvee$}

\begin{proof}
The statement is obvious if the absolute rotation is not finite, so we can suppose that $\arc$ has bounded rotation.
It also suffices to prove the lemma for the subarc of $\arc$ parameterized on $[0,s(K)-\eta]$, for an arc length parameterization of $\arc$. 
Note that $T_r$ is defined everywhere on the subarc. Abusing notation, we will denote this subarc by $\arc$ in the following.

Let $\epsilon>0$. Let $\param$ be the arc length parameterization of $\arc$ 
 and $0=t_0<t_1<\cdots<t_n=s(K)$ be such that 
$\bigvee_0^{s(K)} T_r(\arc) - \sum d_S(\param_r'(t_i),\param_r'(t_{i+1}))\leq \epsilon$.
Let $L$ be a broken line inscribed in $\arc$ such that $|\kappa|(\arc)-|\kappa|(L)\leq \epsilon$.
  This last inequality remains true if we add vertices to $L$, so we can moreover suppose that the images of the $t_i$ by the arc length parameterization $\tilde{\param}$ of $L$ are vertices of $L$, and  
also that 
the angle between $\tilde{\param}_r'(t_i)$ and $\param_r'(t_i)$ is less than $\frac{\epsilon}{2n}$. This is possible if the vertex of $L$ next to the image of $t_i$ is sufficiently close to it, because $\param$ has a right derivative at $t_i$.
Also, the inequality we want to prove is clearly true for broken lines.
Summing up:

\begin{align*}\bigvee_0^{s(K)} T_r(\arc)& \leq  \epsilon + \sum d_S(\param_r'(t_i),\param_r'(t_{i+1}))\\
& \leq  \epsilon+ \sum d_S(\param_r'(t_i),\tilde{\param}_r'(t_{i}))+d_S(\tilde{\param}_r'(t_i),\tilde{\param}_r'(t_{i+1}))+
d_S(\tilde{\param}_r'(t_{i+1}),\param_r'(t_{i+1})) \\
&\leq  2\epsilon + |\kappa|(L) \\ &\leq 3\epsilon + |\kappa|(\arc) \end{align*}
and as $\epsilon$ was arbitrary we have $\bigvee_0^{s(K)} T_r(\arc)\leq |\kappa|(\arc)$.
\end{proof}

A function $f:[a,b]\to \R$ is \emph{\dc}\index{\dc}  if 
there exist continuous convex functions $g_{1}, g_2:[a,b]\to \R$, such that $f=g_1-g_2$. A parameterized arc is called \dc (or DC \index{DC function}) if its components are \dc functions. See Remark~\ref{remark DC} about \dc functions in metric geometry.

\begin{proposition}\label{prop dc curve}
The arc length parameterization of an arc  of bounded rotation is $\delta$-convex.
\end{proposition}
\begin{proof}
By Lemma~\ref{lem: BR BV cercle}, the arc length parameterization $\param$ of an arc of bounded rotation has a right derivative of bounded variation in the unit circle on $[a,b-\eta]$ for any small $\eta>0$. For $x=(x_1,x_2),y=(y_1,y_2)$ unit vectors in the plane, if $d_S$ is the intrinsic distance on the circle, there is a positive constant $c$ such that $c(|x_1-y_1|+|x_2-y_2|)\leq |x-y| \leq d_S(x,y)$.  Hence  $(\param_1)_r'$ and $(\param_2)_r'$ are real-valued functions of bounded variation, and hence, for example, 
 $(\param_1)_r'=f_1-f_2$, where $f_1$ and $f_2$ are  two bounded non-decreasing functions on $[a,b-\eta]$. 
It is known that the function $F_i$, $i=1,2$, defined by
$$F_i(t)=\int_a^t f_i~ $$
 is convex on $(a,b)$, and by properties of the integral, $F_i$ is continuous and hence convex over $[a,b-\eta]$. Moreover,
$$(F_1-F_2)'_r(t)=\lim_{h\to 0, h>0}(f_1-f_2)(t+h)=\lim_{h\to 0, h>0}(\param_1)_r'(t+h)~.$$

As $f_i$ are non-decreasing, they are  continuous except on an at most countable set of points $E_i$, hence $(z_1)'_r$ is continuous on $(a,b)\setminus (E_1\cup E_2)$, hence 
on this set $(F_1-F_2)'_r=(z_1)'_r$. 
As $F_i$ and $\param_1$ are Lipschitz, 
they have  derivative  almost everywhere on $(a,b)$ by Rademacher Theorem, so almost everywhere $(F_1-F_2)'=z_1'$,
 and
$$\param_1(t)-\param_1(a)=\int_a^t \param_1' =\int_a^t (F_1-F_2)'  $$
so, up to add a constant to $F_1$, $\param_1=F_1-F_2$, as all those functions are continuous. Hence $\param_1$ is \dc on $[a,b-\eta]$ for any $\eta$.

Now there are two ways to conclude. We can do the same argument with the left derivative instead of the right derivative to conclude that $\param_1$ is also \dc on $[a+\eta,b]$, and use the fact that a locally \dc function is \dc \cite{hartman}.
Otherwise, as $\param$ has a left derivative at $b$, it is not hard to see that we can add a little segment to the endpoint of $\arc$, colinear to $\param_l'(b)$, and that the new curve has also bounded rotation. As our starting curve is a proper subarc, the partial result we obtained proves the statement of the proposition.
\end{proof}

A \dc function inherits many properties of convex functions. In particular, it follows that for any $t_0\in (a,b]$,
\begin{equation}\label{eqRV}\lim_{h\to 0, h>0} T_r(\arc)(t-h)=T_l(\arc)(t)~, \end{equation}
see e.g. \cite{RV}. We have a similar relation by exchanging the role of $T_r$ and $T_l$.

Hence, for an arc $\arc$ of bounded rotation,  we may extend $T_r(K)$, that is defined only on $[a,b)$, to a function $\tilde{T}_r(\arc) $ defined on $[a,b]$ by setting 
$\tilde{T}_r(\arc)=T_r(\arc) $ on $[a,b)$ and 
$\tilde{T}_r(\arc)(b)=T_l(\arc)(b)$.

\begin{proposition}\label{prop:tatrix kapap}
Let $\arc$ be a rectifiable arc such that 
$\tilde{T}_r(\arc)$ is defined. Then 
$$\bigvee_0^{s(\arc)} \tilde{T}_r(\arc)= |\kappa|(\arc)~. $$
\end{proposition}
\begin{proof}
We need to prove the converse inequality of Lemma~\ref{lem: BR BV cercle}.
The proof is similar, we leave it to the reader.
\end{proof}

Note that it is almost immediate that the absolute rotation of the graph of a continuous convex function is equal to the angle between the directions of the tangents at its extremities.  Deeper, $\delta$-convex functions are essentially characterized by the fact that their right derivative has bounded variation \cite{RV}, and this implies the following result.

\begin{corollary}
An arc with a \dc parameterization has bounded rotation.
\end{corollary}

\begin{corollary}
\label{cor:BR split}
Let $\arc$ be an arc of bounded rotation with arc length parameterization $\param$. For  $t\in (0,s(K))$, if $\arc$ is splited at $\param(t)$ into arcs $\arc_1$ and $\arc_2$,
then
$$|\kappa|(\arc)=|\kappa|(\arc_1)+|\kappa|(\arc_2)+\angle(t)~, $$ 
where $\angle(t)$ is the angle between $T_r(\arc)$ and $T_l(\arc)$ at $t$.
\end{corollary}
\begin{proof}
From Proposition~\ref{prop:tatrix kapap},
$$
 |\kappa|(\arc) = \bigvee_0^{s(K)} \tilde{T}_r(\arc) =\bigvee_0^{t} \tilde{T}_r(\arc)  + \bigvee_t^{s(K)} \tilde{T}_r(\arc) = \bigvee_0^{t} \tilde{T}_r(\arc)  + \bigvee_t^{s(K)} \tilde{T}_r(\arc_2)~.$$

Using the triangle inequality, if $d_S$ is the intrinsic distance of the circle, it is easy to see that 
 
 $$\bigvee_0^{t} \tilde{T}_r(\arc)
= \bigvee_0^{t} \tilde{T}_r(\arc_1)  + 
d_S(T_r(\arc)(t),\lim_{h\to 0, h>0}T_r(\arc)(t-h))~.$$

The result follows from \eqref{eqRV} and that of course 
$$\angle(t)=d_S(T_r(\arc)(t),T_l(\arc)(t))~.$$
~\end{proof}

\subsubsection{Bounded rotation and bounded turn}
\label{sec:tot rot and bound rot}
In this part, we relate the absolute rotation to the rotation and the angular function, both introduced in Section~\ref{sec regular arcs}.
We can first note that for a regular
 arc,
$$\vert \kappa\vert (\arc) =\int_\arc |k|(\arc)~.$$
This is however more general. We have defined the rotation for the arcs of the class $\tilde{\Delta}$. Now let us introduce the \emph{rotation function}\index{rotation function} of the arc $\arc$:  $\kappa(t)=\kappa(K|_t)$. Of course this is not defined for any arc, as we need that for any $t$, $K|_t$ belongs to $\tilde{\Delta}$. But that is defined for an arc of the class $\Delta$:\index{$\Delta$ (class of arcs)} an arc in the class $\tilde{\Delta}$ having a parameterization with non-zero left and right derivative at each interior point.
By definition, for a parameterization on $[a,b]$, $\kappa(K)=\kappa(b)$.

For this class of arcs we can now give the following relation: an arc of bounded rotation is an arc of the class $\Delta$ with bounded total variation of the rotation function.

\begin{lemma}\label{lemma:rotation and argument}
Let $\param:[a,b]\to \C$ be  a suitable parameterization of an arc of bounded rotation. Then
\begin{equation}\label{eq:kappa tilde delta}\kappa(K)=\operatorname{arg}(z_l'(b))-\operatorname{arg}(z_r'(a))~. \end{equation}
\end{lemma}
\begin{proof}
We have from \eqref{eqRV}
$$\operatorname{arg}(z_l'(b))-\operatorname{arg}(z_l'(a))=\int_a^b \operatorname{arg}(z'(t))'\D t~, $$
and 
$$\operatorname{arg}(z')'=\arctan\left(\frac{z_2'}{z_1'}\right)' =\frac{z_1'z_2''-z_1''z'_2}{|z'|^2}=k(z)|z'|~.$$~
\end{proof}

\begin{lemma}\label{lemma:variation rotation}
For an arc $\arc$ of bounded rotation, if $\kappa$ is its rotation function on $[a,b]$, then
$$|\kappa|(K)=\bigvee_a^b  \kappa~.$$
\end{lemma}
\begin{proof}
By Lemma~\ref{lemma:rotation and argument},
$\bigvee_a^b  \kappa$ is the total variation of the principal argument of the left derivative of a parameterization of $\arc$. This is exactly the total variation of the left tantrix of the arc. The result is given by (the analog for the left tantrix of) Proposition~\ref{prop:tatrix kapap}.
\end{proof}

Let $\arc$ be an arc of the class $\Delta$, so that the rotation function $\kappa$ is defined. Given  a  finite signed Borelian measure $\mes$ with compact support and a harmonic function $h$, let us define  the functions
$ \kappa_l$ and $\kappa_r$ on a segment $[a,b]$, for a suitable parameterization of the arc:
$$\kappa_l(t)=\kappa(t) -\frac{1}{2\pi}\iint \varphi_l(\arc,z,t)\D\mes(z) +(h^*(z(a))-h^*(z(t)))~, $$
$$\kappa_r(t)=\kappa(t) +\frac{1}{2\pi}\iint \varphi_r(\arc,z,t)\D\mes(z)-(h^*(z(a))-h^*(z(t)))~. $$
(Compare   \eqref{eq:def turn} and \eqref{eq:def turnb}.)
 The arc 
is said of \emph{bounded turn}\index{bounded turn} if the functions $ \kappa_l$ and $\kappa_r$ are of bounded total variation. As the measure $\mes$ if finite, its total variation is finite (see Section~\ref{sec: delta sub} for a remainder about the total variation of a measure), and it follows from \eqref{eq:left+right} that it suffices that only one of the functions  $ \kappa_l$ or $\kappa_r$ has bounded total variation for the arc to be of bounded turn.

Now, if $\arc$ is of bounded rotation, Lemma~\ref{lemma:variation rotation} says that the function $\kappa$ is of bounded total variation. The function $h^*\circ z$ is also of bounded total variation, as  $\arc$ is rectifiable, so its parameterizations are of bounded total variation, and $h^*$ is smooth.

\begin{lemma}
We have 
$$\bigvee_a^b \varphi_l (K,z) \leq |\kappa|(K)+\pi~.$$
\end{lemma}
The above result was proved by Radon in \cite{radon}. Another proof can be found in Lemma~8.3 in \cite{R60II}.
Actually, the statement in this reference is only for $\varphi (K,z) $, but it is easily extended to  $\varphi_l (K,z)$ by lower semicontinuity of the total variation.
Putting things together gives  the following result.

\begin{theo}
\begin{theorem}[{\cite[Theorem 4]{R63III} }]\label{thm:br implique bt}
Let $\arc$ be an arc of bounded rotation  in the plane. 
Then, for any finite signed Borelian
measure $\mes$  with compact support and any harmonic function $h$, $\arc$ is an arc of bounded turn.
\end{theorem}
\end{theo}

The main difference of nature between arcs of bounded turn and arcs of bounded rotation is that the latter class depend only on the Euclidean structure of the plane. To define arcs of bounded turn, one needs the differentiable structure of the plane (to define the $\Delta$ class), a signed measure $\mes$ and a harmonic function $h$.
Of course, for $\mes=0$ and $h=0$, bounded turn and bounded rotation are the same thing. Apart from this trivial case, the converse result to Theorem~\ref{thm:br implique bt} is much more harder to prove. It will be stated in Section~\ref{sec:turn rotation}.

\section{Subharmonic functions}\label{sec:sh}
We saw at the very end of Section~\ref{sec regular arcs}, that if $\delta_\zeta$ is the Dirac measure with support $\{\zeta\}$,  
if $p(\delta_\zeta)$ is defined on the plane by
$$p(z;\delta_\zeta)=\frac{1}{2\pi}\ln |z-\zeta| $$
then, for the weak Laplacian $\Delta$, in the sense of measures, 
$$\Delta p(\delta_\zeta)=\delta_\zeta~. $$
A similar result is  actually true for any positive measure with compact support. We will only review some basic facts about potential theory in sections \ref{sec: 3.1 def} and \ref{sec:approx}. For more details, we refer to the numerous  textbooks existing on this topic, see e.g. \cite{rado}, \cite{tsuji}, \cite{ransford}, \cite{hayman}, \cite{AG}, \cite{landkof}, \cite{demailly}, 
\cite{hor}. 

In Section~\ref{sec: hausdorff sub}, we will focus on the fact that the set of points where a subharmonic function has infinite values is a set of Hausdorff dimension zero. This is a fundamental fact, as it will imply  that the integral
of  the restriction of a difference of subharmonic functions on a rectifiable arc is well defined. 
In Section~\ref{sec:localization} we present in full details the \res Localization Theorem, which allows to
change locally a harmonic term to a difference of subharmonic functions into a constant.

\subsection{Definition}\label{sec: 3.1 def}

Let $\mes$ be a finite positive (Borelian) measure with compact support over the plane. 
It will be useful to recall that such measures are actually \emph{regular measures}\index{regular measures}:\footnote{This is the actual meaning of regular measure, and is different from the regular measures studied in \cite{R63III}, see Remark~\ref{rem:regular sens de resh}.
}
 for every Borel set $E$, 
$$\mes(E)=\inf \{ \mes(V) : E \subset V, \, V\mbox{ open}\}$$
and 
$$\mes(E)=\sup \{ \mes(K) : K\subset E,\, K \mbox{ compact}\}~.$$

For $z$ in the plane, let us denote
\begin{equation}\label{eq:potential}
p(z;\mes)=\frac{1}{2\pi} \iint \ln |z-\zeta  | \D\mes(\zeta)~.
\end{equation}

 We denote by $Q_r(\zeta)$\index{$Q_r(\zeta)$} the open disc of center $\zeta$ and (Euclidean) radius $r$, and set  $Q_r=Q_r(0)$.

\begin{lemma}\label{lem:potential bfb}
Let $r>1/2$ be such that the support of $\mes$ is contained in $Q_r$. 
For any $z\in  Q_r$, $p(z;\mes)$ is well-defined and $$p(z;\mes)\leq \frac{\mes(\C)}{2\pi}\ln(2r)~. $$
\end{lemma}
\begin{proof}
We have
$$p(z;\mes)=\frac{1}{2\pi}\iint_{|\zeta|\leq r}\ln|z-\zeta| \D\mes(\zeta) $$
so if $z\in Q_r$, i.e. $|z|\leq r$, we have 
$|z-\zeta|\leq 2r$. 
In turn:
\begin{equation}\label{eq:potential bfb}\ln |z-\zeta| =\left(\ln 2r\right)- \left( \ln 2r-\ln |z-\zeta| \right)~, \end{equation}
and both functions in the brackets are non-negative,  and the Lebesgue integral of $\ln 2r$ is finite because $\mes$ is finite. 
\end{proof}

It follows that for any $z$, 
$p(z;\mes)$ is well-defined in $[-\infty,+\infty)$, so we obtain a function  
$p(\mes):\R^2\to [-\infty,+\infty)$, the \emph{potential}\index{potential} of $\mes$. 

\begin{lemma}\label{lem:eq:supermean eq}
Let $\mes$ be a finite positive measure with compact support. Then for any $r>0$
\begin{equation}\label{eq:supermean eq}
p(z_0;\mes)\leq \frac{1}{2\pi r}\int_{C_r(z_0)} p(\mes)~.
\end{equation}
 \end{lemma}
 Recall that $C_r(z_0)$ is the plane circle of center $z_0$ and radius $r$.
\begin{proof}
 First note that for any $\zeta$,  the function $z\mapsto \ln\vert z-\zeta\vert$ satisfies \eqref{eq:supermean eq}. Indeed, either it is harmonic (so \eqref{eq:supermean eq} is satisfied with equality) or it is equal to $-\infty$.  Then by Fubini
\begin{equation*}
\begin{split}
 \frac{1}{2\pi}\int_0^{2\pi} p(z+r\E^{\I t};\mes)\D t = &
\frac{1}{2\pi}\iint \frac{1}{2\pi}\int_0^{2\pi}  \ln\vert z+r\E^{\I t}- \zeta\vert \D t \D\mes(\zeta) \\
 \geq & \frac{1}{2\pi}\iint  \ln\vert z- \zeta\vert  \D\mes(\zeta) = p(z;\mes)~.
\end{split}
\end{equation*}
~\end{proof}

\begin{lemma}\label{lem: pmu lower}
The potential $p(\mes)$ is upper semicontinuous,  in particular, it is bounded from above on any compact set.
\end{lemma} 
\begin{proof} 
By \eqref{eq:potential bfb},  for any $z\in M$,  $\zeta\mapsto \ln(2r)-\ln|z-\zeta|$ is non-negative, and  Fatou's lemma applies: for $z_n\to z$, 
\begin{equation*}
\begin{split}
-p_\mes(z)& =\frac{1}{2\pi}\iint \liminf_{n\to \infty} \left(\ln(2r)- \ln |z_n-\zeta  | \right) \D\mes(\zeta) \const \mes(\C)\ln(2r) \\
& \leq  \frac{1}{2\pi}\liminf_{n\to \infty} \iint \left(\ln(2r)- \ln |z_n-\zeta  | \right)\D\mes(\zeta)  \const \mes(\C)\ln(2r) \\
\  &=  \liminf_{n\to \infty} -p_\mes(z_n)~.\end{split}
\end{equation*}
The assertion about boundedness is a classical property of upper semicontinuous functions, see e.g., Section~2.1 in \cite{ransford}.
\end{proof}

\begin{lemma}\label{lem:harmonic out support}
The potential $p(\mes)$ is harmonic over $\C\setminus \operatorname{supp}\mes$.
\end{lemma}
\begin{proof}
A continuous function that satisfies equality in \eqref{eq:supermean eq} for sufficiently small circles is harmonic, see e.g., \cite[Theorem 1.2.7]{ransford}. To prove the reverse inequality is along the same line 
as  for the proof of Lemma~\ref{lem:eq:supermean eq}, with the difference that 
for $\C\setminus \operatorname{supp}\mes$, $ \vert z-\zeta\vert >0$ for $\zeta \in  \operatorname{supp}\mes$.
\end{proof}

\begin{example}\label{ex:potentiel cercle}{\rm
The simplest example of potential is for a weighted mass
$\mes=\mes_0\delta_{z_0}$, for which we have already seen that $p(z;\mes)=\frac{\mes_0}{2\pi}\ln \vert z-z_0\vert$.
Now let us look at the potential of the measure $\mes_r$ which is the extension to the plane of the angular measure over $C_r(0)$ (the total measure is $2\pi$). This is the beginning of Example~2 in \cite{R60II}.
We have
\begin{equation*}
\begin{split}
p(z;\mes_r)&= \frac{r}{2\pi}\int_0^{2\pi}  \ln\left|z-r\E^{\I\theta}\right|\D\theta \\
&= \frac{r}{2\pi}\int_0^{2\pi}  \ln\left|\vert z\vert -r\E^{\I(\theta-\operatorname{arg}(z))}\right| \D\theta \\
&=\frac{r}{2\pi}\int_0^{2\pi}  \ln\left|\vert z\vert -r\E^{\I\theta}\right|\D\theta
\end{split}
\end{equation*}
hence $p(\mes_r)$ is radial over $Q_r$: $p(z;\mes_r)=f(t)$ with $|z|=t$. Moreover by Lemma~\ref{lem:harmonic out support}, it is harmonic out of the circle $C_r$. For a radial function we have
$$\Delta p(z;\mes_r)= f''(t)+f'(t)/t~, $$
that is a second order Euler equation, but can be handle  easily by considering 
$y=f'$. The solutions of $y'+y/r=0$ are $c/t$ for some constant $c$, hence the radial harmonic functions on the plane have the form $a\ln |z| +b$, for some constants $a,b$. As $p(\mes_r)$ is harmonic out of  the support of $\mes_r$,
it is harmonic  over $Q_r$, hence constant there (because $\ln|\cdot|$ is not harmonic at the origin), and taking $z=0$,  we have $p(z;\mes_r)=r\ln r$. As 
$\lim_{\vert z\vert\to\infty} p(z,\mes_r)-r\ln \vert z\vert=0$, we obtain that out of $\overline{Q}_r$,
$p(z;\mes_r)=r\ln \vert z\vert$.
}\end{example}

It is easy to see that $z\mapsto \ln \vert z \vert$ is in $L^p_{loc}$ for $p\geq 1$. Indeed, 
$\iint_{Q_R} \vert \ln \vert^p = 2\pi \int_0^R r\vert \ln r\vert^p \D r $, and the result follows 
for example by the substitution $t=\ln r$. This elementary result can be generalized using the following classical result, see e.g. \cite[Theorem~2.6.2]{ransford} or \cite[Theorem~3.3]{Rudin}.

\begin{lemma}[{Jensen inequality (Cauchy inequality in \cite{R60II})}]\label{jensen inequality}\index{Jensen inequality}\index{Cauchy inequality}
Let $\mes$ be a finite positive measure with compact support, non identically equal to zero, and let $\psi:\mathbb{C}\to (-\infty,+\infty)$ be a convex function. For any $\mes$-integrable function $f:\mathbb{C}\to (-\infty,+\infty)$, we have
$$\psi \left(\frac{1}{\mes(\C)}\iint f \D\mes\right)  \leq \frac{1}{\mes(\mathbb{C})} \iint \psi\circ f \D\mes~. $$
\end{lemma}

\begin{lemma}\label{lem:pot Lp}
A potential is in $L^p_{loc}$  for $p\geq 1$.
\end{lemma}
\begin{proof}
A potential is Borel measurable  by semicontinuity.
If $\mes=0$ then $p(\mes)=0$ and we are done. Otherwise,
let us use Lemma~\ref{jensen inequality} with $\psi(z)=\vert z\vert^p$. For a compact set $O$,
\begin{equation*}
\begin{split}\iint_O \left\vert p(\mes)\right\vert^p 
&= \iint_O \left\vert \frac{1}{\mes(\C)} \iint \ln \vert z-\zeta\vert^{\frac{\mes(\C)}{2\pi}}  \D\mes(\zeta) \right\vert^p \D z \\
&\leq  \frac{1}{\mes(\C)} \iint_O \iint  \left\vert\ln \vert z-\zeta\vert^{\frac{\mes(\C)}{2\pi}}  \right\vert^p \D\mes(\zeta)  \D z \end{split}\end{equation*}
and the result follows by Fubini and because $\ln \vert \cdot\vert$ is locally in $L^p$.
\end{proof}

\begin{remark}\label{rem:pot fini a.e.}{\rm
It follows from Lemma~\ref{lem:pot Lp} that a potential  is finite almost everywhere on $M$. We will see a stronger result in Section~\ref{sec: hausdorff sub}.}
\end{remark}

Let $\mes_1, \mes_2$ be two positive finite measures with compact support, and let $h$ be a continuous function with compact support. Recall that the \emph{convolution}\index{convolution} $\mes_1\star\mes_2$ is the measure defined by 
$$\iint h(z) \D (\mes_1\star\mes_2)(z)=\iint \iint h(\zeta+z)\D\mes_1(\zeta)\D\mes_2(z)~.$$

If $\mes_1$ has a Lebesgue density $f_1$, we can define the convolution of
 $f_1$ and $\mes_2$:
\begin{equation*}
 \begin{split}
\iint h(z) \D (f_1\star\mes_2)(z)= & \iint\iint h(\zeta+z)f_1(\zeta)\D\zeta \D\mes_2(z) \\
=& \iint \iint h(z)f_1(z-\zeta)\D z\D\mes_2(\zeta)~,
\end{split}
\end{equation*}
i.e.
\begin{equation}\label{eq:conv fct mes}(f_1\star\mes_2)(z)=\iint f_1(z-\zeta)\D\mes_2(\zeta) \end{equation}
and if $\mes_2$ has also a Lebesgue density $f_2$, then we recover the usual definition of convolution for functions:
$$(f_1\star f_2)(z)=\iint f_1(z-\zeta)f_2(\zeta) \D\zeta~. $$

\begin{remark}\label{rem:lap convolution}{\rm
The definition \eqref{eq:conv fct mes} gives in particular, 
$$p(z;\mes)=(\ln \vert\cdot\vert \star \mes) (z)~.$$
Now let $\varphi$ be a smooth function with compact support. By Lemma~\ref{lem:pot Lp}, it is meaningful to speak about the weak Laplacian of a potential. Then, a direct utilization of Fubini theorem and Lemma~\ref{lem:potential dirac} give
\begin{equation*}
\begin{split}
\Delta p(\mes) (\varphi) =& \iint \Delta \ln \vert \cdot -\zeta\vert (\varphi) \D\mes (\zeta) \\
= & \iint \varphi (\zeta) \D\mes(\zeta)~.
\end{split}
\end{equation*}
}\end{remark}

Let us state Remark~\ref{rem:lap convolution} as follows.

\begin{proposition}\label{propr: potentiel l1}
In the sense of positive measures,
\begin{equation}\label{eq:delta potentiel}\Delta p(\mes) = \mes~. \end{equation}
\end{proposition}

\begin{definition}\label{def:subh}
A function $u:M\to [-\infty,+\infty)$, defined over a domain $M$ of the plane, is a \emph{subharmonic function}  \index{subharmonic function} if there   exists a 
finite positive measure with compact support $\mes$ and a harmonic function $h$ over $M$ such that 
$$u=p(\mes) + h~.$$
\end{definition}

As harmonic functions are $C^\infty$
and satisfy equality in \eqref{eq:supermean eq}, subharmonic functions inherit all the properties of the potential we have seen. But be careful that if a potential is defined over the entire plane, it is not necessarily the case for a subharmonic function, because of the harmonic term.

%
%
%
%
Let us recall the following classical result. It is a straightforward adaptation of a property of holomorphic functions, see e.g. \cite{ransford} for details.
\begin{lemma}[Identity principle]\label{lem:Identity principle}
Let $h_1$ and $h_2$ be two harmonic functions on a domain $M$ in the plane. If $h_1=h_2$ over a non-empty open subset of $M$, then they are equal over $M$.
\end{lemma}

It follows from \eqref{eq:delta potentiel} that if two subharmonic functions $p(\mes_1) + h_1$ and $p(\mes_2) + h_2$ coincide over an open set $U$, then $\mes_1\vert_U=\mes_2\vert_U$. 
However there is no identity principle as Lemma~\ref{lem:Identity principle}, as the functions $u_1=0$ and $u_2=\max(\operatorname{Re}(z),0)$ show. As a consequence, one may change locally the measure and the harmonic function defining a subharmonic function. Section~\ref{sec:localization} will be devoted to describe a procedure that allows to replace the harmonic term by a constant by  changing the measure over discs. 


We will denote by $p(\mes,h)$\index{$p(\mes,h)$} the subharmonic function which at $z$ takes the value
$$p(z;\mes,h)=p(z;\mes)+h(z)~. $$

\begin{remark}{\rm
Definition~\ref{def:subh} is the most adapted to our needs. Subharmonic functions are  more often defined as  upper semicontinuous functions satisfying  the super-mean inequality \eqref{eq:supermean eq}. Actually, this alternative definition may allow the constant function equal to $-\infty$ as a subharmonic function, that is excluded by our definition. The relation between this alternative definition of subharmonic function and our's is given by 
 the  \emph{Riesz Decomposition Theorem},\index{Riesz decomposition} that says  that an upper semicontinuous functions satisfying  the super-mean inequality satisfies our Definition~\ref{def:subh} on any relatively compact open subset of $M$, see \cite[Theorem 3.7.9]{ransford}. (The function $-p(\mes)$ which is  used in Reshetnyak's articles is the \emph{antipotential} of $\mes$, and is
a \emph{superharmonic} function.)}
\end{remark}

\begin{remark}\label{rem:c2ssh} {\rm
Suppose that the boundary of $M$ is regular, and
let $u$ be a $C^2$ function over $\overline{M}$ such that 
$\Delta u\geq 0$. Then $u$ is subharmonic over $M$.
Indeed, this is a straightforward consequence of Lemma~\ref{lem:poisson}: for $z\in M$
$$u(z)=p(z;\mes)+h(z) $$
where 
$\mes=\Delta u \mathcal{L}_{\vert M}$ where  $\mathcal{L}_{\vert M}$ is the restriction of the Lebesgue measure to $M$, so that $\mes$ is a positive finite measure with compact support; and 
$$h(z)=
\frac{1}{2\pi} \int_{\partial M}  u\frac{\partial \ln|z- \cdot|}{\partial \nu}-\ln|z-\cdot| \frac{\partial u }{\partial  \nu}$$
defines a harmonic function over $M$.
}\end{remark}

\begin{remark}[Sobolev regularity]\label{rem:W1p}{\rm
Recall that for a function $f$ which is locally $L^1$ over $M$, a function $v\in L^1_{loc}(M)$ is the \emph{weak derivative}\index{weak derivative} of $f$ in the direction $x$ if for any $C^1$ function $\varphi$ with compact support in $M$,
$$\iint_M v \varphi =- \iint f \frac{\partial \varphi}{\partial x}~, $$
and we define similarly the weak derivative in the direction $y$. For example, when it is defined, for $z=x+\I y$, 
$f(z)= \frac{x}{\vert z\vert^2}$ is the partial derivative of $\ln\vert \cdot\vert$ in the direction of $x$, and using polar coordinates it is immediate that it is locally in $L^1$, and it follows easily that $f$ is the weak derivative of $\ln\vert \cdot\vert$. Actually, $f$ is locally in $L^p$ for $1\leq p <2$. It is exactly the same 
for the weak partial derivative in the direction of $y$. We then have that $\ln\vert \cdot\vert$  is locally in  the \emph{Sobolev space} $W^{1,p}$ for $1\leq p <2$: it is locally in $L^p$ and its weak first partial derivatives are also locally in $L^p$ for $1\leq p <2$. See e.g. \cite{jost-post} for more details.

An utilization of Jensen inequality in a way similar to the proof of Lemma~\ref{lem:pot Lp} implies that any subharmonic function is locally in $W^{1,p}$  for $ 1\leq p <2$.
}
\end{remark}

\subsection{Approximation}\label{sec:approx}

We will use the fact that a radial subharmonic function is increasing with respect to the radius. This can be easily seen as follows:
along any line, as the weak Laplacian is non-negative, the weak second derivative is non-negative hence we have convexity with minimum at the origin by \eqref{eq:supermean eq}. But we will give an argument based on the following maximum principle.

\begin{lemma}[Maximum Principle]\label{princip max}
Let $u$ be an upper semicontinuous function satisfying \eqref{eq:supermean eq} 
 over a  domain $M$. If $u$ attains its maximal value $A$, then $u$ is constant.

In particular, if $\overline{M}$ is compact and $u$ is defined over $\overline{M}$, then
for any $z\in M$,
$$u(z)\leq \sup_{\partial M} u~.$$
\end{lemma}
\begin{proof}
Let 
$$X=\{z\in M \vert u(z)=A\}$$
and suppose that $X$ is not empty. As $A$ is the maximal value, it follows from \eqref{eq:supermean eq}, by taking sufficiently small values of $r$, that $X$ is open. On the other hand, 
$$M \setminus X = \{z\in M\vert u(z)<A\}$$
is open as $u$ is upper semi continuous. So $M=X$.

If $u$ is defined over the compact set $\overline{M}$, then it has an upper bound that is attained. By the preceding property, either the maximal value is attained on the boundary, or $u$ is constant, that leads to the same conclusion.
\end{proof}

\begin{remark}\label{rem:harmonic majorant}{\rm The Maximum Principle Lemma~\ref{princip max} justifies the name ``subharmonic''. Indeed, if $u$ is a subharmonic function over $Q_r(z)$, for any harmonic function $h$ over the disc such that $u=h$ on the boundary, then $u-h$ satisfies the condition of the Maximum Principle Lemma~\ref{princip max}, so that $u\leq h$.
}\end{remark}

Let us mention the following obvious formula:
if $f$ is continuous, then for sufficiently small $r$, $\frac{1}{2\pi r}\int_{C_r(z_0)}f \leq f(z_0)+\epsilon$, from which it is easy to see that
$$\frac{1}{2\pi r}\int_{C_r(z_0)} f\xrightarrow[r\to 0]{}f(z_0)~.$$
Note also that as
$$\frac{1}{\pi r^2}\iint_{Q_r(z_0)}f
=\frac{1}{\pi r^2}\int_0^r \int_{C_R(z_0)}
f \D R \leq \frac{1}{\pi r^2}\int_0^{r}2\pi R(f(z_0)+\epsilon)\D R
=f(z_0)+\epsilon~,$$
we also have 
$$\frac{1}{\pi r^2}\iint_{Q_r(z_0)} f\xrightarrow[r\to 0]{}f(z_0)~.$$

\begin{lemma}\label{lem:fonctions radiales}
 Let $\mes$ be a finite positive measure with compact support.
The function $r\mapsto \frac{1}{2\pi r}\int_{C_r(z_0)} p(\mes)$ is  non-decreasing and 
$$\frac{1}{2\pi r}\int_{C_r(z_0)} p(\mes)\xrightarrow[r\to 0]{} p(z_0;\mes)~.$$
\end{lemma}
\begin{proof}
Let us define 
$$\bar{v}(r)=\frac{1}{2\pi r}\int_{C_r(z_0)} p(\mes)$$
and
$$v(z)=\frac{1}{2\pi }\int_0^{2\pi} 
p(z\E^{\I t};\mes)\D t~. $$
Similarly to the proof of Lemma~\ref{lem:eq:supermean eq}, it is easy to see that as $p(\mes)$ satisfies \eqref{eq:supermean eq}, then $v$ satisfies \eqref{eq:supermean eq}. As $p(\mes)$ is bounded on compact sets and upper semicontinuous, it follows that $v$ is also upper semicontinuous (see the proof of Lemma~\ref{lem: pmu lower}). So Lemma~\ref{princip max} applies, and as $v$ is radial, we obtain that for $z\in Q_r(0)$, i.e. $\vert z\vert=r'<r$,
$$\bar{v}(r')=v(z)\leq \sup_{z\in C_r(0)}v(z)=\bar{v}(r)$$
so $\bar{v}$ is non-decreasing. 
It also follows that $\liminf_{r\to 0} \bar{v}(r)\geq \bar{v}(0)$, and the last result follows because $v$ is upper semicontinuous: $\limsup_{r\to 0} \bar{v}(r)\leq \bar{v}(0)$.
\end{proof}

The following fundamental result (Proposition~\ref{prop:approx potentiel}) is obtained from a simple convolution (or \emph{Sobolev averaging}\index{Sobolev averaging} as it is called in \cite{R61}), but \eqref{eq:supermean eq}
allows a strong conclusion.

Recall that a sequence of finite positive measure with compact support $(\mes_n)_n$ 
\emph{converges weakly}\index{weak convergence} to $\mes$ if for any continuous bounded function $f$ over the plane,
\begin{equation}\label{eq:weak conv def}\iint f \D\mes_n \xrightarrow[n\to\infty]{} \iint f\D\mes~. \end{equation}

\begin{remark}\label{rem:def weak cv mes}{\rm
The terminology of weak convergence of measures is ambiguous in the literature. However, we will be concerned only 
by sequences of measures which have their support in a same compact set. In that case, the weak convergence is equivalent to
consider \eqref{eq:weak conv def} only for continuous $f$ with compact support.
}~\end{remark}

\begin{proposition}\label{prop:approx potentiel}
Let $\mes$ be a finite positive measure with compact support. Then there exists a sequence  $\left(\mes_n\right)_n$  of finite positive measure, with compact support and smooth Lebesgue densities, such that $\mes_n$ converge weakly  to $\mes$.

Moreover,  $p(\mes_n)$ is a non-decreasing sequence pointwise converging to $p(\mes)$.
\end{proposition}

\begin{proof}

Let us take a Sobolev hat $\alpha$ defined e.g., by
$$\alpha(z)=\begin{cases}
c\operatorname{exp}\left(\frac{1}{|z|^2-1}\right) & \mbox{if }|z|< 1~\\
0 & \mbox{if }|z|\geq 1~
\end{cases}~,$$
where  $c$ is such that
$\iint \alpha =1.$  Let us define
$$\alpha_h(z)=\frac{1}{h^2}\alpha\left(\frac{z}{h}\right)\,~,\, f_h= \alpha_h\star \mes$$
and let $\mes_h$ be the measures with $C^\infty$ Lebesgue density $f_h$. 
Let $\epsilon>0$ and let $\varphi$ be a continuous function with compact support. In particular, $\varphi$ is uniformly continuous over compact sets, and  for any $z$ and for sufficiently small $h$, 
$\sup_{\zeta\in \overline{Q}_h}\vert\varphi(z+\zeta)-\varphi(z)\vert <\epsilon$. So we have
\begin{equation*}\begin{split}
\left\vert \iint \varphi \D\mes_h -\iint \varphi \D\mes\right\vert\leq \iint\iint_{Q_h} \vert\varphi(z+\zeta)-\varphi(z)\vert \alpha_h(\zeta) \D\zeta \D\mes(z) \\
\leq \epsilon \mes(\mathbb{C}) \iint_{Q_h}\alpha_h(\zeta) \D\zeta =
\epsilon \mes(\mathbb{C})~,
\end{split}
\end{equation*}
i.e.,  $\mes_h$  converge weakly to $\mes$ when $h\to 0$. 

Moreover,
$$p(\mes_h)=\ln\vert\cdot\vert \star \mes_h
=\ln\vert\cdot\vert \star f_h=\ln\vert\cdot\vert \star (\alpha_h\star \mes)=(\ln\vert\cdot\vert \star \mes)\star \alpha_h=p_\mes\star \alpha_h$$
is a $C^\infty$ function, and as $\alpha_h$ is a radial function,  substituting $t=r/h$ in the last line:
\begin{equation*}\begin{split}
p(z;\mes_h)=& \iint p(z-\zeta;\mes)\alpha_h(\zeta)\D\zeta \\
=&  \int_0^{h} \left( \int_0^{2\pi}p(z-r\E^{\I\theta};\mes) \D\theta\right)  \frac{r}{h^2} \alpha\left(\frac{r}{h}\right)\D r  \\
=& \int_0^1 \left( \int_0^{2\pi}p(z-ht\E^{\I\theta};\mes) \D\theta\right)  t \alpha\left(t\right)\D t ~.
\end{split}
\end{equation*}

By Lemma~\ref{lem:fonctions radiales},  the term into the parenthesis non-decreases to $2\pi p(z;\mes)$ when $h\to 0$, so by Levi's monotone convergence, $p(z;\mes_h)$ non-decreases to
$p(z;\mes)$. 
\end{proof}

It is proved in Lemma~5 in \cite{R61} that if a sequence of measures  converges weakly, then the associated sequence of potential converges in $L^p$ for $p>1$. For more results about convergence, see \cite{landkof}, especially Theorem 3.8.

Let us end this section with some applications of Proposition~\ref{prop:approx potentiel}.

Let us first mention the following famous result. It  may be proved by a simple convolution, see e.g. \cite[Lemma 3.7.1]{ransford} or \cite[Theorem 4.20]{demailly}, or using the regularity of the  mean equality more explicitly \cite[Theorem 4.7]{dacorogna}, \cite{rao}.

\begin{theorem}[Weyl Lemma]\label{thm:Weyl's Lemma}\index{Weyl's lemma}
A locally integrable function over a domain $M$ such that its weak Laplacian is zero is equal almost everywhere to a harmonic function.
\end{theorem}

A locally integrable function over $M$ such that its weak Laplacian is a positive finite measure with compact support is called \emph{almost subharmonic}\index{almost subharmonic}.

\begin{lemma}\label{lem:a.e.=pp sh}
An   almost subharmonic function is equal almost everywhere to a subharmonic function. 
\end{lemma}
\begin{proof}
Let $u$ be an almost subharmonic function.  Then $\Delta (u-p(\Delta u))=0$, hence by the Weyl Lemma Theorem~\ref{thm:Weyl's Lemma},
  $u-p(\Delta u)$ is equal almost-everywhere to a harmonic function $h$, so an almost subharmonic function is equal almost everywhere to the subharmonic function $p(\Delta u)+h$.
  \end{proof}
   
A subharmonic function is an upper semicontinuous  almost subharmonic function. But those two properties do not characterize subharmonic functions, 
 as shows the following elementary example taken from 
 \cite{demailly}.
Let
$v$ be the function that assigns the value $1$ to the points of a bounded closed set of measure zero in the plane, and zero otherwise. It is upper semicontinous, and its weak Laplacian is the zero measure $\mes$. So by our definition, if $v$ was subharmonic, it should be harmonic, that  is not. Let us mention the following result,
  that was initially proved in \cite{szpilrajn},  see also \cite{schwartz}.
  
  \begin{lemma}[Szpilrajn Lemma]\label{lem almost cont}
  A continuous almost subharmonic function is subharmonic.
  \end{lemma}
\begin{proof}
Let $u$ be a continuous almost subharmonic function. It is equal almost everywhere to a subharmonic function 
$v$. 
Let us consider $u_n$ and $v_n$ which are obtained from  $u$ and $v$ by convolution by the  regularizing $C^\infty$ functions from the proof of Proposition~\ref{prop:approx potentiel}. 
As $u=v$ almost everywhere, then $u_n=v_n$. As $u$ is continuous, $u_n$ converge pointwise to $u$, see e.g., \cite{folland}. By Proposition~\ref{prop:approx potentiel}, $v_n$ converge pointwise to $v$, so $u=v$.
\end{proof}

\begin{lemma}\label{lem:weak identity}
Let $u_1, u_2$ be two subharmonic functions equal almost everywhere over $M$. Then they are equal over $M$.
\end{lemma}
\begin{proof}
Let us first suppose that $u_1$ and $u_2$ are potentials. The approximations given by the proof of Proposition~\ref{prop:approx potentiel} are then equal, and passing to the limit the functions are equal. In the general case, harmonic functions on $M$ are added to the potentials, they may not be integrable over the whole $M$. To overcome this, it suffices to make the convolution at the step $n$ on the set of points at distance larger than $1/n$ from the boundary of $M$.
\end{proof}

%
%

\subsection{Polar sets}\label{sec: hausdorff sub}

Let us introduce the \emph{infinity set}\index{infinity set} of a finite positive measure over the plane, with compact support, $\mes$ :
\begin{equation}\label{eq:sing mes}
\sing =\{z\in \C : p(z;\mes)=-\infty\}~. 
\end{equation}
Properties of the potential imply
\begin{itemize}
\item  $\sing=\bigcap_{n\in \N} \{ z : p(z;\mes)< -n\}$, so
$\sing$ is a $G_\delta$ set: a countable intersection of open  sets (because $p(z;\mes)$ is upper semicontinuous by
Lemma~\ref{lem: pmu lower});
\item $\sing\subset \operatorname{supp}(\mes)$;
\item $\sing$ is a Borel set of zero Lebesgue measure.
\end{itemize}

We will need a more precise description of $\sing$, in order to define  the restriction of the difference of two subharmonic functions to some arcs.


For a subset $A$ of a metric space, for $\alpha\geq 0$ and $\delta \in (0,+\infty]$, let us denote
$$\mathcal{H}^\alpha_\delta(A)=c_\alpha\inf \left\{ \sum_{i\in I} (\operatorname{diam}(A_i))^\alpha  : \operatorname{diam}(A_i)<\delta,\, A \subset \bigcup_{i\in I} A_i \right\}~,$$
where $c_\alpha$ is a positive constant which will be fixed below.

Note that for a given $\alpha$, the function $\delta\mapsto \mathcal{H}^\alpha_\delta(A)$ is  non-negative and non-increasing. 
The \emph{$\alpha$-dimensional Hausdorff measure} of $A$ is $$\mathcal{H}^\alpha(A)=\sup_{\delta>0} \mathcal{H}^\alpha_\delta (A)~,$$ and it is actually a (Borel) measure and the \emph{Hausdorff dimension} of $A$ is $$\dim_H(A)=\inf\{\alpha>0 : \mathcal{H}^\alpha(A)=0\}~.$$ 

We will need the following, see e.g. \cite{AT}.
\begin{proposition}\label{prop: pte Haus}
 For any Borel set $A\subset \R^n$, $\mathcal{H}^n(A)$ is equal to the Lebesgue measure of $A$ (for a suitable constant $c_n$).
\end{proposition}

We want to show that the infinity set has Hausdorff dimension zero. We essentially follow Section~5.9 in \cite{AG}.

\begin{lemma}\label{lem: majoration sup}
Let $\mes$ be a finite  positive measure over $\C$ with compact support, and let $p(\mes)$ be its potential. Then for any $\alpha > 0$ and $z\in \mathbb{C}$,

$$ p(z;\mes) \geq \frac{1}{2\pi \alpha} \inf \left\{-\mes(Q_t(z))t^{-\alpha}: t\in ]0,1]\right\}~.  $$ 

\end{lemma}

\begin{proof}
Let $z\in \C$. Let $\chi_1(\zeta)=1$ if $\vert z-\zeta\vert \leq 1$ and zero otherwise. Let $\chi_2(t)=1$ if $t\geq \vert z-\zeta\vert$ and zero otherwise. By Fubini, 

\begin{eqnarray*}
\ 2\pi p(z;\mes)&=& \iint \ln|z-\zeta| \D\mes(\zeta) \\
& \geq &  \iint_{ Q_1(z)} \ln|z-\zeta| \D\mes(\zeta)   \\
\ &=& -\iint_{\vert z-\zeta\vert\leq 1}
\int_{\vert z-\zeta\vert}^1 t^{-1}\D t\D\mes (\zeta)\\ &=& -\iint \chi_1(\zeta) \int_0^1 t^{-1} \chi_2(t) \D t \D\mes(\zeta)\\
\ &=& -\int_0^1 \iint (\chi_1 \times \chi_2)(\zeta,t)\D\mes(\zeta) t^{-1} \D t\\
&=& -\int_0^1 \mes(\{\zeta : \vert z-\zeta\vert\leq t\}) t^{-1} \D t\\
&=& -\int_0^1 \mes(Q_t(z))t^{-\alpha}t^{\alpha-1}\D t~.
\end{eqnarray*}
~\end{proof}

For the following statement see 
e.g.,
\cite{AT}, \cite{AG}, \cite{Falconer}.

\begin{lemma}[Vitali Covering Theorem]\label{lem: vitali covering}
For any collection of discs $\{Q_{r_i}(z_i), i\in I\}$ such that $\sup_{i\in I} r_i <+\infty$, there is a countable subset $J\subset I$ such that the discs $Q_{r_j}(z_j)$ are disjoint for $j\in J$, and
$$\bigcup_{i\in I} Q_{r_i}(z_i) \subset \bigcup_{j\in J} Q_{5r_j}(z_j)~.$$

\end{lemma}

\begin{lemma}\label{cor:hausdorff}
Let $\alpha >0$. For any $\varepsilon >0$, there is a countable   cover of 
$\sing$ by open discs of radii $r_i$,  with $\sum r_i^\alpha <\epsilon$.
\end{lemma}
\begin{proof}
 For any  $z\in \sing$,  by definition,  Lemma~\ref{lem: majoration sup} gives $$\inf\left\{-\mes(Q_t(z))t^{-\alpha}, t\in]0,1] \right\}=-\infty~,$$
i.e.
$$\sup\left\{\mes(Q_t(z))t^{-\alpha}, t\in]0,1] \right\}=+\infty~.$$ 
  Hence there exists 
$0<r_z\leq 1$ such that 
\begin{equation}\label{eq: mu disque rz}\frac{\mes(Q_{r_z}(z))}{r_z^\alpha} > \frac{5^\alpha\mes(\C)}{\varepsilon}~.\end{equation}
As $r_z\leq 1$ for any $z$, we can apply Lemma~\ref{lem: vitali covering}: there exists a countable set $J$ such that
$\sing\subset \bigcup_{j\in J} Q_{5r_{z_j}}(z_j)$, and
the $Q_{r_{z_j}}(z_j)$ are disjoint. Hence by \eqref{eq: mu disque rz},
$$\sum_{j\in J} (5r_{z_j})^\alpha \leq \varepsilon \frac{\sum_{j\in J} \mes(Q_{r_{z_j}}(z_j) )}{\mes(\C)}\leq  \varepsilon~,
$$
where the last inequality holds because the family is disjoint.
\end{proof}

\begin{proposition}\label{prop:haus zero}
The Hausdorff dimension of $\sing$ is zero, i.e. for any $\alpha >0$, the $\alpha$-dimensional Hausdorff measure of $\sing$ is zero.
\end{proposition}

Let us recall that a subset of $\R^n$ with Hausdorff dimension $<1$ is totally disconnected: its connected components are points.
\begin{proof}
Let $\alpha>0$ and $\delta>0$. For any $\delta^\alpha > \epsilon >0$, by Lemma~\ref{cor:hausdorff}, there exists a covering of $\sing$
by sets of diameter $r_i$ such that $\sum r_i^\alpha<\epsilon$. In particular, $r_i < \delta$. Hence $\mathcal{H}_\delta^\alpha(\sing)<\epsilon$, hence  $\mathcal{H}_\delta^\alpha(\sing)=0$ as $\epsilon$ was arbitrary. But $\delta$ is also arbitrary, hence $\mathcal{H}^\alpha(\sing)=0$.
\end{proof}

The following result is a particular case of the Area  Formula for Lipschitz mappings, see \cite[Remark 2.72]{AFP}. 
Recall that, by Rademacher Theorem, a Lipschitz mapping is differentiable almost everywhere.

\begin{theorem}\label{thm:area formula}
Let $f:\R \to \R^2$ be an injective Lipschitz mapping, 
and $E\subset \R$ a Lebesgue measurable set, then
$$\mathcal{H}^1(f(E))=\int_E | f' |~. $$
\end{theorem}

\begin{lemma}\label{lem: measure 0 intervalle}
Let $\param:[0,l]\to \C$ be 
the arc length parameterization of a rectifiable  arc. 
Then $\param^{-1}(\sing)$ has zero Lebesgue measure in $[0,l]$.
\end{lemma}
\begin{proof}
Up to apply the Tietze Extension Theorem, we use Theorem~\ref{thm:area formula}, for $E=\param^{-1}(\sing)$, which is a Borel set.
By Proposition~\ref{prop:haus zero}, $\mathcal{H}^1(\sing)=0$. Up to a set of Lebesgue zero measure, $|\param'|=1$, hence the Lebesgue measure of $\param^{-1}(\sing)$ is zero.
\end{proof}

 We can be a bit more precise in the case of curves of bounded rotation, having Remark~\ref{rem Gammah} in mind. We need the following result, whose proof is left to the reader.

\begin{lemma}\label{lem:haus lip}
Let $f$ be a $K$-Lipschitz mapping between two metric spaces $X$ and $Y$. Then for any $A\subset X$, for any $\alpha \geq 0$,
$\mathcal{H}^\alpha(f(A))\leq K^\alpha \mathcal{H}^\alpha(A)~. $
\end{lemma}
\begin{lemma}
Let $\param:[0,l]\to \C$ be 
the arc length parameterization of a rectifiable  arc. Suppose that
 $\param^{-1}$ is Lipschitz.
Then $\param^{-1}(\sing)$ is Lebesgue measurable and has zero Hausdorff dimension.
\end{lemma}
\begin{proof}
As $\param^{-1}$ is Lipschitz, $\param^{-1}(\sing)$ is Lebesgue measurable. The result is then immediate 
by Proposition~\ref{prop:haus zero} and Lemma~\ref{lem:haus lip}.
\end{proof}

%

In potential theory, a set $E$ is \emph{polar}\index{polar set} if 
there exists a measure with potential  constant $=-\infty$ on $E$. Obviously $\sing$ is a polar set, and  Proposition~\ref{prop:haus zero} is the classical result that a polar set has Hausdorff dimension zero. 
A property holds \emph{quasi-everywhere}\index{quasi-everywhere} on a domain $M$ if it holds up to a polar set. In the following, the only polar set we will need to consider is $\sing$, and ``quasi-everywhere'' can always be read as ``up to a set of zero Hausdorff dimension''. As an example,  a subharmonic function is finite quasi-everywhere (recall that the constant function equal  to $-\infty$ is not considered as a subharmonic function in the present text).

The simplest example of a subharmonic function with a non-empty polar set is, for $\mes_0\geq 0$, $p(z;\mes_0\delta_\zeta)=\ln \left| z-\zeta\right|^{\frac{\mes_0}{2\pi}}$, whose polar set is $\{\zeta\}$. This simple example can be made wilder, by taking an uncountable dense subset supporting Dirac measure, suitably weighted see e.g.,  \cite[Theorem 2.5.4]{ransford}, \cite[5.34]{rao}, \cite[Example 3.3.2]{AG}.  See also
  \cite{imomkulov}, \cite{cartan},   Theorem 5.5.8  in \cite{AG} for related results. 
We also have the following. 
\begin{theo}
\begin{theorem}[{\cite[Theorem~1]{R63}}]
Let $p(\mes)$ be a potential over a bounded domain $M$.
For any $\epsilon>0$,  there exists a sequence $(Q_n)_n$ of open discs with sum of diameters less than $\epsilon$ such that 
$u$ is continuous out of the union of the discs $Q_n$.
\end{theorem}
\end{theo}

%
%
%

\subsection{Delta-subharmonic functions}\label{sec: delta sub}


Delta-subharmonic functions are simply differences of subharmonic functions. But as a subharmonic function may be infinite at some points, a difference of subharmonic functions may not be defined everywhere. 

\begin{definition}
Let $M$ be a bounded domain in the plane. A function 
$v$, defined quasi-everywhere on $M$, is called  \emph{$\delta$-subharmonic function over $M$}\index{$\delta$-subharmonic function} if there exists two subharmonic functions $u_1$ and $u_2$, such that 
$v=u_1-u_2$.
\end{definition}

Let us write $u_i=p(\mes_i)+h_i$. 
The $\delta$-subharmonic function exists only on $\C\setminus \left(\mathcal{I}_{\mes_1}\cap \mathcal{I}_{\mes_2}\right)$, i.e., it exists quasi-everywhere. Whenever it is defined, it belongs to $[-\infty,+\infty]$. The function has finite values quasi-everywhere. 

We have the following  example. 

\begin{lemma}\label{lem:ln module holo}
Let $f:M\to \C$ be a function. Then $\ln |f|$ is
\begin{itemize} 
\item harmonic if $f$ is a holomorphic function without zeros;
\item subharmonic if $f$ is a holomorphic function;
\item $\delta$-subharmonic if $f$ is meromorphic.
\end{itemize}
\end{lemma}
\begin{proof}
The first item follows from the fact that $f$ is locally the real part of a holomorphic function (considering the complex logarithm). 
If $f$ is holomorphic, its zeros are isolated, and around such a point the function is the sum of an entire power plus a non-zero harmonic term.
\end{proof}

\begin{remark}[$C^2$ functions are $\delta$-subharmonic]\label{rem:c2 delta ssh}{\rm
For a real valued function $f$, let us denote
$f^+(x)=\max(0,f(x))$ and $f^-(x)=\max(0,-f(x))$,\index{$f^+$} so that $f^+, f^-$ are non-negative and $f=f^+-f^-$. 
Now, let us suppose that the boundary of $M$ is regular and that $u$ is $C^2$ over $\overline{M}$.  By Lemma~\ref{lem:poisson}, over $M$,
$$u(z)=\frac{1}{2\pi}\iint_M  \ln |z-\cdot| (\Delta u)^+ - \frac{1}{2\pi}\iint_M  \ln |z-\cdot| (\Delta u)^- + h(z)~, $$
where $h(z)=\frac{1}{2\pi} \int_{\partial M}   u\frac{\partial \ln|z- \cdot|}{\partial \nu}-\ln|z-\cdot| \frac{\partial  u }{\partial  \nu}$
is harmonic over $M$. It follows that  $u$ is $\delta$-subharmonic.}
\end{remark}

\begin{remark}[Regularity]\label{ex:concave}\label{remark:regularity}{\rm 
Convex functions on an open subset are characterized by the fact that their Hessian in the sense of distribution is non-negative \cite{schwartz}, \cite{dudley2}, \cite{dudley1}. Hence, they are almost subharmonic, and by continuity, Szpilrajn Lemma~\ref{lem almost cont} says that they are subharmonic. 
 It follows  that 
any \dc function is $\delta$-subharmonic. This implies that 
$C^{1,1}$ functions, i.e., functions with locally Lipschitz first derivatives, are $\delta$-subharmonic, as well as for example piecewise affine functions, see \cite{alexdc}. (This is also proved in \cite{arsove}.)  As from Remark~\ref{rem:W1p}, we know that a subharmonic function is locally in the Sobolev space $W^{1,p}$ for $p\in [1,2)$, loosely speaking
 we have:

$$C^2\subset C^{1,1} \subset \mbox{ \dc} \subset \mbox{  $\delta$-subharmonic }
\subset \bigcap_{1\leq p <2} W^{1,p}_{\operatorname{loc}}~. $$
}\end{remark}

\begin{remark}[DC functions]\label{remark DC}{\rm
The $\delta$-convex functions (or DC functions) \index{DC function} 
form a  class of functions that has some importance in metric geometry, in the wake of the Leningrad Geometric School, and  manly after \cite{perelman}. See also \cite{alexdc}, or more recently \cite{ABDC}, \cite{PRZ}.
}\end{remark}

We will often consider $\delta$-subharmonic functions as potentials in the following manner. 
Let $\mes$ be a  signed (Borel) measure\index{signed measure}. By definition, it takes only finite values. Recall that the \emph{total variation measure}\index{$\vert\mes\vert$} $|\mes|$ of $\mes$ is the finite positive measure defined by
$$|\mes|(E)=\sup \sum_{i=1}^\infty |\mes(E_i)|  $$
where the supremum is taken over all the partitions $ \{E_i\} $ of $E$. In general, $|\mes|(E)\geq |\mes(E)|$.
The \emph{positive and negative variations} of $\mes$ are the finite positive Borel measures defined by\index{$\mes^+$}
$$\mes^+=\frac{1}{2}(|\mes|+\mes),\, \mes^-=\frac{1}{2}(|\mes|-\mes)~,$$
 and we have the \emph{Jordan decomposition} of the signed measure $\mes$ as a difference of two finite positive measures: $\mes=\mes^+-\mes^-$. Note also that $|\mes|=\mes^++\mes^-$.
This special decomposition has the following property: if $\mes_1$ and $\mes_2$ are positive measures, and $\mes=\mes_1-\mes_2$, then
$$\mes_1\geq \mes^+,\,\mes_2\geq \mes^-~.$$
In particular, if $\varphi=\mes_1+\mes_2$, then $\varphi \geq |\mes| $. Also, if 
$\mes(E)=0$, then $\mes^+(E)=\mes^-(E)=0$. Let us also recall that $|\mes|(E)=0$ if and only if $\mes(F)=0$ for every measurable subset of $E$. The \emph{support}\index{support (signed measure)} of $\mes$ is defined as the support of $|\mes|$.

If a signed measure $\mes$ has compact support, then 
$\mes^+$ and $\mes^-$ have compact support, and in this case we can define
$$p(\mes):=p(\mes^+) - p(\mes^-)~, $$
 \index{$p_\mes$} and similarly, if $h$ is a harmonic function,
\begin{equation}\label{eq:p mes sing}p(\mes,h):=p(\mes^+) - p(\mes^-)+h~. \end{equation}

Of course, $p(\mes,h)$ is a $\delta$-subharmonic function.

\subsection{Localization of delta-subharmonic functions}\label{sec:localization}

We now give a detailed proof  of the 
Localization Theorem. We will explain its interest at the end of the section, after its proof.
Let us  give some simple preparatory results about harmonic functions.
Let us first recall the following classical formula from the  Neumann problem for the disc.
\begin{lemma}[Dini Formula]
Let $h$ be a harmonic function on the closure of the disc $Q_r(z_0)$. Then
\begin{equation}\label{eq:dini}
\forall z\in Q_r(z_0),\, h(z)=h(z_0)-\frac{1}{\pi}\int_{C_r(z_0)} \ln|z-\cdot|\frac{\partial h}{\partial \nu}~.
\end{equation}
\end{lemma}
\begin{proof}
Without loss of generality, let us prove the result for $C_r(z_0)=C_1(0)=:C$.
Let $h$ be the real part of the holomorphic function $f$. As on the circle $C$ we have
$$\frac{\partial h}{\partial \nu}(z)=\left(z_1\frac{\partial h}{\partial x}(z) + z_2\frac{\partial h}{\partial y}(z)\right) $$
and as 
$$z f'(z)=Df(z)(z)=z_1\frac{\partial f}{\partial x}(z)+z_2\frac{\partial f}{\partial y}(z)~, $$
we have, on the circle,
\begin{equation}\label{eq:normale harm hol}\frac{\partial h}{\partial \nu}(z)=\operatorname{Re} (zf'(z))~.\end{equation}

 Now, recall the Poisson equation 
 that gives the solution of the Dirichlet problem for the disc: for a continuous function $H$ on $C$, define
 $$F(z)=\frac{1}{2\pi }\int_{0}^{2\pi} \frac{\E^{\I\theta}+z}{\E^{\I\theta}-z} H(\theta) \D\theta~.$$
Then $F$ is a holomorphic function over $Q$, and $\operatorname{Re}(F)$ is the unique harmonic function on $Q$ that extends continuously as $H$ on $C$. In particular,

\begin{equation}\label{eq:poisson}
F(z)-F(0)=\frac{z}{\pi }\int_0^{2\pi}\frac{H(\theta)}{\E^{\I\theta}-z}\D\theta~. 
\end{equation} 
 
Now, let  $H=\frac{\partial h}{\partial \nu}$. Then by uniqueness, $F(z)=zf'(z)$ by \eqref{eq:normale harm hol}, and \eqref{eq:poisson} gives 

$$f'(z)=\frac{1}{\pi } \int_0^{2\pi}\frac{\partial h}{\partial \nu}(\theta)\frac{1}{\E^{\I\theta}-z}\D\theta~,$$
so that, if $\operatorname{Ln}$ is the principal determination of the complex logarithm,
$$f(z)=\alpha-\frac{1}{\pi } \int_0^{2\pi}\frac{\partial h}{\partial \nu}(\theta)\operatorname{Ln}( \E^{\I\theta}-z)\D\theta~, $$
where $\alpha$ is a constant.
As $h=\operatorname{Re}f$, we obtain
$$h(z)=\operatorname{Re}\alpha - \frac{1}{\pi}\int_{C} \ln|z-\cdot|\frac{\partial h}{\partial \nu}~. $$
By \eqref{eq:int conjugate}, for $z=0$,
$\int_{C} \ln|z-\cdot|\frac{\partial h}{\partial \nu}=0$, hence $\operatorname{Re}\alpha=h(0)$.
\end{proof}

It will be suitable to write \eqref{eq:dini} in the following manner. 
\begin{lemma}[Localization of Harmonic Functions]\label{lem:localisation harmonic}
Let $h$ be a harmonic function on a domain $M$, and let $z_0\in M$, $r>0$ such that $\overline{Q}_r(z_0)\subset M$. Then there exists a  signed measure  
 $\psi_r$ over the plane with
\begin{itemize}
\item  support the circle $C_r(z_0)$,
\item  bounded variation, and   
 $|\psi_r|(\mathbb{C})\xrightarrow[r\to 0]{} 0$,
 \item $\psi_r(\C)=0$ 
\end{itemize} 
 such that for 
$z\in Q_r(z_0)$,
\begin{equation}\label{eq:loc harm}h(z)= h(z_0) -\frac{1}{\pi}\iint  \ln|z-\zeta|\D\psi_r(\zeta)~.\end{equation}
\end{lemma}
\begin{proof}

Let us write
\begin{equation*}
\begin{split}
\ \int_{C_r(z_0)}  \ln|z-\cdot|\frac{\partial h}{\partial \nu}&= r \int_0^{2\pi} \ln|z-z_0-r\E^{\I\theta}|\frac{\partial h}{\partial \nu}(z_0+r\E^{\I\alpha})\D\theta \\
\   & = \int_0^{2\pi}\ln |z-z_0-r\E^{\I\theta}| \D\tilde{\psi}_r(\theta) 
\end{split}
\end{equation*}
where  $\D\tilde{\psi}_r$ is the Lebesgue–Stieltjes measures over $C_r(z_0)$ given by the (smooth) function
$$\tilde{\psi}_r(\alpha)=r\int_0^{\alpha} \frac{\partial h}{\partial \nu}(z_0+r\E^{\I\theta})\D\theta~. $$
Note that by \eqref{eq:int conjugate}, we have $\tilde{\psi}_r(2\pi)=0$.
From \eqref{eq:dini}, we obtain 
$$h(z)=h(z_0)-\frac{1}{\pi}\int_0^{2\pi} \ln |z-z_0-r\E^{\I\theta}| \D\tilde{\psi}_r(\theta)~. $$

If we denote by $\psi_r$  the measure 
over the plane  with support  $C_r(z_0)$
which 
 extends $\tilde{\psi}_r$, we obtain \eqref{eq:loc harm}. We have

\begin{equation*}\label{eq:total angular measure}
\begin{split}
\ |\psi_r|(\mathbb{C})=|\D\tilde{\psi}_r|([0,2\pi])= \bigvee_0^{2\pi} \tilde{\psi}_r(\alpha) & =r \int_0^{2\pi} \left|\left( \int_0^\alpha 
\frac{\partial h}{\partial \nu}(z_0+r\E^{\I\theta})\D\theta\right)' \right|\D\alpha \\
\ &=r\int^{2\pi}_0 \left| \frac{\partial h}{\partial \nu} (z_0+r\E^{\I\alpha})\right| \D\alpha~,
\end{split}
 \end{equation*}
and  by  regularity of $h$, the right-hand side is bounded from above by $r$ times a constant that depends only on $M$, so $|\psi_r|(\mathbb{C})\to 0$ when $r\to 0$.
\end{proof}

\begin{remark}\label{rem:localization harmonic}{\rm The Localization of Harmonic Function Lemma~\ref{lem:localisation harmonic} will be used at several places. Indeed, given a subharmonic function $p(\mes,h)$ over $M$, this says that over an open disc whose closure is inside $M$, there are a measure $\psi$ and a constant $c$ such that
$$p(\mes,h)=p(\mes+\psi)+c~.$$}
\end{remark}

The localization of (difference of) subharmonic function, Theorem~\ref{thm:localization} below, follows the same procedure, by applying Lemma~\ref{lem:localisation harmonic} to the harmonic part of the subharmonic function.

\begin{lemma}\label{lem:arcsin}
Let $r>0$ and $\zeta\in \C$ with 
$|\zeta|>r$. Then the angular function of the circle (see \eqref{eq:angular function}) satisfies
$$\bigvee_0^{2\pi}\varphi(C_r(0),\zeta,\cdot) = 4 \arcsin \frac{r}{|\zeta|}~.$$
\end{lemma}
\begin{proof}
Let $\param:[0,2\pi]\to \C$ be a parameterization of
$C_r(0)$ such that $\param(0)$ is the orthogonal projection of $\zeta$ onto $C_r(0)$. Then
$\varphi:=\varphi(C_r(0),\zeta,\cdot)$ is decreasing on $[0,\pi/2]$ and $[3\pi/2, 2\pi]$, and increasing on $[\pi/2,3\pi/2]$ hence
\begin{align*} 
\bigvee_0^{2\pi}\varphi &=\bigvee_0^{\pi/2}\varphi+\bigvee_{\pi/2}^{\pi}\varphi
+\bigvee_{\pi}^{3\pi/2}\varphi 
 +\bigvee_{3\pi/2}^{2\pi}\varphi
 \\ &=-2\varphi(C_r(0),\zeta,\pi/2)+2\varphi(C_r(0),\zeta,3\pi/2)\end{align*}
and the result follows, see Figure~\ref{fig:arcsin}.
\end{proof}

%
%

\begin{theo}
\begin{theorem}[{Localization Theorem, \cite[Theorem 6]{R63III}}]\label{thm:localization}
Let $h$ be a harmonic function on a domain $M$, and let $\mes$ be a signed measure with compact support.
For any $\epsilon,\eta >0$, there exist $\delta >0$, $k\in \R$ and a signed measure 
$\tilde{\omega}$ such that 
\begin{itemize}
\item $\tilde{\omega}=\omega$ on $Q_\delta(z_0)$,
\item   $\tilde{\omega}=0$ outside of
  $Q_{\delta(1+\eta)}(z_0)$
\item $|\tilde{\omega}|(Q_{\delta(1+\eta)}(z_0)\setminus\{z_0\})<\epsilon$
\item  on $Q_{\delta}(z_0)$, $$p(\mes,h)=p(\tilde{\omega})+k~.$$
\end{itemize}
\end{theorem}
\end{theo}

\begin{figure}
\begin{center}
\includegraphics{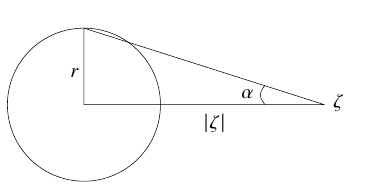}
\caption{To Lemma \ref{lem:arcsin}: $\displaystyle \alpha=-\varphi(C_r(0),\zeta,\pi/2)=\arcsin \frac{r}{|\zeta|}$.}\label{fig:arcsin}
\end{center}
\end{figure}
\begin{proof}
Let $r>0$ be such that $Q_{r(1+\eta)}(z_0)\subset M$.
Let $\mes_r$ be the signed measure that is equal to $0$ on $ \overline{Q}_{r(1+\eta)}(z_0)$ and to 
$\mes$ on the complementary.
We have
$$p(\mes,h)=p(\mes|_{Q_{r(1+\eta)}(z_0)})+p(\mes_r)+h~.$$

As $h$ is harmonic,  by Lemma~\ref{lem:localisation harmonic}, there exists a measure  $\psi_r^1$ concentrated on $C_{r(1+\eta /2)}(z_0)$ such that  for $z\in Q_{r(1+\eta /2)}(z_0)$,
$$h(z)= \iint  \ln|z-\zeta|\D\psi_r^1(\zeta)+h(z_0)$$
and $|\psi^1_r|(\C)\rightarrow 0$ when $r\to 0$. This gives
$$p(\mes,h)=p(\mes|_{Q_{r(1+\eta)}(z_0)})+p(\mes_r)+p(\psi_r^1)+h(z_0)~.$$

Also,
$\overline{Q}_{r(1+\eta /2)}(z_0)$ is in the complementary of the support of $\mes_r$, hence in the complementary of the support of $\mes_r^+$ and $\mes_r^-$, and then  $p(\mes_r)$ is harmonic on $\overline{Q}_{r(1+\eta /2)}(z_0)$. 
 Lemma~\ref{lem:localisation harmonic} gives that there exists a measure  $\psi^2_r$ concentrated on $C_{r(1+\eta /2)}(z_0)$ such that
 for $z\in Q_{r(1+\eta /2)}(z_0)$,
$$p(z;\mes_r)=\iint  \ln|z-\zeta|\D\psi_r^2(\zeta)+p(z_0;\mes_r)~. $$

This gives 
$$p(\mes,h)=p(\mes|_{Q_{r(1+\eta)}(z_0)})+p(\psi^2_r)+p(\psi_r^1)+h(z_0)+p(z_0;\mes_r)~.$$

However, as the harmonic function $p(\mes_r)$ depends on $r$, we cannot directly conclude that the total variation of $\psi_r^2$ goes to $0$ when $r\to 0$. The measure $\psi_r^2$ is defined from the Lebesgue--Stieljes measure $\tilde{\psi}_r^2$ defined by (see the proof of  Lemma~\ref{lem:localisation harmonic})
$$\tilde{\psi}_r^2(\alpha)=r(1+\eta /2)\int_0^\alpha \frac{\partial p(\mes_r)}{\partial \nu}(z_0+r(1+\eta /2)\E^{\I t})\D t~. $$
Using the definition of $p(\mes_r)$, an exchange of derivative and integration, Fubini and Lemma~\ref{lem angle cas c1}:

\begin{align*}
\tilde{\psi}_r^2(\alpha)&=\frac{r(1+\eta /2)}{2\pi}\int_0^\alpha \frac{\partial}{\partial \nu}   \iint_{|\zeta-z_0|>r(1+\eta)} \ln|\zeta-z_0-r(1+\eta /2)\E^{\I t} | \D\mes(\zeta) \D t \\
&=\frac{1}{2\pi} \iint_{|\zeta-z_0|>r(1+\eta)} \int_0^\alpha r(1+\eta /2) \frac{\partial \ln|\zeta -\cdot |}{\partial \nu}(z_0+r(1+\eta /2)\E^{\I t}) \D t \D\mes(\zeta) \\
& =-\frac{1}{2\pi} \iint_{|\zeta-z_0|>r(1+\eta)} \varphi(C_{r(1+\eta/2)}(z_0),\zeta,\alpha)\D\mes(\zeta)~.
\end{align*}

By Lemma~\ref{lem:arcsin}, we obtain 

\begin{align*}|\psi_r^2|(\mathbb{C})&=|\D\tilde{\psi}^2_r|([0,2\pi]) =\bigvee_0^{2\pi}\tilde{\psi}_r^2 \\
& \leq \frac{1}{2\pi}\iint_{|z_0-\zeta|>r(1+\eta)}\bigvee_0^{2\pi}\varphi(C_{r(1+\eta/2)}(z_0),\zeta,\cdot)\D|\mes|(\zeta) \\
&= \frac{2}{\pi} \iint_{|z_0-\zeta|>r(1+\eta)} \arcsin \frac{r}{|z_0-\zeta|}\D|\mes|(\zeta)\\
&= \frac{2}{\pi} \iint A_r(\zeta) \D|\mes|(\zeta)
\end{align*}
where 
$$A_r(\zeta)=
  \left\{
      \begin{aligned}
     \arcsin \frac{r}{|z_0-\zeta|} & \mbox{ when } & |z_0-\zeta|>r(1+\eta)~,\\
      0 & \mbox{ } & \mbox{ otherwise}~,\\
      \end{aligned}
    \right.
$$
which goes to $0$ when $r\to 0$ for any $\zeta$.
 Hence $|\psi_r^2|(\mathbb{C})\rightarrow 0$ when $r\to 0$.

Let us set $\psi_r=\psi^1_r+\psi_r^2$. We have $|\psi_r|(\C)\rightarrow 0$ when $r\to 0$. Let us set
$$\tilde{\mes}_r=\mes|_{Q_{r(1+\eta)}(z_0)}+\psi_r~. $$
Note that, on $Q_r(z_0)$, $\psi_r=0$.
If we set $c_r=h(z_0)+p(z_0;\mes_r)$, we have
$$p(\mes,h)=p(\tilde{\mes}_r)+k_r~,$$
and one checks easily that $\tilde{\mes}_r$ satisfies the required properties for $r$ sufficiently small.
\end{proof} 

The Localization Theorem~\ref{thm:localization} will be essential at at least two places:
\begin{itemize}
\item  to prove the properties of the distance stretching, itself being fundamental in the proof of the Distances Convergence Theorem, see Section~\ref{sec can stretch}, or in the proof of the conformal nature of the distance preserving mappings, see Section~\ref{sec:conformal};
\item to prove that if the measure is not too concentrated, shortest arcs are arcs of bounded rotation with non-positive left and right turns, and arcs of bounded turn are of bounded rotation, see Section~\ref{sec:turn rotation}.
\end{itemize}

\section{Subharmonic distances on the plane}\label{sec:distance}

In all this section, unless specific precision, $M$ is  a bounded domain of the plane, $\mes$ is a signed Borelian  measure over the plane with compact support, $h$ is a harmonic function over $M$.

We defined the function 
$p(\mes,h)$ over $M$ as the difference of two subharmonic functions, see \eqref{eq:p mes sing} and  Definition~\ref{def:subh}.  We denote
\begin{equation}\label{eq:def pl}\lambda(\mes,h)=\E^{-2(p(\mes)+h)}~, \end{equation}
i.e., for $z\in M$,
$$\ln \lambda(z;\mes,h)=-\frac{1}{\pi}\iint \ln \vert z-\zeta\vert \D\mes(\zeta) -2h(z)~. $$
\index{$\lambda(\mes,h)$}
Recall that such formula holds only quasi-everywhere on $M$, that will be implicit most of the time.

Once $\mes$ and $h$ are given, we denote $\lambda(\mes,h)$ by $\lambda$.  In turn, in the sense of distribution,  Proposition~\ref{propr: potentiel l1} applied  to the positive and negative parts of $\mes$ gives that
 \begin{equation}\label{eq:mesure PL}-\frac{1}{2}\Delta \ln \lambda =\mes~. \end{equation}

We may call $\lambda |\D z|^2$ a \emph{subharmonic metric}\index{subharmonic metric}, and call $\mes$ the \emph{curvature measure}\index{curvature measure} of the metric. Apart from the fact that it is defined only quasi-everywhere, conversely to Riemannian metrics, it may reach values $0$ or $+\infty$ at some points. Standard examples are the ones of flat cones (see Lemma~\ref{lem:ln module holo}):
\begin{itemize}
\item if the curvature of the cone is positive, then the subharmonic metric is infinite at the vertex;
\item if the curvature of the cone is negative, then the subharmonic metric is zero at the vertex.
\end{itemize}

We will define the $\lambda$-length of an arc, rectifiable for the Euclidean metric, in Section~\ref{sec:length structure}, and in Section~\ref{sec:intrinsic length} we will define a distance from the length structure, and we present some preliminary convergence results. Section~\ref{sec can stretch}
presents  Reshetnyak's method of canonical stretching of a distance, based on Localization  Theorem~\ref{thm:localization}.  Then Reshetnyak's convergence theorems are presented in Section~\ref{sec cv}.
 
%
%
%

\subsection{Length of arcs}\label{sec:length structure}


Let $\param$ be the arc length parameterization of a rectifiable arc.  By Lemma~\ref{lem: measure 0 intervalle}, $\lambda\circ \param$  is defined almost everywhere on $[0,s(K)]$, where $s(\arc)$ is  the Euclidean length of the arc.
%
 In turn, if now $\param:[0,1]\to \C$ denotes the parameterization proportional to the arc length of a rectifiable arc $\arc$, 
we can define the \emph{length of $\arc$ for the metric $\lambda |\D z|^2$}\index{length of a curve}:
\begin{equation}\label{eq:tilde s}\tilde{s}_\lambda(\arc)= s(\arc)\int_0^1 \lambda^{1/2}(\param(t); \mes,h) \D t~.\end{equation}


In this section, we prove a theorem about convergence of length (Theorem~\ref{thm: basic I}).

The following estimate is based on the  simple idea that, as $\lambda$ is the exponential of the integral of a logarithm, we may use Jensen inequality (Lemma~\ref{jensen inequality}).
The following result is a part of the fundamental Theorem 3.1 in \cite{R60II}. We present it under the form of Lemma 4.2 in \cite{tro-principe}.

\begin{lemma}
\label{re-tro inequality}
Let $\alpha$ and $\beta$ be such that
$$ 0\leq \beta = \frac{\alpha}{\pi}\mes^+(\C)~. $$
Then  if $M$ is included in $Q_R$,
$$\lambda^\alpha(z;\mes)=\E^{2 \alpha(p(\mes^-)-p(\mes^+))(z)}\leq \frac{k_1^\alpha}{\omega^+(\mathbb{\C})}\iint \vert z - \zeta \vert^{-\beta} \D\mes^+(\zeta)~, $$
where $k_1=2R^{\frac{\mes^-(\C)}{\pi}}$.
\end{lemma}
\begin{proof}
First,  $\E^{2p(\mes^-)}$ is bounded from above by  $k_1$ by Lemma~\ref{lem:potential bfb}. 
Then, we write
$$\E^{-2\alpha p(\mes^+)}=\exp \left( \frac{1}{\omega^+(\C)}\iint \ln \vert z-\zeta\vert^{-\beta} \D\mes^+(\zeta)\right)~.$$

The function $\zeta\to -\beta \ln \vert z-\zeta\vert$ is integrable because $\beta \geq 0$ and the following decomposition (up to increase $R$ to get $R>1/2$):
$$-\ln |z-\zeta| = \left(-\ln |z-\zeta| + \ln 2R\right) - \left(\ln 2R\right)~. $$
Both functions in the brackets are non-negative. We can apply Lemma~\ref{jensen inequality}, and the result follows.
\end{proof}

An easy computation \cite[Lemma 8]{R60II} shows that, for any arc length parameterized arc $\param$
 of absolute rotation bounded from above by $A<\pi$, for any $0\leq \beta <1$, for any $\zeta\in M$, then 
\begin{equation}\label{eq:inequalit gamma(h)}\int_0^{s(K)} \frac{\D t}{|\zeta-\param(t)|^\beta} < \frac{2^{1+\beta}s(K)^{1-\beta}}{(1-\beta)\cos(A/2)^\beta}~. \end{equation}

Using this inequality together with an easy improvement of the proof of  Lemma~\ref{re-tro inequality}, \res obtained the following fundamental estimate.
\begin{theo}
\begin{theorem}[{\cite[Theorem 3.1]{R60II} }]\label{prop:fun estimate}
Let $\arc$ be an arc with absolute rotation bounded from above by $A < \pi$. Let $\arc(r)$ be the $r$-neighborhood of $\arc$. For each $\alpha, \beta$ such that
\begin{equation}\label{eq:hyp fund est}0\leq \beta = \frac{\alpha}{\pi}\omega^+(K(r)) <1~,  \end{equation}
we have
$$\int_K \lambda(\mes)^\alpha \leq  A_1 \frac{s(K)^{1-\beta}}{\cos(A/2)^\beta r^b}$$
where $b$ and $A_1$ depend only on $\alpha, \mes^+(K(r))$ and $\mes^\pm(\C)$.
\end{theorem} 
\end{theo}

In the case $\alpha=\frac{1}{2}$, the estimate in Theorem~\ref{prop:fun estimate} gives an upper bound for $\tilde{s}_\lambda$. The condition \eqref{eq:hyp fund est} then becomes
\begin{equation}\label{eq: mes+Kr}\mes^+(K(r))<2\pi~.\end{equation}

In turn, if for a point $z\in M$, 
\begin{equation}\label{eq:mes plu pet 2pi}\mes^+(\{z\}) <2\pi~,\end{equation} as $\mes^+$ is regular, there exists a neighborhood $U$ of $z$ such that 
$\mes^+(U)<2\pi$. Note that as the measures are finite, there is a finite number of points $z$ with $\mes^+(\{z\})\geq 2\pi$. The necessity of condition 
\eqref{eq:mes plu pet 2pi} is given by \eqref{eq:inequalit gamma(h)}. But it also has the following interpretation. Let us write
$$p(z;\mes)=\frac{\mes(\{z_0\})}{2\pi}\ln \vert z-z_0\vert+\frac{1}{2\pi}\int_{\C\setminus \{z_0\}}\ln \vert z-\zeta\vert \D\mes(\zeta)~,$$
then
\begin{equation}\label{eq:decomp mes lambda point}\lambda^{1/2}(z;\mes)=\vert z-z_0\vert^{-\frac{\mes(\{z_0\})}{2\pi}}\E^{-\frac{1}{2\pi}\int_{\C\setminus \{z_0\}}\ln \vert z-\zeta\vert \D\mes(\zeta)}~.\end{equation}


We will prove the following result, which is a consequence of the estimate in Theorem~\ref{prop:fun estimate}. The statement below is indeed a particular case of Theorem~\ref{thm: basic I} that we will mention later. The proof of Proposition~\ref{prop: cv lg}, that we present below, is hence a simplified version of the one of Theorem~\ref{thm: basic I}. The proof of Theorem~\ref{thm: basic I} may be found in \cite{R60II}.

\begin{proposition}\label{prop: cv lg}
Let $(\arc_n)_n$ be a sequence of arcs converging to an arc $\arc$, with absolute rotation uniformly bounded by 
$A<\infty$.
  If \begin{equation}\label{eq:mes cond conv leng}\mes^+(\{z\})<2\pi~\forall z\in \arc~,\end{equation}
 then
$$\tilde{s}_{\lambda}(\arc_n)\xrightarrow[n\to\infty]{}  \tilde{s}_\lambda(\arc)~. $$
\end{proposition}

The elementary Example­~\ref{ex:cylindre 1} of a flat cylinder shows that the result would be false without the assumption \eqref{eq:mes cond conv leng}: on can easily find a sequence of segments converging to the origin (considered as a degenerated curve),
with  constant $\tilde{s}_{\lambda}$ length.

To simplify the notation, we will only deal with the case $h=0$. In the general case, as $h$ is $C^\infty$ over $M$, it is easy to deal with it. For example, if $\lambda=\E^h$, with $h$ harmonic, 
as by Lemma~\ref{lem: BR longueur cv}, $s(K_n)\to s(K)$, it is straightforward  that $\tilde{s}_{\lambda}(\arc_n)\to  \tilde{s}_\lambda(\arc)$.

Due to Lemma~\ref{lem:uniform cv subarcs}, we may consider that $\arc_n$ and $\arc$ have absolute rotation $<\pi$.
By Remark~\ref{rem Gammah}, the arc length parameterization of these subarcs, as well as the inverse of the parameterizations, are equi-Lipschitz mappings.

The proof of Proposition~\ref{prop: cv lg} is based on the following result, which is a combination of a theorem of de  La Vall\'ee-Poussin  and a theorem of D. Vitali, see  Chapter VI, 3, in \cite{natanson}, Theorem~7 p. 159 and  Theorem~4 p. 157, respectively.

\begin{theorem}\label{thm:LVP}
Let $F_n, F_0$ be non-negative  measurable functions over $[0,1]$ such that:
\begin{itemize}
\item $F_n$ \emph{converge in measure}\index{convergence in measure} to $F_0$,  i.e., for any $\epsilon >0$, if $\mathcal{B}$ are the Borelians of the plane, $$\mathcal{L}(\{E\in \mathcal{B} : |F_n(E)-F_0(E)|\geq \epsilon\})\xrightarrow[n\to 0]{}0~,$$ 
where $\mathcal{L}$ is the Lebesgue measure,
\item  there exists a positive increasing function $\Phi$ on $\R^+$ such that $\frac{\Phi(y)}{y}\to \infty$ when $y\to\infty$, and a constant $B$ such that
\begin{equation}\label{eq:la vallee poussin}
\int_0^1 \Phi(F_n(t))\D t \leq B~.
\end{equation}
\end{itemize}
Then
$$\int_0^1 F_0 = \lim_{n\to\infty} \int_0^1 F_n~.$$
\end{theorem}

We will use the following functions defined over $[0,1]$:
$$f^\pm_n:=p(\mes^\pm)\circ \param_n~,\, f^\pm_0:=p(\mes^{\pm})\circ \param~,$$
where $\param_n$ and $\param$ are the parameterization proportional to the arc length of $\arc_n$ and $\arc$.

For convenience, we repeat the proof of Lemma 7.2 in \cite{R60II}.

\begin{lemma}\label{lem:P_n cv f_n}
With the notation above,
$$\int_0^1 f_n^\pm \xrightarrow[n\to\infty]{}  \int_0^1 f_0^\pm~. $$

\end{lemma}
\begin{proof}
Let us denote

$$P_n(\zeta)=s(K_n)\int_0^{1} \ln|\param_n(t)-\zeta|\D t~. $$
For any $0<\beta<1$, 
$$\ln|1+z|\leq \ln (1+|z|)=\frac{1}{\beta}\ln(1+|z|)^\beta\leq \frac{1}{\beta}\ln(1+|z|^\beta)\leq \frac{|z|^\beta}{\beta}~, $$
and using this together with \eqref{eq:inequalit gamma(h)}, we obtain
$$P_n(\zeta_1)-P_n(\zeta_2)\leq \frac{s(K_n)}{\beta}\int_0^1\left|\frac{\zeta_1-\zeta_2 }{\param_n(t)-\zeta_2}\right|^\beta \D t\leq   \frac{2^{1+\beta}s(\arc_n)^{1-\beta}}{\beta(1-\beta)\cos(A/2)^\beta} |\zeta_1-\zeta_2|^\beta~.$$

As the above inequality is symmetric for $\zeta_1$ and $\zeta_2$, it follows that 
the $P_n$ are an equicontinuous family ($s(\arc_n)$ are uniformly bounded by \eqref{aq: rot diam long}). Also, if $\zeta\notin \arc$, then $P_n(\zeta)\to P(\zeta)$. 
By Arzel\`a--Ascoli Theorem, $(P_n)_n$ converge uniformly  to $P$ on any compact set.

Let us apply Fubini Theorem:
$$\int_0^1 f_n^\pm(t)\D t=\frac{1}{s(\arc_n)}\iint P_n(\zeta) \D\mes(\zeta)~, $$
and the result follows by uniform convergence, as $\mes$ is finite and has compact support, and the fact that by Lemma~\ref{lem: BR longueur cv}, $s(K_n)\to s(K)$.
\end{proof}
%
%

\begin{proof}[Proof of Proposition~\ref{prop: cv lg}]
As by Lemma~\ref{lem: BR longueur cv} we have $s(\arc_n)\to s(K)$, it will suffice to check the hypothesis of Theorem~\ref{thm:LVP} for 
 $$F_n=\E^{-( f_n^+-f_n^-)}=\sqrt{\lambda}\circ \param_n~.$$

It is classical that, as $f_n^\pm$ converge almost everywhere to $f_0^\pm$ and as the integrals converge (Lemma~\ref{lem:P_n cv f_n}), then
$f_n^\pm$ converge to $f_0^\pm$ in $L^1$. The $L^1$ convergence implies the convergence in measure, and the composition of a sequence converging in measure by a continuous function still converge in measure, see e.g. \cite[Theorem 2.5.2]{Ash} and \cite{BJ61} respectively. Hence $(F_n)_n$ converges in measures to $F$.

As $\mes^+(K(r))<2\pi$, if $\mes^+(K(r))\not= 0$, we have 
$\frac{\pi}{\mes^+(K(r)}>\frac{1}{2}$, so we can
take $\epsilon>0$ such that
$$\frac{1}{2}<\alpha = \frac{1}{2}+\frac{\epsilon}{2}<\frac{\pi}{\mes^+(K(r))}~, $$
hence we can apply Theorem~\ref{prop:fun estimate}: there exists $B$ such that

$$\int_0^1 \lambda^\alpha \circ \param_n =\int_0^1
F_n^{1+\epsilon}\circ \param_n < B~.
 $$

And finally, $\Phi(y)=y^{1+\epsilon}$ satisfies the requirements of Theorem~\ref{thm:LVP}. If  $\mes^+(K(r))= 0$, the argument is similar.
\end{proof}

An improvement of the proof of Proposition~\ref{prop: cv lg} gives the following.

\begin{theo}
\begin{theorem}[{Basic Lemma I ---Length Convergence Lemma \cite{R60II}}]\label{thm: basic I}
For all $n\in \mathbb{N}^*$, let   $(\mes_n^1)_n$ and $(\mes_n^2)_n$ be  two sequences of positive
 measures   whose supports are contained in a same disc, 
  that converge weakly to $\mes_0^1$ and $\mes_0^2$
respectively.  
  Let us define
  $\mes_n=\mes_n^1-\mes_n^2$, $\mes_0=\mes_0^1-\mes_0^2$,  $\lambda_n=\lambda(\mes_n)$.

  Let $(\arc_n)$, $\arc_0$ be arcs with uniformly  bounded rotation, such that $\arc_n\to \arc_0$ and 
\begin{equation}\label{eq: cond hab prem thm cv long}\mes^1_0(\{z\})<2\pi~, \forall z\in \arc~.\end{equation}
Then
$$\tilde{s}_{\lambda_n}(\arc_n)\xrightarrow[n\to\infty]{}  \tilde{s}_{\lambda_0}(\arc_0)~. $$
\end{theorem}
\end{theo}

The condition about measures convergence in Theorem~\ref{thm: basic I}  is stronger than the weak convergence of $\mes_n$ to $\mes_0$. Also, $\mes^1_0$ and $\mes^2_0$ may not be the positive and negative parts of $\mes_0$, as the following example shows.

\begin{example}{\rm
Let $\mes_n=\mes^1_n-\mes^2_n$ be the  measure such that $\mes^1_n=2\pi\delta_0$, and $\mes^2_n$ is $n$ times the angular measure on $C_{1/n}(0)$ (the total measure is $2\pi$). It is easy to see that $\mes_n$ converges weakly to the zero measure over the plane. But $\mes^1_n$  converge weakly to $\mes^1_0=2\pi\delta_0$, and $\mes^2_n$ converge weakly to $\mes^2_0=2\pi\delta_0$, which are not the positive and the negative part of the zero measure. Even if  $\mes_0=\mes^1_0-\mes_0^2$ is indeed the zero measure,  as \eqref{eq: cond hab prem thm cv long} fails at $0$, we cannot apply Theorem~\ref{thm: basic I} for arcs containing $0$. Indeed, 
for any segment $\arc$ from the origin, 
$\tilde{s}_{\lambda_n}(K)\to +\infty$.
 (See Example~\ref{ex:potentiel cercle} for the potentials  of $\mes^1_n$ and $\mes^2_n$.)
}\end{example}

\subsection{Length distance}\label{sec:intrinsic length}


We want to define a distance from the length $\tilde{s}_\lambda$ defined in \eqref{eq:tilde s}, in the way it is done in the case of Riemannian metrics (see Remark~\ref{remark:def distance riemn}). 

We will give the definition of the distance among a more general framework. The following definitions and main properties can be found e.g., in \cite{R93}, \cite{bbi}, \cite{papadopoulos}.
Let $X$ be a Hausdorff topological space.  A \emph{set of admissible arcs} $\mathcal{A}$ is a subset of the space of arcs of  $X$ which is closed under
restriction and concatenation. In this general framework, we will abuse notation denoting by the same letter an arc and a parameterization of the arc. Given a set of admissible arcs $\mathcal{A}$, a function $\tilde{L}:\mathcal{A}\to [0,+\infty]$ is a \emph{length mapping}\index{length mapping}   if
\begin{itemize}
\item $\tilde{L}$  depends neither on the orientation nor on reparameterization by linear homeomorphism;
\item $\tilde{L}$ is additive: $\tilde{L}(c|_{[a,b]})=\tilde{L}(c|_{[a,t]})+\tilde{L}(c|_{[t,b]})$ for any $t\in[a,b]$;
\item if $\tilde{L}(c)$ is finite, then the function $t\mapsto \tilde{L}(c|_{[a,t]})$ is continuous;
\item \label{loc dist} for any $x\in  X$, there is a neighborhood $U_x$ of $x$ such that 
\begin{equation}\label{eq:loc dist}\inf \{ \tilde{L}(c) : c(a)=x, c(b)\in X\setminus U_x\}>0~.\end{equation}
\end{itemize}

 A \emph{length distance}\index{length distance} on $X$ is the function $d:X\times X \to [0,+\infty]$ defined by
\begin{equation}\label{def:d_L}d(x,y)=\inf\{ \tilde{L}(c) :c \in \mathcal{A}(x,y) \}~, \end{equation}
where $\mathcal{A}(x,y)$ is the subset of $\mathcal{A}$ of elements joining $x$ and $y$. It is rather straightforward that  $d$ is a \emph{distance}\index{distance} on $X$, i.e. a non-negative function on $X  \times X$ such that
\begin{itemize}
\item $d(x,y)=0$ if and only if $x=y$;
\item $d(x,y)=d(y,x)$;
\item $d(x,z)\leq d(x,y)+d(y,z)$.
\end{itemize}

 Note that as the infimum of the empty set is $+\infty$, $d(x,y)=+\infty$ if there is no admissible arc between $x$ and $y$. A point $x$ such that $d(x,y)=+\infty$  for all $y\in X$ is called a  \emph{point at infinity}\index{point at infinity}. 
 Note that however a point at infinity is joined to itself by an arc reduced to a point, that has zero length because of the additivity property, hence we have $d(x,x)=0$ ---this requires to extend a little the definition of arcs we made.
 A distance is \emph{finite}\index{finite distance} if it is a function with values in $[0,\infty)$, i.e., any two points are at finite distance.

We will use at several places the following classical trick.  For a metric space $(X,d)$, we will denote by $B_r(x)$\index{$B_r(x)$} the open ball of radius  $r$ and center $x$.

\begin{lemma}\label{lem: trick discs}
Let  $r>0$ and for $o\in X$, let 
$x,y\in B_{r}(o)$. Then 
$$d(x,y)=\inf\{\tilde{L}(c) : c\in  \mathcal{A}(x,y), c\subset B_{2r}(o) \}~.$$
\end{lemma}
\begin{proof}
We have $d(x,y)\leq 
 d(x,o)+d(o,y)<2r$. Let $c\in \mathcal{A}(x,y)$ 
 and suppose it leaves $B_{2r}(o)$. Then $\tilde{L}(c)\geq 2r$. 
\end{proof}

 These concepts  apply to the case $X=M$ and $\mathcal{A}$ is the set of broken lines  contained in $M$. 
The following result and the properties of the integral imply that $\tilde{s}_\lambda$ is a length mapping.
\begin{theo}
\begin{theorem}[{\cite[Theorem 4.1]{R60II}}]\label{lem:minor lengthBP}
Let $\arc$ be a broken line in $M$ joining $z_1,z_2\in M$, $z_1\not= z_2$.
Then:
$$\tilde{s}_\lambda(\arc) \geq A |z_1-z_2| \exp \left( - \frac{B}{|z_1-z_2|} \right)~, $$
where $A>0$ and $B>0$ depend only on $R$,  $\mes^+(\C)$ and $\mes^-(\C)$. 
\end{theorem}
\end{theo}

 We now know that the length mapping $\tilde{s}_\lambda$ gives a distance.
 \begin{definition}
We denote by  $\rho_\lambda$\index{$\rho_\lambda$} the distance given by the length $\tilde{s}_\lambda$ on the set of broken lines of $M$. A distance obtained in this way is called a \emph{subharmonic distance}\index{subharmonic distance}.
\end{definition}

 One may pay attention to the fact that the distance $\rho_\lambda$ is determined by $\lambda$, hence by a measure $\omega$ and a harmonic function, but it is also determined by the domain $M$, as to define the distance we consider only the broken lines contained in $M$.

  
 A length distance $d$ endows $X$ with a  topology  that is a priori different from the starting one.   Let us denote denote
  $$\mathcal{A}^*=\{c\in\mathcal{A} : \tilde{L}(c)<\infty \}~.$$
  \index{$\mathcal{A}^*$}

 \begin{lemma}\label{lem:comp top struct dist}
 \begin{itemize}
 \item The mapping $i:(X,d)\to X$ is continuous.
 \item An element of $\mathcal{A}^*$ is an arc in $(X,d)$.
 \item The topologies of $X$ and $(X,d)$ coincide if and only if $\forall x\in X$, $\forall\epsilon>0$, there exists an open set $U_{x,\epsilon}\ni x$ of $X$ such that each point of $U_{x,\epsilon}$ is connected to $x$ by an admissible arc of length $<\epsilon$.
 \end{itemize}
 \end{lemma}

\begin{proof}
Let $U$ be an open set in $X$ with $x\in U$ and $r_x>0$ be the number given by 
\eqref{eq:loc dist}. Hence $B_{r_x/2}(x)\subset U$ and $U$ is open in $(X,d)$. For the second point, let $c:[a,b]\to X$ be an arc with $\tilde{L}(c)<\infty$. We have to check that it is continuous for $(X,d)$. But for any $a\leq t<t'\leq b$,
$$d(c(t),c(t'))\leq \tilde{L}(c_{|[t,t']})=\tilde{L}(c_{|[a,t']})-\tilde{L}(c_{|[a,t]})  $$
 and this last quantity goes to $0$ when $t\to t_0$ by the definition of $\tilde{L}$.
  
   For the last point, the hypothesis says that $U_{x,\epsilon}\subset B_\epsilon(x)$, and it is easy to deduce that $B_\epsilon(x)$ is open. The conclusion follows from the first point.
\end{proof}

\begin{example}\label{ex:cylindre 1}{\rm
We have seen flat cones in Section~\ref{Sec: flat cones}.
Let us consider the case when 
the cone angle at the origin is non-positive. The length of any broken line 
passing through the origin is infinite. These Euclidean arcs are not arcs for the topology given by $\rho_\lambda$. 
In turn, the set of arcs for the topology induced by $\rho_\lambda$ joining the origin to any other point is empty, so the origin is a point at infinity. And of course the topology induced by $\rho_\lambda$ is not the Euclidean one: for any sequence with $z_n\not= 0$ and  $\vert z_n\vert\to 0$, we have  $\rho_\lambda(z_n,0)=\infty$.
 }\end{example}


Let us first note an immediate consequence of the estimate from Theorem~\ref{prop:fun estimate} applied on segments.

\begin{lemma}\label{proposition: tame point topologies coincide}
Let $z\in M$ be such that $\mes^+(\{z\}) <2\pi$. There exists $r$ such that $Q_r:=Q_r(z)\subset M$ and
\begin{itemize}
\item the Euclidean topology and the topology induced by $\rho_\lambda$ coincide over $Q_r$,
\item $\rho_\lambda$ is finite over $Q_r$.
\end{itemize}
\end{lemma}
\begin{proof}
As the measure $\mes^+$ is regular, then there is $r$ such that $\mes^{+}(Q_r)<2\pi$. Let $\epsilon>0$ and $z'\in Q_r$, and let $U_d$ be an open Euclidean disc centered at $z'$ of radius $d$. 
If $d$ is small enough, we also have $\mes^{+}(U_d)<2\pi$. Let $z''\in U_d$.  For a segment $\arc$ joining 
$z'$ to $z''$, $s(K)=|z-z'|<d$ and $|\kappa|(\arc)=0$. Let $\bar{
r}$ be such that $\arc(\bar{r})\subset U_d$.
Theorem~\ref{prop:fun estimate} and the definition of $\rho_\lambda$ say that 
\begin{equation}\label{eq:longueur segment}\rho_\lambda(z',z'')\leq \tilde{s}_\lambda(\arc)\leq A_1\frac{d^{1-\frac{1}{2\pi}\mes^+(K(\bar{r}))}}{\bar{r}^b}~, \end{equation}
where $A_1$ and $b$ depend only on $\mes$. So if $d$ is small enough, every point in $U_d$ is connected to $z'$ by a broken line of length 
$<\epsilon$, and by Lemma~\ref{lem:comp top struct dist}, both topologies coincide over $Q_r$.
By \eqref{eq:longueur segment} again, the distance from $z$ to any other point in $Q_r$ is finite.
The distance $\rho_\lambda$ is then finite over $Q_r$ by the triangle inequality.
Here we have considered only the case  $\lambda:=\lambda(\mes)$. For $\lambda=\lambda(\mes,h)$, as $h$ 
 is continuous over $M$ and $K\subset M$ is compact, an analogous result occurs.
\end{proof}

Let us return to a more general framework. We have a Hausdorff topological space $X$, a set of admissible arcs $\mathcal{A}$ and a length structure $\tilde{L}$. 
Those elements gave a distance $d$ over  $X$. 
Now from the distance $d$, we can define   another length structure.  Let  $c$ be an arc of $X$. The \emph{intrinsic length}\index{length (intrinsic)} of $c$ is
$$L(c)=\bigvee_a^b c~.$$ 
 Here of course the total variation is computed with the distance $d$, and is independent of the choice of the parameterization. More precisely, if $c:[a,b]\to X$, then $L(c)$ is the supremum of $\sum d(c(t_i),c(t_{i+1}))$ over all the  decompositions of $[a,b]$.
 Obviously from the definition, for any arc $c$ from $x$ to $y$,
 \begin{equation}\label{eq:dis plus petite long}
 d(x,y)\leq L(c)~.
 \end{equation}
 
 By properties of the total variation, it follows that $L$ endows $X$ with another length structure. 
The arc  $c$ is \emph{rectifiable}\index{rectifiable (for $d$)} if it has a parameterization with bounded variation, i.e., $L(c)$ is finite. With this notion of rectifiablity in hands, we can define \emph{parameterization by arc length}\index{parameterization by arc length} and \emph{parameterization proportional to the arc length}\index{parameterization proportional to the arc length} for a rectifiable arc, in a similar way than for curves in the Euclidean plane, see e.g. \cite{papadopoulos} for details.

So now, we have $\tilde{L}$ and $L$ to compute length of arcs. We always have the following relation.

 \begin{lemma}\label{lem:L Lt A}
For $c\in \mathcal{A}^*$, $L(c)\leq \tilde{L}(c)$.
 \end{lemma}
\begin{proof}
 Let us denote by $\mathcal{A}_i$ the elements of $\mathcal{A}^*$ from $c(t_i)$ to $c(t_{i+1})$. It is not empty as $c\in \mathcal{A}^*$. We have
  $$\sum_i d(c(t_i),c(t_{i+1}))=\sum_i \inf_{c_i\in \mathcal{A}_i} \tilde{L}(c_i)\leq 
  \sum_i \tilde{L}(c_{|[t_i,t_{i+1}]}) =\tilde{L}(c) $$
  and taking the supremum over the partitions of $[a,b]$ of the left-hand side of this inequality leads to the result.
\end{proof}

A length structure $(\mathcal{A}, \tilde{L})$ is \emph{lower semicontinuous} when, if $c_n\to c$ in $\mathcal{A}$,  then 
$$\tilde{L}(c)\leq \liminf_n \tilde{L}(c_n)~. $$

We refer to Theorem~2.4.3 in \cite{bbi} for a proof of the following result.

\begin{lemma}\label{prop:= length admissible}
If $\tilde{L}$ is lower semicontinuous, then for $c\in \mathcal{A}^*$, $\tilde{L}(c)=L(c)$.
\end{lemma}

The length structure $L$ induces itself a distance, but fortunately it can be proved that it is   nothing but $d$.

A metric space is said to be \emph{intrinsic}\index{intrinsic distance}\footnote{Intrinsic metric space may also be called \emph{inner}\index{inner metric space} or \emph{internal}\index{internal metric space}.} if the distance between two points is the infimum of the intrinsic length of \emph{all} the arcs  joining the two points. 
Now if $d$ is defined by a lower semicontinuous length mapping, 
from Proposition~\ref{prop:= length admissible} and \eqref{eq:dis plus petite long}, it follows that $d$ is intrinsic. 

Let us come back to our length structure $\tilde{s}_\lambda$ over $M$. We denote by $s_\lambda$\index{$s_\lambda$} the intrinsic length defined from the distance $\rho_\lambda$.
Let us denote by $\mathcal{B}$ the set of
arcs, rectifiable for the Euclidean metric, and contained in $M$. Let us begin with an easy lemma. 

\begin{lemma}\label{lem: Iomega lsc}
The function $\tilde{s}_\lambda:\mathcal{B}\to \R$ is lower semicontinuous.
\end{lemma}
\begin{proof}
Let $\arc_n\to \arc$, where the arcs are parameterized proportionally to  arc length. On the one-hand, the Euclidean length is lower semicontinuous:
$$s(\arc)\leq \liminf_n s(\arc_n)~, $$
and on the other hand, by Fatou lemma
$$\int_0^1 \sqrt{\lambda}\circ z \leq \liminf_n  \int_0^1 \sqrt{\lambda}\circ z_n~.$$
~\end{proof}

In particular, as broken lines are rectifiable, it follows that for broken lines with finite  length  $\tilde{s}_\lambda$, $\tilde{s}_\lambda=s_\lambda$, and in turn the distance $\rho_\lambda$ is intrinsic. 

One may ask if the choice of a set of admissible arcs other than broken lines would lead to another distance, or if the length and the intrinsic length may differ for some arcs. 
The convergence result of the preceding section give some information,  when the measure it not too concentrated at some points and when the arc is of bounded rotation.

\begin{lemma}\label{lem metriques coincident}
If every point $z\in M$ satisfies $\mes^+(\{z\})<2\pi$,
then for $z_1,z_2\in M$, $\rho_\lambda(z_1,z_2)$ is the infimum of $\tilde{s}_\lambda$ over the arcs of bounded rotation from $z_1$ to $z_2$ contained in $M$.
\end{lemma}
\begin{proof} Let us
 denote by $\rho'_\lambda(z_1,z_2)$ the infimum for $\tilde{s}_\lambda$ over the  arcs of bounded rotation from $z_1$ to $z_2$.  For any broken line $L$ from $z_1$ to $z_2$, as  broken lines are arcs of bounded rotation:
$$\rho'_\lambda(z_1,z_2)\leq \rho_\lambda(z_1,z_2) \leq \tilde{s}_{\lambda}(L)~.$$
Hence $\rho'_\lambda(z_1,z_2)$ is a minorant of the set of $\tilde{s}_\lambda(L)$ for all the broken lines from $z_1$ to $z_2$. Now let $\epsilon>0$ and let  $\arc$ be such that 
$$\tilde{s}_\lambda(K)\leq \rho'_\lambda(z_1,z_2)+\epsilon­­~.$$
By definition, there exists a broken line $L$ arbitrarily close to $\arc$ and with absolute rotation arbitrarily close to the one of $\arc$, so by  Proposition~\ref{prop: cv lg} one can find such a $L$ with
$$\tilde{s}_\lambda(L)\leq \tilde{s}_\lambda(K)+\epsilon~.$$
Putting the equations together, the minorant is the infimum and we are done.
 \end{proof}
 
\begin{remark}{\rm
The argument in the proof of Lemma~\ref{lem metriques coincident} is valid for any set of arcs of bounded rotation that contains the set of broken lines, when for every point $z\in M$ we have  $\mes^+(\{z\})<2\pi$. In particular, we obtain the same distance considering only e.g., piecewise $C^\infty$ arcs.
}\end{remark}  
 
 Lemma~\ref{prop:= length admissible}, Lemma~\ref{lem: Iomega lsc} and Lemma~\ref{lem metriques coincident} give the following.
\begin{lemma}\label{lem:long coincide}
If every point $z\in M$ satisfies $\mes^+(\{z\})<2\pi$,
then  for any arc $\arc$ of bounded rotation and with $\tilde{s}_\lambda(\arc)<\infty$, then  $\tilde{s}_\lambda(\arc)=s_\lambda(\arc)$.
\end{lemma} 

 \begin{remark}{\rm
 As a byproduct of the Distances Converging Theorem~\ref{thm: distances convergence theorem}, a more general result than Lemma~\ref{lem:long coincide} holds, see Theorem~\ref{thm length coincide}.
 }\end{remark}
 
 For an intrinsic distance, the distance between two points is the infimum of the lengths of the arcs joining these points. A question remains, to know if this infimum is reached. In this case, we have obtained a shortest arc. 
In other terms, an arc $c$ between two points at finite distance in the metric space $(X,d)$ is a  \emph{shortest arc}\index{shortest arc} if its intrinsic length is equal to the distance between its extremities.
 
%

\begin{lemma}\label{lem existence shortest arcs}
Let $z\in M$ with $\mes^+(\{z\})<2\pi$. Then there 
is a neighborhood of $z$ onto which any pair of points is joined by a shortest arc.
\end{lemma}
\begin{proof}
Let $r>0$ such that $B_{r}(z)\subset M$, $\rho_\lambda$ is finite over $B_{r}(z)$ and the induced topology is the Euclidean topology, see Lemma~\ref{proposition: tame point topologies coincide}.
Let $z_1,z_2\in B_{r/2}(z)$.

 Let us take a sequence of broken lines from $z_1$ to $z_2$ whose lengths are converging to $\rho_\lambda(z_1,z_2)$. 
By Lemma~\ref{lem: trick discs}, we can consider only the arcs from $z_1$ to $z_2$ contained in $B_{r}(z)$, and of length $<r$ by the triangle inequality. 
 If we consider parameterizations proportional to   arc length for the distance $\rho_\lambda$,  as the lengths are uniformly bounded from above,  these parameterizations are equi-Lipschitz, and by Arzelà–Ascoli theorem, there is a converging subsequence.
\end{proof}

So far, $M$ was a domain of the plane. But the distance $\rho_\lambda$ may also be defined over a closed set ---it is straightforward if we take the zero function as a harmonic function over $M$.

The following statement is the one of the improvement given by Lemma 6.3 in \cite{R60I}. Its proof is contained in \cite{R60II}.
\begin{theo}
\begin{theorem}[{Basic Lemma II ---Estimate of Shortest Arcs Rotation  \cite{R60II}}]\label{them:basic II}
For all $n\in \mathbb{N}^*$, let   $(\mes_n^1)_n$ and $(\mes_n^2)_n$ be  two sequences of positive
 measures  with $C^1$ Lebesgue densities, whose supports are contained in a same disc. Suppose that those sequences
 converge weakly to $\mes_0^1$ and $\mes_0^2$
respectively.  
  Let us define
  $\mes_n=\mes_n^1-\mes_n^2$, $\mes_0=\mes_0^1-\mes_0^2$,  $\lambda_n=\lambda(\mes_n)$.
  
Let $M$ be the closure of a bounded domain of the plane, whose boundary is a union of a finite number of simple closed curves of bounded rotation.

Let $F$ be a closed subset of $M$ such that
\begin{equation}\label{eq:intermediate cv thm}(\mes_0^1+\mes_0^2)(\{z\})<2\pi,\, \forall z\in F~.\end{equation}

  Then there exists $A<\infty$ such that for any  shortest arc $\arc\subset F$ for $\rho_{\lambda_n}$, for any $n$, we have $|\kappa|(K)<A$.
\end{theorem}
\end{theo}

Recall that the measure $\mes_0$ may be decomposed as
$$\mes_0=\mes_0^1 -\mes_0^2=\mes_0^+ - \mes_0^-  $$
and that $\mes_0^+\leq \mes_0^1$. Recall also Remark~\ref{rem:def weak cv mes} about weak convergence. The main properties of weak convergence that are needed in the sequel are (see e.g., \cite{Ash}, \cite{mattila}): 
\begin{itemize}
\item for any compact set $C$, $\mes_0^i (C)\geq \limsup_n \mes_n^i(C)$;
\item for any open set $O$, $\mes_0^i (O)\leq \liminf_n \mes_n^i(O)$.
\end{itemize}

Note that we know that shortest arcs exist by Lemma~\ref{lem existence shortest arcs}.  
The proof of Theorem~\ref{them:basic II} is mainly based on properties of arcs of bounded rotations stated in Section­~\ref{sec abs rot}.
As the densities are $C^1$, then $\lambda$ are $C^2$ (see Section 3.5 in \cite{jost} for precise results), 
 hence the metrics $\lambda|\D z|^2$ are Riemannian. One of the step for the proof of Theorem~\ref{them:basic II} is to note that if $M$ is endowed with a Riemannian conformal metric, then any shortest arc has bounded rotation. This is obvious if the arc is in the interior of $M$, but the arc may contain parts of the boundary of $M$, which is made of general arcs of bounded rotation. This is the content of Lemma~9.1 and its Corollary in \cite{R60II}. Basically, the necessity of condition \eqref{eq:intermediate cv thm} in Theorem~\ref{them:basic II} comes from \eqref{eq: rot shortest riem}, see Theorem~9.1 in   \cite{R60II} for details.

We finally obtain:

\begin{theorem}[Intermediate Distances Convergence Theorem]\label{thm: intermediate convergence}
For all $n\in \mathbb{N}^*$, let   $(\mes_n^1)_n$ and $(\mes_n^2)_n$ be  two sequences of positive
 measures,   whose supports are contained in a same disc, 
  that converge weakly to $\mes_0^1$ and $\mes_0^2$
respectively.  
  Let us define
  $\mes_n=\mes_n^1-\mes_n^2$, $\mes_0=\mes_0^1-\mes_0^2$,  $\lambda_n=\lambda(\mes_n)$.
  
  Let $M$ be the closure of a bounded domain of the plane, whose boundary is a union of a finite number of simple closed curves of bounded rotation.
  
  Suppose that  \begin{equation}\label{eq:somme conditionb}(\mes^1_0+\mes_0^2)(\{z\})<2\pi~,\forall z\in M~.\end{equation}

Then $(\rho_{\lambda_n})_n$  converge uniformly to $\rho_{\lambda_0}$.
\end{theorem}

We already have a theorem about convergence of length, Theorem~\ref{thm: basic I}. This is the first main tool for proving Theorem~\ref{thm: intermediate convergence}. Using classical manipulations (see the paragraphs before Lemma~6.1 in \cite{R60II}, Theorem~\ref{thm: basic I} gives that if  $z_n\to z_0$ and $\zeta_n\to \zeta_0$ in $M$, then
\begin{equation}\label{cor:lumisup distances}\rho_{\lambda_0}(z_0,\zeta_0)\geq \limsup_n \rho_{\lambda_n}(z_n,\zeta_n)~. \end{equation}

By smooth approximation of measures and a classical 
diagonal argument \cite[Lemma 5.2]{R60II}, 
the proof of Theorem~\ref{thm: intermediate convergence} reduces to the case when the measures $\mes_n$, $n>0$, have $C^1$ Lebesgue densities. In this case, it is rather straightforward to deduce from Theorem~\ref{thm: basic I} and Theorem~\ref{them:basic II}  (see the paragraph next to the proof of Lemma~6.1 in \cite{R60I}) that if  $z_n\to z_0$ and $\zeta_n\to \zeta_0$ in $M$, then
\begin{equation}\label{cor:lumiinf distances}\rho_{\lambda_0}(z_0,\zeta_0)\leq \liminf_n \rho_{\lambda_n}(z_n,\zeta_n)~. \end{equation}
This together with \eqref{cor:lumisup distances} gives the uniform convergence of the distances. Indeed, recall that a sequence of distances $d_n$ on a space $X$ is said to \emph{converges uniformly}\index{uniform convergence (distances)} to a distance $d$ if it converges uniformly as functions 
on $X\times X$:
$$\sup_{x,x'\in X}\vert d_n(x,x') - d(x,x')\vert \xrightarrow[n\to \infty]{} 0~.$$
As \eqref{eq:somme conditionb} 
 implies that for all $z\in M$, $\mes^1_0(\{z\})
<2\pi$, we have that  $\rho_{\lambda_0}$ is continuous for the Euclidean topology by Proposition~\ref{proposition: tame point topologies coincide}. Hence  the sequential characterization of uniform continuity gives an equivalent definition: when $z_n\to z_0$ and $\zeta_n\to\zeta_0$ in $M$, then
we require that $\rho_{\lambda_n}(z_n,\zeta_n)$ converge to 
  $\rho_{\lambda_0}(z_0,\zeta_0)$. But it is exactly what is given by \eqref{cor:lumiinf distances} and \eqref{cor:lumisup distances}.
 
The main result of the theory, Theorem~\ref{thm: distances convergence theorem}, is an improvement of Theorem~\ref{thm: intermediate convergence}, whose proof is based on the concept of stretching of a distance, that we present in the next section.

\subsection{Canonical stretching, points at infinity}\label{sec can stretch}

Canonical stretching allows to compare the local behavior of a distance with the distance of a flat cone. It is the main  tool used to deal with the case when points have a measure $\geq 2\pi$. 
 Although it is the byproduct of the Intermediate Distances Convergence Theorem~\ref{thm: intermediate convergence}, the Localization Theorem~\ref{thm:localization}, and an affine change of variable, we have preferred to make a comprehensive presentation.

Let $\lambda=\lambda(\mes,h)$. For $r>0$, let us look at the transformation of the plane given by $A_{z_0,r}(z)=z_0+rz$. This sends the circle $C_h(0)$ to the circle 
$C_{rh}(z_0)$. Also,  for any rectifiable arc $\arc$,
$$\int_\arc (r^2\lambda\circ A_{z_0,r})^{1/2} =  \int_\arc \lambda^{1/2}\circ A_{z_0,r}  |DA_{z_0,r}| = \int_{A_{z_0,r}\circ \arc} \lambda^{1/2}~.$$
In turn, the function $A_{z_0,r}$  preserves  the distance from  $Q_h(0)$ with the distance $\rho_{\tilde{\lambda}_{r,z_0}}$, where
$\tilde{\lambda}_{r,z_0}:=r^2\lambda\circ A_{z_0,r}$, to $Q_{rh}(z_0)$ with the distance $\rho_\lambda$. Let 
 $$c(r):=\frac{2\pi}{\tilde{s}_\lambda(C_r(z_0))}$$ and define 
$\lambda_{r,z_0}:=c(r)\tilde{\lambda}_{r,z_0}$. The function $A_{z_0,r}$ is a similarity  
from $Q_h(0)$ with the distance $\rho_{\lambda_{r,z_0}}$ to $Q_{rh}(z_0)$ with the distance  $\rho_\lambda$, with similarity coefficient equal to $c(r)$. Moreover,
\begin{equation}\label{eq:def cr}\tilde{s}_{\lambda_{r,z_0}}(C_1(0))=2\pi~. \end{equation}

To sum up, for $(r,z_0)\in ]0,1[\times M$, the \emph{canonical stretching}\index{canonical stretching} (with factor $(r,z_0)$) of the distance $\rho_\lambda$
is the distance $\rho_{\lambda_{r,z_0}}$\index{$\rho_{\lambda_{r,z_0}}$}, with
\begin{equation}\label{eq: can stre def}\lambda_{r,z_0}(z)=c(r)r^2\lambda(z_0+rz)~. \end{equation}

Note that, if $\mes^+(\{z_0\})<2\pi$, by Theorem~\ref{prop:fun estimate}, 
$c(r)\to\infty$ when $r\to 0$.
\begin{theo}
\begin{theorem}[{Canonical Stretching \cite[Theorem 11.1]{R60II}, \cite[Lemma 11]{R63III}}]\label{thm:canoncial streching}
Let $z_0\in \C$ and let $\delta_{z_0}$ be the Dirac measure concentrated at $z_0$ and let $\mes_0=\mes(\{z_0\})\delta_{z_0}$. 
\begin{itemize}
\item Let us define $\rho_{\lambda_{r,z_0}}$ over  
$R_{\epsilon,1}=\{z\in \mathbb{C} : 0<\epsilon \leq |z| \leq 1\}$. Then when $r\to 0$, the distances $\rho_{\lambda_{r,z_0}}$  converge uniformly to $\rho_{\lambda(\mes_0)} $ over $R_{\epsilon,1}$. 
\item Let us define $\rho_{\lambda_{r,z_0}}$ over  
$Q_1(0)$. If $\mes(\{z_0\})<2\pi$, then when $r\to 0$, the distances $\rho_{\lambda_{r,z_0}}$  converge uniformly to $\rho_{\lambda(\mes_0)} $ over $Q_1(0)$.
\end{itemize}
\end{theorem}
\end{theo}

\begin{proof}
Let us prove the first point. The proof of the second point is similar. By the Localization Theorem~\ref{thm:localization}, we know that there exists a signed measure $\tilde{\mes}_r$ with compact support such that
\begin{itemize}
\item $\vert\tilde{\mes}_r\vert(Q_r(z_0)\setminus \{z_0\})\xrightarrow[r\to 0]{} 0$
\item on $Q_r(z_0)$, there is a constant $k_r$ such that 
$p(\mes)=p(\tilde{\mes}_r)+k_r$.
\end{itemize}

Let us define $\mes_r=(A_{z_0,r}^{-1})_*(\tilde{\mes}_r)$, i.e., for any Borel set $E$, $\mes_r(E)=\tilde{\mes}_r(rE+z_0)$.
We have  $\lambda=\E^{-2(p_\mes+h)}$, and   for $z\in Q_r(z_0)$, there is a constant $\alpha(r)$ such that
$$\ln \lambda(z)=-\frac{1}{\pi}\iint \ln|z-\zeta |\D\tilde{\mes}_r(\zeta) + \alpha(r)
~. $$

By a change of variable, there is a constant $\beta(r)$ such that 
\begin{equation*}
\begin{split}
 \ln \lambda(z_0+rz)&= -\frac{1}{\pi}\iint \ln\left|r\left(z-\frac{\zeta-z_0}{r} \right)\right|\D\tilde{\mes}_r(\zeta) + \alpha(r)\\
 &=-\frac{1}{\pi}\iint \ln |z-\zeta| \D\mes_r(\zeta) + \beta(r)~. 
 \end{split} 
\end{equation*}

So there is a positive constant $A(r)$ such that for $z\in Q_1(0)$,

\begin{equation}\label{eq stretch fin} \lambda_{r,z_0}(z) = A(r)^2\E^{-2 p_{\mes_r}(z)}~. \end{equation}

 It is rather clear that the positive and negative parts of $\mes_r$  converge weakly to the positive and the negative parts of $\mes_0$ respectively. Also, over $R_{\epsilon,1}$, the condition \eqref{eq:intermediate cv thm}
is trivially satisfied by the measure $\mes_0$. Hence by  Theorem~\ref{thm: intermediate convergence}, the distances $\rho_{\lambda(\mes_r)}$  converge uniformly to 
$\rho_{\lambda(\mes_0)}$ over $R_{\epsilon,1}$. 

On the other hand, $\rho_{\lambda_{r,z_0}}=A(r)\rho_{\lambda(\mes_r)}$. By Lengths Convergence Lemma Theorem~\ref{thm: basic I},  $$\tilde{s}_{\lambda(\mes_r)}(C_1(0)) \xrightarrow[r\to 0]{} \tilde{s}_{\lambda(\mes_0)}(C_1(0))~$$
and by Remark~\ref{rem lon cercle}, $\tilde{s}_{\lambda(\mes_0)}(C_1(0))=2\pi$.
But by \eqref{eq:def cr} and  \eqref{eq stretch fin},
 $2\pi=A(r)\tilde{s}_{\lambda(\mes_r)}(C_1(0))$. Hence $A(r)\to 1$ when $r\to 0$, so $\rho_{\lambda_{r,z_0}}$ uniformly converge to $\rho_{\lambda(\mes_0)}$ over  $R_{\epsilon,1}$.
\end{proof}

Let us present direct applications of the canonical stretching. The detail of the proofs are in paragraphs 12 and 13 of \cite{R60II}. 
Recall that we have encounter a \emph{point at infinity}\index{point at infinity} in Example~\ref{ex:cylindre 1}, that is,  a point $z\in M$ such that for any point $z'\in M$, $\rho_\lambda(z,z')=+\infty$. 
\begin{theo}
\begin{theorem}[{\cite[Theorem 12.1,Theorem 12.2, Theorem 13.2]{R60II}}]\label{thm:fintie points}

Let us denote
$\sigma_r=\tilde{s}_\lambda (C_r(z)).$
\begin{itemize}
\item If $\sigma_r \xrightarrow[r\to 0]{} l\in]0,+\infty]$, then 
$z$ is a point at infinity.
\item If $\mes(\{z\})>2\pi$, then $\sigma_r \xrightarrow[r\to 0]{} \infty$. In particular, $z$ is a point at infinity.
\item If $z$ is not a point at infinity, then the diameter of $Q_r(z)$ for the distance $\rho_\lambda$ goes to zero when $r\to 0$.
\end{itemize}
\end{theorem}
\end{theo}

Let us call a \emph{finite point}\index{finite point} a point $z\in M$ which is at finite distance from all the other points in a neighborhood. We saw in Lemma~\ref{proposition: tame point topologies coincide} that if $\mes(\{z\})<2\pi$, then 
$z$ is a finite point. 

\begin{example}{\rm
If $\mes(\{z\})=2\pi$, Example~\ref{ex:cylindre 1} is a simple example where $z$ is a point at infinity. A more involved example (Example 2 in \cite{R60II}) shows that there exists a subharmonic  distance  with $\omega(\{z\})=2\pi$ and  a segment starting at $z$ with finite $s_\lambda$.
 Hence $z$ is not a point at infinity, and, as it will be noted later, it follows that $z$ is a finite point. 

A simple example of these different behaviors is in \cite[2.2]{HT}. 
} \end{example}
 
 So nothing can be said a priori in the case $\mes(\{z\})=2\pi$,
  except that  only two cases can occur. Indeed, 
a byproduct of Theorem 13.1 in \cite{R60II}  is that for any subharmonic distance, a point which is not at infinity is a finite point.
Behavior of curves may be wild when points with $\mes(\{z\})=2\pi$ enters the picture, see \cite{AZ}, \cite{Zal}.

Let us note that things are easier in the case of positive curvature.

\begin{lemma}\label{lem:mes pos point infini}
Suppose that 
$\mes$ is a finite positive measure. A point $z_0\in M$ with
$\mes(\{z_0\})\geq 2\pi$ is a point at infinity.
\end{lemma}
\begin{proof}
The case $>2\pi$ is general and comes from Theorem~\ref{thm:fintie points}.
From \eqref{eq:decomp mes lambda point} and Lemma~\ref{lem:potential bfb},
there is a positive constant $C$ such that, in a neighborhood of $z_0$,

$$\lambda^{1/2}(z)=\vert z-z_0\vert^{-1}\E^{-\frac{1}{2\pi}\int_{M\setminus \{z_0\}}\ln \vert z-\zeta\vert \D\mes(\zeta)-h(z)}\geq \vert z-z_0\vert^{-1} C~,$$
so for any arc $\arc$ with extremity $z_0$, $\tilde{s}_\lambda(\arc)=+\infty$.
\end{proof}

Eventually, the following improvement of Lemma~\ref{proposition: tame point topologies coincide} is obtained. Note that one direction is a general fact, see Lemma~\ref{lem:comp top struct dist}.

\begin{theo}
\begin{theorem}[{\cite[Theorem 13.3]{R60II}}]\label{thm:met coinc}
Any finite point has  a  neighborhood over which  the topology induced by  $\rho_\lambda$ coincides with the Euclidean topology.
\end{theorem}
\end{theo}

\begin{example}[A distance which is not subharmonic]{\rm
This example is from p. 24 in \cite{AZ}. Le us consider a 
flat cylinder with boundary on one side. Let us consider the metric space obtained from it by identifying the boundary to a single point $O$. Let us consider the distance on this space which is the smallest number between the original distance and the sum of the distances of the points to $O$. By definition, it is an intrinsic metric space. Moreover, it is easy to see that in a small neighborhood of a point different from $O$, this new distance is the distance of the original cylinder. It follows that a neighborhood of $O$ is homeomorphic to a plane domain. Let us suppose that this is a subharmonic distance.   Clearly, $O$ is not a point at infinity, that contradicts Theorem~\ref{thm:fintie points}, as the lengths of the circles centered at $O$ are bounded from below by a positive constant.

Another way to consider this example, is to note that the plane minus a disc with non-empty interior is not conformally equivalent to the plane minus a point, see \cite{tro-ouvert}.
}\end{example}

\subsection{Distances Convergence Theorem and consequences}

\label{sec cv}

\subsubsection{Convergence of distances}

The first consequence of the results of Section~\ref{sec can stretch}
is the following enhancement of the Intermediate Distances Convergence Theorem~\ref{thm: intermediate convergence}. Its proof may be found  after the statement of Lemma~C in \cite{R60I}.\footnote{This theorem is called \emph{Metrics Convergence Theorem} in \res articles, as the definitions do not fit with ours, see the introduction.} 

 \begin{theo}
 \begin{theorem}[{Distances Convergence Theorem, Theorem III in \cite{R60I}}]\label{thm: distances convergence theorem}
For all $n\in \mathbb{N}^*$, let   $(\mes_n^1)_n$ and $(\mes_n^2)_n$ be  two sequences of positive
 measures,   whose supports are contained in a same disc, 
  that converge weakly to $\mes_0^1$ and $\mes_0^2$
respectively.  
  Let us define
  $\mes_n=\mes_n^1-\mes_n^2$, $\mes_0=\mes_0^1-\mes_0^2$,  $\lambda_n=\lambda(\mes_n)$.
  
  Let $M$ be the closure of a bounded domain of the plane, whose boundary is a union of a finite number of simple closed curve of bounded rotation.
  
  Suppose  that 
\begin{equation}\label{eq:toujours la meme}\mes_0^1(\{z\})<2\pi~, \forall z\in M~.\end{equation}

Then $(\rho_{\lambda_n})_n$  converge uniformly to $\rho_{\lambda_0}$. 
\end{theorem}
\end{theo}


As for the Intermediate Distances Convergence Theorem~\ref{thm: intermediate convergence}, to prove  Theorem~\ref{thm: distances convergence theorem} one proves that when $z_n\to z_0$ and $\zeta_n\to\zeta_0$ in $M$, then
$\rho_{\lambda_n}(z_n,\zeta_n)$ converge to 
  $\rho_{\lambda_0}(z_0,\zeta_0)$. In the second part of the proof of  Theorem~\ref{thm: distances convergence theorem} in \cite{R60II}, \res checks that 
  it suffices to prove the theorem under the additional assumption that $(\mes^1_0+\mes_0^2)(\{z_0\})<2\pi$ and 
$(\mes^1_0+\mes_0^2)(\{\zeta_0\})<2\pi$. Then the Intermediate Distances Convergence Theorem~\ref{thm: intermediate convergence} can be used, in a careful way around the points for which  $(\mes^1_0+\mes_0^2)\geq 2\pi$. The behavior  of shortest arcs in the neighborhood of such points is known from the results of Section~\ref{sec can stretch}.

The next statement gives a criterion to relax 
\eqref{eq:toujours la meme}. 
\begin{theo}
\begin{theorem}[{\cite[Proof of Lemma 7.1]{R60I}}]\label{prop:conv increas}
For all $n\in \mathbb{N}^*$, let   $(\mes_n^1)_n$ and $(\mes_n^2)_n$ be  two sequences of positive
 measures,   whose supports are contained in a same disc, 
  that converge weakly to $\mes_0^1$ and $\mes_0^2$
respectively.  
  Let us define
  $\mes_n=\mes_n^1-\mes_n^2$, $\mes_0=\mes_0^1-\mes_0^2$,  $\lambda_n=\lambda(\mes_n)$.
  
  Let $M$ be the closure of a bounded domain of the plane, whose boundary is a union of a finite number of simple closed curves of bounded rotation.
  
  Suppose that
\begin{equation}
\label{eq:remp croiss tjs mm}
\lambda_n\leq \lambda_0~,\end{equation}
then $\rho_{\lambda_n}$ converge to
$\rho_{\lambda_0}$ uniformly over any closed set, contained in  the interior of $M$, on which $\rho_{\lambda_0}$ is finite.
\end{theorem}
\end{theo}

 Basically, \eqref{eq:remp croiss tjs mm} allows to control the radii of discs around  points with measure $\geq 2\pi$. This is why the proof of Theorem~\ref{prop:conv increas} cannot be done for points on the boundary of $M$. For the  details of the proof, see the second part of the proof of Lemma~7.1 in \cite{R60I}. In the first part of Lemma~7.1 in \cite{R60I}  is proved the following statement, that says that subharmonic distances are limits of Riemannian  distances.

\begin{theo}
\begin{theorem}[{\cite[Lemma 7.1]{R60I} }]\label{thm:approx lisse met}
Let $M$ be the closure of a bounded domain of the plane, whose boundary is a union of a finite number of simple closed curves of bounded rotation.
Let $\mes$ be a signed measure with compact support.

There exists a sequence of $C^\infty$ Lebesgue densities $\mes_n$ such that $\mes_n^+\to\mes^+$ and $\mes_n^-\to\mes^-$ weakly, the $C^\infty$ functions $\lambda(\mes_n)$ converge pointwise to $\lambda(\mes)$, and
$\rho_{\lambda(\mes_n)}$ converge to
$\rho_{\lambda(\mes)}$ uniformly over any closed set contained in the interior of $M$ on which $\rho_{\lambda(\mes)}$ is finite.
\end{theorem}
\end{theo}

Theorem~\ref{thm:approx lisse met} is simply obtained 
by approximation of  
$\mes^+$ and $\mes^-$ by the sequences given by Proposition~\ref{prop:approx potentiel}. 
More precisely, we  approximate $\mes^+$ by smooth Lebesgue densities $\mes^+_n$ using Proposition~\ref{prop:approx potentiel}. Its conclusion allows to use Theorem~\ref{prop:conv increas}, so that $\rho_{\lambda(\mes_n^+-\mes^-)}$ converge to
$\rho_{\lambda(\mes)}$ uniformly over any closed set contained in the interior of $M$ on which $\rho_{\lambda(\mes)}$ is finite. Now for any $n$, $(\mes_n^+-\mes^-)(\{z\})<2\pi$ for any $z$, so 
we can take a smooth approximation of $\mes^-$, uses the Distances Convergence Theorem~\ref{thm: distances convergence theorem}, and a diagonal argument as in Lemma~5.1 in \cite{R60I} allows to conclude the proof of Theorem~\ref{thm:approx lisse met}.

Note that the statements  above remain true if the functions $\lambda_n$ are defined also by a harmonic function defined over the plane, which does not depend on $n$. It is the way theses results are presented in \cite{R60I}.

Theorem~\ref{thm:approx lisse met} can be modified as follows.

\begin{theo}
\begin{theorem}[{Riemannian Approximation Theorem \cite[Proof of Theorem~I]{R60I}}]\label{thm:approx met lisse bord}
Any point in an open domain $M$, which is not a point at infinity for the distance $\rho_{\lambda(\mes,h)}$, has a neighborhood homeomorphic to a disc that is the uniform limit of Riemannian distances with uniformly bounded total variation  of the curvature measure and uniformly bounded  total variation of the turn of the boundary.
\end{theorem}
\end{theo}

Of course, the interest in Theorem~\ref{thm:approx met lisse bord} is that the sequence of Riemannian metrics has uniformly bounded curvature, as many metric spaces can be approximated by Riemannian surfaces, see e.g., \cite{cassorla}.

To prove  Theorem~\ref{thm:approx met lisse bord}, one first 
skips the harmonic term using the localization procedure (see Remark~\ref{rem:localization harmonic}) on a closed disc contained in $M$ and centered at a finite point, and then use a smooth approximation.
The fact that  the absolute turn of the boundary 
is uniformly bounded comes from the expression of the turn given by \eqref{eq:turn riem harm}.

The next theorem is stated in a stronger form than in \cite{R60I}, but the proof is the same. If the $\lambda_n$ below are taken to be smooth, this gives a somehow converse to Theorem~\ref{thm:approx met lisse bord}. 

\begin{theo}
\begin{theorem}[{\cite[$\S 8$]{R60I}}]\label{thm:limit met bc disque}
Let $\rho$ be a finite intrinsic distance over $\overline{Q}_1(0)$, such that both topologies agree. 
Let  $\lambda_n=\lambda(\mes_n,h_n)$ 
be such that
\begin{itemize}
\item $h_n$ are harmonic functions defined on domains containing $\overline{Q}_1(0)$,
\item $\rho_{\lambda_n}$ is a finite distance over $\overline{Q}_1(0)$,
\item $\vert \mes_n \vert$ are uniformly bounded and their support is contained in a same disc,
\item the total variation of the (left or right) turn of $C_1(0)$ for $\rho_{\lambda_n}$ is uniformly bounded, 
\item $\rho_{\lambda_n}$  converge uniformly to $\rho$.
\end{itemize}
Then $\rho$ is a subharmonic distance over $\overline{Q}_1(0)$. 
\end{theorem}
\end{theo}

The proof is based on the following theorem, which is called Prokhorov Theorem if the measures are normalized 
to be probability measures, or is more generally a consequence of the Alaoglu's Theorem in functional analysis \cite{folland}. 
See \cite[Theorem 1.23]{mattila} for a direct proof in the case of positive measures.

\begin{theorem}\label{thm:alaoglu}
Let $(\mes_n)_n$ be a sequence of signed measures with supports contained in a same compact set, and such that 
 $\vert \mes_n \vert$ are uniformly bounded. Then  the sequence subconverge to a  signed measure.
\end{theorem}

It follows from Theorem~\ref{thm:alaoglu} that, under the hypothesis of Theorem~\ref{thm:limit met bc disque}, the sequence
$(\mes_n)_n$ subconverges. Then, we write $h_n$ as the potential of a measure $\psi_n$ plus a constant as in Lemma~\ref{lem:localisation harmonic}. The definition \eqref{eq:def turn} of the turn implies that $\psi_n$ are also uniformly bounded, hence subconverge.  A direct argument shows that the associated sequence of constants also subconverges. In turn, there are  a measure $\mes$ and a constant $c$ such that  $\rho_{\lambda_n}$  uniformly subconverge to $\rho_{\lambda(\mes,c)}$ on every compact set without  points with measure $\geq 2\pi$. It is possible to prove that the uniform convergence is actually on the whole disc, and to conclude that $\rho_{\lambda(\mes,c)}=\rho$. 

It is necessary to suppose that the limit distance in Theorem~\ref{thm:limit met bc disque} is a distance on a plane domain. Indeed, it is easy to construct a sequence of cones in the Euclidean $3$-space, which satisfies the other hypothesis of Theorem~\ref{thm:limit met bc disque}, but converges to the distance induced on the union of a plane and a segment. See \cite[p. 145]{R93} for more details.

\subsubsection{Consequences on curves}\label{sec:turn rotation}

Recall the notion of turn of a curve in a domain endowed with a subharmonic distance introduced in Section~\ref{sec:tot rot and bound rot}. Actually, the notion of turn only requires a measure and not a distance, however  turn and distance are related as we will see. The two theorems below are first proved in the cases $\mes= \mes_0\delta_0+\psi$, with $\mes_0<2\pi$ and $\psi$ a $C^\infty$ Lebesgue density, with some condition on $|\psi|$ insuring that
$$|\psi|<2\pi~,$$
see Section 7 of \cite{R63III}. By an approximation argument, the assumption about the regularity of $\psi$ can be relaxed in the statements. 
Then the proof of the theorems below for  the  case of a general measure 
 follows form the Localization Theorem~\ref{thm:localization}. Indeed, let us write the function $\tilde{\omega}$ obtained in this latter theorem in the following way:
 $$\tilde{\omega}(E)=\tilde{\omega}(E\setminus \{z_0\})+\tilde{\omega}(\{z_0\}\cap E)~.$$
  Hence, the Localization Theorem~\ref{thm:localization} says that, locally, one may write any signed measure with compact support under the form
 
  $$\psi + \mes_0\delta_{z_0}~.$$

By the properties of $\tilde{\mes}$, the measure $\psi$ is zero outside of a disc, and its total variation is arbitrary small.\footnote{We are using  the notation $\psi$ here to fit with the one from \cite{R63III}, but it is not the function used in Lemma~\ref{lem:localisation harmonic} or in the proof of the Localization Theorem~\ref{thm:localization}.}

\begin{theo}
\begin{theorem}[{\cite[Theorem 8]{R63III}}]\label{thm:shortest boundedturn} If a shortest arc has no point $z$ with $\mes(\{z\})\geq 2\pi$, then it is an arc of bounded turn, and its left and right turns
 are non-positive.
 \end{theorem}
 \end{theo}

The following result  is the converse result to
 Theorem~\ref{thm:br implique bt}, which said that an arc of bounded rotation is an arc of bounded turn.   

\begin{theo}
\begin{theorem}[{\cite[Theorem~10]{R63III}}]\label{thm:BT BR}
If an arc of bounded turn has no point $z$ with $\mes(\{z\})\geq 2\pi$, then it is an arc of bounded rotation.
\end{theorem}
\end{theo}

These results are also described at the end of Section~8.1 in \cite{R93}.

It follows from Lemma~\ref{lem:mes pos point infini} that if $\mes$ is positive, then a shortest arc $\arc$ has no point $z$ with $\mes(\{z\})\geq 2\pi$. Moreover, in this case, by  \eqref{eq:left+right} and Theorem~\ref{thm:shortest boundedturn},   $\kappa_l(K)=\kappa_r(K)=\mes(K^\circ)=0$.
We hence obtain the following result.
\begin{corollary}
If $\mes$ is a finite positive measure, then 
a shortest arc $\arc$ has bounded rotation, zero left and right turn, and $\mes(\arc^\circ)=0$.
\end{corollary}

Recall that a subharmonic distance between two points is defined as the infimum 
of the length $\tilde{s}_\lambda$ of the broken lines joining them.  
So far, the results we mentioned were based on the case when no point has a curvature measure $\geq 2\pi$, and by Lemma~\ref{lem metriques coincident}, the distance can be  defined using curves of bounded rotation, that is necessary in the arguments for proving the Distances Convergence Theorem~\ref{thm: distances convergence theorem}. From Theorem~\ref{prop:conv increas}, the following result is deduced.

\begin{theo}
\begin{theorem}[{\cite[Theorem 5]{R63}}]\label{thm length coincide}
For any rectifiable curve, the length $\tilde{s}_\lambda(K)$ and
the intrinsic length for the distance $\rho_\lambda$, $s_\lambda(K)$, coincide.
\end{theorem}
\end{theo}

Actually, the result given by \cite[Theorem 5]{R63}  is more general: it says that both lengths agree for \emph{any} curve. Of course, it is necessary to generalize the definition of $\tilde{s}_\lambda(K)$, that is done in \cite{R63}.

\subsection{Contraction onto a cone}

Let $\overline{Q}$ be a closed disc in the plane. 
A distance $\rho_n$ onto  $\overline{Q}$ is \emph{polyhedral}\index{polyhedral distance}, or equivalently the couple $(\overline{Q},\rho_n)$ is a \emph{(two dimensional) polyhedron}\index{polyhedron} if 
\begin{itemize}
\item it is intrinsic and induces the  topology of the disc;
\item any interior point has a neighborhood isometric to an open subset of a cone with positive cone angle;
\item any boundary point has a neighborhood isometric to an open subset of a circular sector ---in particular the boundary is a geodesic broken line.
\end{itemize}
In turn, a polyhedral distance is finite.

\begin{remark}{\rm
A polyhedral distance may also be defined as a gluing of 
Euclidean triangles along isometric edges. Such a gluing gives a polyhedral distance on the disc, although
that may be not so straightforward ---for example, one has to check that the intrinsic distance induced by the gluing is a genuine distance and not only a pseudo-distance. See e.g. Theorem~4.3 in \cite{bonahon}, and \cite{troyanov-euclidean}, for more details. (The general notion of gluing of metric spaces was developed by A. D.~Alexandrov, we refer to \cite{alexandrovintr}, \cite{AZ}, \cite{R93}.) 
Actually, any polyhedral metric can be obtained in this way, see e.g., Theorem~5.2.2 in \cite{R93}.
}\end{remark}

A distance $\rho$ over a space $U$ homeomorphic to a closed disc  \emph{admits a polyhedral approximation}\index{polyhedral approximation}  if
\begin{itemize}
\item it is a finite intrinsic distance that induces the topology of $U$;
\item there exists a sequence $(\rho_n)_n$ of polyhedral distances  over $\overline{Q}$,
uniformly converging to $\rho$, and 
there is a constant $A<\infty$ such that for all $n$, the positive part of the curvature of the polyhedron $R_n$ satisfies $\omega_n^+(R_n)<A$.
\end{itemize}
From Theorem~\ref{thm:alaoglu}, we know that up to extract a subsequence, 
$(\omega_n^+)_n$   converge weakly to a positive measure $\omega_0$. We imply that it is so when speaking about polyhedral approximation.

A \emph{convex cone} is a polyhedron with a unique interior vertex, such that when the turn of the boundary is non-zero at a point of the boundary, then it is positive.

The  following theorem is proved by approximation from 
 polyhedra. The polyhedral version of the following theorem is proved in \cite{R61b}  ---one may consult Section~5.6 and 6.5 in \cite{R93} for more around \cite{R61b} and \cite{R62}  respectively. The content of \cite{R61b}  and \cite{R62}  is independent of the theory of subharmonic distances.

\begin{theo}
\begin{theorem}[{Contraction Onto a Cone \cite[Theorem~1*]{R62} }]\label{thm contraction}
Let $\rho$ be an intrinsic distance over a closed disc $\overline{Q}$ that admits a  polyhedral approximation, and let $D$ be a proper subset of $\overline{Q}$ homeomorphic to a closed disc, whose boundary $\partial D$ is rectifiable for $\rho$.

If for the limit measure $\omega_0$, $\omega_0(D^\circ)<2\pi$, then there exists a convex cone  $Q$ such that 
$$\omega(Q)\leq \omega_0(D^\circ) $$
and there is a contracting mapping from $Q$ onto $D$, mapping $\partial Q$ onto $\partial D$, preserving the arc length.
\end{theorem}
\end{theo}

Recall that a \emph{contracting mapping}\index{contracting mapping}\footnote{The following alternative terminologies are also used for contracting mapping:  
\emph{inextensible}\index{inextensible},
 \emph{non-extensible}\index{non-extensible}, \emph{non-expansive}\index{non-expansive}, \emph{short}\index{short}, \emph{metric}\index{metric map}.} is a mapping $\varphi$ from a metric spaces $(M_1,\rho_1)$ to a metric space $(M_2,\rho_2)$ such that
$$\rho_2(\varphi(X),\varphi(Y))\leq \rho_1(X,Y)~. $$

We say that  a domain $D$ that satisfies the conclusion of Theorem~\ref{thm contraction} \emph{admits a contraction onto a cone}\index{contraction onto a cone}.

Actually, any subharmonic distance locally admits a polyhedral approximation, as says the following theorem. The proof is parallel to the one for
Riemannian approximation (Theorem~\ref{thm:approx met lisse bord}), and can be found  pp. 113--115 of \cite{R93}.
\begin{theorem}[Polyhedral Approximation Theorem]\label{tm:poly approx}
Any point in an open domain $M$, which is not a point at infinity for a subharmonic distance $\rho_{\lambda(\mes,h)}$, has a neighborhood homeomorphic to a disc that is the uniform limit of polyhedral distances, with uniformly bounded total variation  of the curvature measure, and uniformly bounded  total variation of the turn of the boundary.
\end{theorem} 

From Theorem~\ref{thm contraction} and Theorem~\ref{tm:poly approx} one immediately gets that if a domain does not have too much positive measure, then the subharmonic distance admits a contraction onto a cone, see Theorem~1 in \cite{R62}.   
%
%

Let $(X,d)$ be a metric space, $x\in X$, and let $\param_1,\param_2:[0,1]\to X$  be  parameterizations of  arcs $\arc_1$ and $\arc_2$, with $\param_1(0)=\param_2(0)=x$. For any $t\in(0,1]$, there is a unique, up to global isometries, triangle  in the Euclidean plane with edge-length $d(x,\param_1(t))$,  $d(x,\param_2(t))$,  $d(\param_1(t),\param_2(t))$. Let $\alpha_t$ be the angle of the triangle opposite to the edge of length $d(\param_1(t),\param_2(t))$. The \emph{upper angle}\index{upper angle} between $\arc_1$ and $\arc_2$ at $x$ is 
\begin{equation}\label{eq upper angle}\overline{\alpha}_x(\arc_1,\arc_2)=\limsup_{t\to 0}\alpha_t~. \end{equation}
As an angle in a Euclidean triangle can be written in term of edge-lengths, it is easy to check that the definition above does not depend on the choice of the parameterization. See e.g. \cite{AZ}, \cite{R93}, \cite{bbi} for details.

\begin{remark}{\rm \label{rem:angle}
A theorem of \cite{AZ} (see IV. 3 there) says that, if a distance admits a polyhedral approximation, 
 then the angle between two shortest arcs issued from the same point exists. That means that in \eqref{eq upper angle}, the $\limsup$ can be replaced by $\lim$, and the limit exists. 
By Theorem~\ref{tm:poly approx}, angles exist at any finite point of a subharmonic metric, and so we will speak about \emph{angle}\index{angle} instead of upper angle.  The same remark holds
for two-dimensional manifolds of bounded curvature,
that we will introduce soon.
}\end{remark}

A \emph{triangle}\index{triangle} in $S$ is the data of three distinct points together with three shortest paths joining them.
We first have that a contraction onto a cone also contracts angles, see $\S 7$ in \cite{R62}.

\begin{lemma}\label{lemm:comp angle cone}
Under the assumption of the Contraction Onto a Cone Theorem~\ref{thm contraction}, let us suppose that $D$ is a triangle $T$, and let $\overline{\alpha}$ be the  angle at a vertex of $T$, and let $\alpha$ be the angle at the corresponding vertex of $Q$. Then
$$\overline{\alpha}\leq \alpha~. $$

\end{lemma}

Lemma~\ref{lemm:comp angle cone} implies the following. 
A \emph{comparison triangle}\index{comparison triangle} 
of a triangle $T$ is the (unique up to isometry) Euclidean triangle with the same edge-length than $T$.

\begin{theorem}[First Comparison Theorem]\label{thm:first comparaison theorem}
Under the assumption of Lemma~\ref{lemm:comp angle cone},
let  $\overline{\alpha}$ be the  angle at a vertex of $T$, and let $\alpha_0$ be corresponding angle of a comparison triangle. Then
$$\overline{\alpha}-\alpha_0\leq \mes^+(T^\circ)~. $$
\end{theorem}

The proof of the First Comparison Theorem~\ref{thm:first comparaison theorem} goes as follows: it is first proved in the case when $T$ is a convex cone with total curvature $<2\pi$, then this is applied to the general case with the help of Lemma~\ref{lemm:comp angle cone}. All details are provided below Theorem~8.2.2 in \cite{R93}. 

\subsection{Two-dimensional manifolds of bounded curvature}\label{sec:sh est bic}


 A triangle is said to be \emph{simple}\index{triangle (simple)} if the union of the three shortest paths is the boundary of a domain homeomorphic to a disc, and also if it is convex with respect to its boundary. We won't be more precise, and refer to \cite{AZ}, \cite{R93}, \cite{creutz2021triangulating} for more details.

For a simple triangle $T$, its \emph{upper excess}\index{upper excess} is
\begin{equation}\label{eq:upper excess}\overline{\delta}(T)=\overline{\alpha}+\overline{\beta}+\overline{\gamma}-\pi~,\end{equation}
where $\overline{\alpha}$, $\overline{\beta}$, $\overline{\gamma}$ are the upper angles at the vertices of $T$ between its edges (see \eqref{eq upper angle}).

\begin{definition}\label{def:bic}
A metric space  $(M,\rho)$, where $M$ is a surface, is a \emph{two-dimensional manifolds of bounded curvature}\index{two-dimensional manifolds of bounded curvature} if:
\begin{itemize}
\item $\rho$ is a finite intrinsic distance that induces the topology of the surface;
\item for any point of  $M$, there exists a neighborhood $U$ of this point homeomorphic to a disc, 
and a constant $C(U)$, such that, for any set of pairwise non-overlapping simple triangles $T_i$ contained in $G$, then
$$\sum_i \overline{\delta}(T) \leq C(U)~.$$ 
\end{itemize}
\end{definition}

The distance on a two-dimensional manifold of bounded curvature is said to be \emph{of bounded curvature}\index{distance of bounded curvature} (or \emph{of bounded integral curvature}\index{distance of bounded integral curvature}). In the present section, we will restrict ourselves to the case when $M$ is a plane domain.

It is almost trivial that a polyhedral distance is a distance of bounded curvature. For a Riemannian distance, the upper angles are well-defined angles, and by the Gauss--Bonnet Formula (see e.g., \cite{spivak3}), the upper excess of a triangle is equal to the curvature measure of that triangle. It follows that Riemannian distances over $M$ are distances of bounded curvature.

The following theorem is a  simple consequence of the Contraction Onto a Cone Theorem~\ref{thm contraction}. The proof can also be found page 96 in \cite{R93}.\footnote{At page 96 in \cite{R93}, ``Theorem~6.3.2'' should be read ``Theorem 6.2.2''.}

\begin{theo}
\begin{theorem}[{\cite[Theorem 2]{R62} }]\label{thm:approx poly implique bic}
Any distance over a bounded domain admitting a polyhedral approximation is a distance of bounded curvature.
\end{theorem}
\end{theo}

Theorem~\ref{thm:approx poly implique bic} was first proved in \cite{AZ} (see Chapter~IV there), with a more intricate argument.

Theorem~\ref{thm:approx poly implique bic} together 
with Theorem~\ref{tm:poly approx} gives the following result, that is one half of the starting  motivation 
 for introducing subharmonic distances at the time.

\begin{theo}
\begin{theorem}[{Theorem I  in \cite{R60I}}]\label{thm:sh is bic}
Any finite subharmonic distance is a distance of bounded curvature.
\end{theorem}
\end{theo}

The argument in the proof of Theorem~\ref{thm:approx poly implique bic}  allows to be more precise: if the curvature measure of the subharmonic distance is $\mes$, then, with notation of Definition~\ref{def:bic}, one can take
 $$C(U)=\mes^+(U)~. $$

There exists a converse to Theorem~\ref{thm:sh is bic}, see Section~\ref{sec:iso coord on bic}.

\section{Conformal aspects of subharmonic distances}\label{sec:conforme}
In the first part of this section, we present two theorems. The first one (Theorem~\ref{thm huber 1}) was proved simultaneously by A. Huber and  Y. Reshetnyak. 
It  says that a distance preserving mapping between domains with subharmonic distances is a conformal (or anti-conformal) mapping. 
The two proofs are different, but both uses the  Distances Convergence Theorem~\ref{thm: distances convergence theorem}.

The second theorem is a formula proved by Huber 
(Theorem~\ref{thm uber chgt coord}), that will imply, in the second part of this section, that 
a subharmonic metric can be defined on a Riemann surface. Roughly speaking, under a conformal mapping, a subharmonic metric changes in the same way as 
a Riemannian conformal metric does.

Conversely to what we did until now, we  use a result from 
theory of surfaces of bounded curvature to present Theorem~\ref{thm uber chgt coord}, see Remark~\ref{rem:huber rem aires}.

\subsection{Conformal mappings}\label{sec:conformal}

As we have an orientation for simple closed curves in the plane,
we say that for a domain $M$, a homeomorphism from $M$ onto another domain of the plane  is \emph{orientation preserving}\index{orientation preserving},  
if any simple closed curve in $M$ and its image have the same orientation. If the orientation is always different, we say that the mapping \emph{reverse the orientation}. The (complex) conjugate mapping reverses the orientation, and the conjugate of an orientation preserving mapping reverses the orientation, and vice-versa.

A mapping $f$ from a domain $M$ of the plane to $\C$ is a \emph{conformal mapping}\index{conformal mapping}, if it is injective and holomorphic. By a standard result of complex analysis, 
this implies that for any $z\in M$, $f'(z)\not= 0$. By the invariance of domain, $f$ is a homeomorphism onto its image, and we have that the inverse mapping is holomorphic. 
By the Cauchy--Riemann equations, the Jacobian is positive, hence a conformal mapping preserves orientation. The conjugate of a conformal mapping is called an \emph{anti-conformal mapping}\index{anti-conformal mapping}.

Let $M, M'$ be two domains of the plane, and let 
$f:M\to M'$ be an  orientation preserving homeomorphism. 
Let $z_0\in M$ and
\begin{equation}\label{eq:vapr psi} \varphi(\delta)=\max_{\vert z-z_0\vert=\delta }\vert f(z)-f(z_0)\vert~,\,
\psi(\delta)=\min_{\vert z-z_0\vert=\delta }\vert f(z)-f(z_0)\vert~,\end{equation}
see Figure~\ref{fig:cercles}. The \emph{circular dilatation}\index{circular dilatation} of $f$ at $z_0$ is 
$$H(z)=\limsup_{\delta\to 0^+} \frac{\varphi(\delta)}{\psi(\delta)}~.$$

\begin{theorem}[D. E. Men'shov]\label{thm:mensh}
Let $M, M'$ be two domains of the plane, and let 
$f:M\to M'$ be an  orientation preserving homeomorphism. If $H\leq 1$ over $M$, except maybe for a finite number of points, 
then $f$ is a conformal mapping.
\end{theorem}

This theorem was  proved in \cite{menshov}. Nowadays, 
an orientation preserving  homeomorphism with $H\leq 1$ 
is called a \emph{$1$-quasi conformal mapping}. This is not the most usual definition, but it is equivalent \cite[Theorem 4.2 p.178]{lehto-virtanen}. With a standard definition, the fact that a $1$-quasi conformal mapping is conformal is a basic result, see \cite{lehto-virtanen}, \cite{gardiner-lakic}, \cite{ahlfors-qc}. 
If the condition $H\leq 1$ fails at a finite number of points of $M$, as $f$ is continuous over $M$, by the Riemann’s Principle of Removable Singularities, $f$ is actually holomorphic over all $M$.

\begin{definition}
A \emph{distance preserving mapping}\index{distance preserving mapping} $f$ between
 two metric spaces $(X_1,d_1)$ and $(X_2,d_2)$ is a mapping satisfying
 $f^*d_2=d_1$, i.e.
 $$\forall x,y\in X_1~,\,\, d_2(f(x),f(y))=d_1(x,y)~.$$
 
\end{definition}

If the metric spaces are homeomorphic to domains of the plane, a distance preserving mapping is continuous and injective, so by the Invariance of Domain  Theorem, it is a homeomorphism onto its image.
If the distance preserving mapping is moreover surjective, the two metric spaces are \emph{isometric}\index{isometry between metric spaces}. 
 
Theorem~\ref{thm:mensh} will give the following result.

\begin{theo}
\begin{theorem}[{\cite{huber}, \cite[Theorem 7]{R63III}}]\label{thm huber 1}
Let $\mes_1, \mes_2$ be signed measures with compact support, $M_1,M_2$ be two bounded domains in the plane and $h_1,h_2$ be harmonic  over $M_1$ and $M_2$ respectively. Let $\lambda_1=\lambda(\mes_1,h_1)$, $\lambda_2=\lambda(\mes_2,h_2)$.  Suppose that $\rho_{\lambda_1}$ and $\rho_{\lambda_2}$ are finite.

If $$f:(M_1,\rho_{\lambda_1})\to  (M_2,\rho_{\lambda_2})$$ is distance preserving, then $f$ is conformal or anti-conformal onto its image.
\end{theorem}
\end{theo}

Note that by the results of Section~\ref{sec can stretch}, we know that  a subharmonic  distance has only a finite number of points at infinity, as such a point $z$ must satisfy $\mes(\{z\})\geq 2\pi$.
 By Theorem~\ref{thm:met coinc}, 
the topologies induced by $\rho_{\lambda_1}$ and $\rho_{\lambda_2}$ coincide with the Euclidean topology, 
so $f$ is a homeomorphism onto its image. Let us suppose that it is orientation preserving.  Up to exclude a finite number of points, one can consider that all points in $M$
have measure $<2\pi$ and similarly for $M'$. 
Using canonical stretching (Theorem~\ref{thm:canoncial streching}), it follows that $f$ has circular dilatation $\leq 1$. This is the argument in \cite{R63III}.
Theorem~\ref{thm huber 1} was first proved in \cite{huber}, by a different argument, but this proof also uses the Distances Convergence Theorem~\ref{thm: distances convergence theorem}.

We now want to see how a distance preserving mapping changes the metric $\lambda \vert \D z\vert^2$. The following is a generalization
of Lemma~\ref{lem:fonctions radiales}.

\begin{theo}
\begin{theorem}[{\cite[Lemma 2]{huber}}]\label{prop:huber 1 aire}
let $\mes$ be a signed measure, $h$ be a harmonic mapping over a bounded domain $M$, and $\lambda=\lambda(\mes,h)$. For any $z_0\in M$ for which $\lambda$ is defined, then

$$\frac{1}{\pi r^2}\iint_{Q_r(z_0)} \lambda \xrightarrow[r\to 0]{} \lambda(z_0)~. $$
\end{theorem}
\end{theo}

So if we know how $\iint_{Q_r(z_0)} \lambda $ changes under a distance preserving mapping, it should help to understand how $\lambda$ changes under the same mapping.  Basically, it will suffice to notice that 
 $\iint_{Q_r(z_0)} \lambda $ is ``the area'' of $Q_r(z_0)$ for the distance
 $\rho_\lambda$.
 
Let us first recall the following classical result of Riemannian geometry.

\begin{lemma}\label{lem:hausdorff lisse}
Let $\lambda$ be $C^\infty$ over a neighborhood of a bounded domain $M$. Then,  for any Borel set $E\subset M$,
$\iint_E \lambda=\mathcal{H}_{\rho_\lambda}^2(E)$, where $\mathcal{H}_{\rho_\lambda}^2$ is the Hausdorff measure for the distance $\rho_\lambda$,  see Section~\ref{sec: hausdorff sub}. 
\end{lemma}
\begin{proof}
Let $z\in M$. There exists an open disc $B\subset T_{z}M\simeq \mathbb{R}^2$ such that the exponential map  is a diffeomorphism onto its image. 
In particular, the exponential map  provides local coordinates around $z$, which are \emph{normal coordinates}, in the sense that the pull-back of the metric is equal to the identity at $0$ and the first derivatives of its coefficients are zero, see e.g, \cite{josteriem}. That gives  
$$\operatorname{exp}^*(\lambda\vert \D z\vert^2)=\vert \D z\vert^2+o(\epsilon)~,$$
i.e., if $\epsilon$ is sufficiently small for that $\vert o(\epsilon)\vert \leq \epsilon$, we have
$$1-\epsilon\leq \lambda\circ \operatorname{exp} \vert \operatorname{Jac}(\operatorname{exp}) \vert\leq 1+\epsilon~,$$
so for a Borel subset $E\subset B$, 
$$(1-\epsilon)\mathcal{L}(E)\leq \iint_{ \operatorname{exp}(E)} \lambda\leq (1+\epsilon)\mathcal{L}(E)  $$
where $\mathcal{L}$ is the Lebesgue measure. By Proposition~\ref{prop: pte Haus}, if $\mathcal{H}^2$ is the Hausdorff measure of the plane,
\begin{equation}\label{eq:prem com hau lamd}(1-\epsilon)\mathcal{H}^2(E)\leq \iint_{ \operatorname{exp}(E)} \lambda\leq (1+\epsilon)\mathcal{H}^2(E)~.\end{equation}

Comparing in the same way the length of smooth paths, we have 
$$\sqrt{1-\epsilon} s \leq \tilde{s}_\lambda \leq \sqrt{1+\epsilon}s~,$$
and by 
Lemma~\ref{lem: trick discs}, we have that over 
a concentric disc $B'\subset B$, 
$$\sqrt{1-\epsilon} d \leq \rho_\lambda \leq \sqrt{1+\epsilon}d~,$$
where $d$ is the Euclidean distance. As $\sqrt{1-\epsilon}\sqrt{1+\epsilon}\leq 1$,
we obtain $\sqrt{1-\epsilon}\rho_\lambda\leq d$
and by Lemma~\ref{lem:haus lip},
$$(1-\epsilon)\mathcal{H}_{\rho_\lambda}^2 \leq \mathcal{H}^2~.$$
In the same way, $d\leq \frac{1}{\sqrt{1-\epsilon}} \rho_\lambda$, hence
$$\mathcal{H}^2 \leq  \frac{1}{1-\epsilon}\mathcal{H}_{\rho_\lambda}^2~,$$
and using these last two inequalities with \eqref{eq:prem com hau lamd}, for a Borel subset $E\subset B'$,
$$(1-\epsilon)^2\mathcal{H}_{\rho_\lambda}^2(\operatorname{exp}(E))\leq \iint_{ \operatorname{exp}(E)} \lambda \leq \frac{1+\epsilon}{1-\epsilon}\mathcal{H}_{\rho_\lambda}^2(\operatorname{exp}(E))~. $$

The equation above holds over a neighborhood of $z$, that depends on $z$, and for any $\epsilon$ sufficiently small. The result follows by compactness of $\overline{M}$.
\end{proof}

\begin{lemma}\label{lem: iso preserve aire}
Under the conditions of Theorem~\ref{thm huber 1}, if $f:(M_1,\rho_{\lambda_1})\to  (M_2,\rho_{\lambda_2})$ is distance preserving, then for any Borel set $E\subset M$,
$$\iint_{f(E)}\lambda_2 = \iint_E \lambda_1~.  $$
\end{lemma}

\begin{remark}\label{rem:huber rem aires}{\rm Lemma~\ref{lem: iso preserve aire}  is a direct consequence of Lemma~\ref{lem:hausdorff lisse} when $\lambda$ is $C^\infty$, as, as it obviously follows from the definition, a distance preserving mapping preserves the Hausdorff measure. In the general case, it is noted in the footnote 7 in  \cite{huber} that Lemma~\ref{lem: iso preserve aire}  follows by approximation. We are not aware of a straightforward argument,  but Lemma~\ref{lem: iso preserve aire} follows 
from Theorem~\ref{thm:formule aire}, itself proved using a result from the theory of two-dimensional manifolds of bounded curvature.}\end{remark}
%
%

\begin{theo}
\begin{theorem}[{\cite{huber}}]\label{thm uber chgt coord}
Let $\mes_1, \mes_2$ be signed measures with compact support, $M_1,M_2$ be two bounded domains in the plane and $h_1,h_2$ be harmonic functions over $M_1$ and $M_2$ respectively. Let $\lambda_1=\lambda(\mes_1,h_1)$, $\lambda_2=\lambda(\mes_2,h_2)$.  Suppose that $\rho_{\lambda_1}$ and $\rho_{\lambda_2}$ are finite. 
If $f:(M_1,\rho_{\lambda_1})\to  (M_2,\rho_{\lambda_2})$ is distance preserving, then 
\begin{equation}\label{eq uber chgt coord}\lambda_1=\vert f'\vert^2\lambda_2\circ f  ~.\end{equation}
\end{theorem}
\end{theo}

The proof of Theorem~\ref{thm uber chgt coord} is contained in \cite{huber}, even if its statement is not explicitly written. 
Let us first do the following observation.

\begin{lemma}\label{lem:haus conf}
Let $f:M\to \C$ be a conformal mapping, and let $X\subset M$ be a set of Hausdorff dimension zero. Then
$f(M)$ has 
Hausdorff dimension zero.
\end{lemma}
\begin{proof}
On any compact subset of $M$, $f$ is Lipschitz so the result follows from Lemma~\ref{lem:haus lip}.
\end{proof}
 
\begin{proof}[Proof of Theorem~\ref{thm uber chgt coord}]
Recall the notation of \eqref{eq:vapr psi}.  Clearly (see Figure~\ref{fig:cercles})
$$Q_{\psi(\delta)}(f(z_0))\subset f(Q_\delta(z_0)) \subset Q_{\varphi(\delta)}(f(z_0))~,$$
hence by Lemma~\ref{lem: iso preserve aire},
$$\frac{\frac{1}{\pi \psi(\delta)^2} \iint_{Q_{\psi(\delta)}(f(z_0))}\lambda_2}
{\frac{1}{\pi \delta^2}\iint_{Q_{\delta}(z_0)}\lambda_1}\leq \frac{\delta^2}{\psi(\delta)^2}$$
and
$$\frac{\frac{1}{\pi \delta^2}\iint_{Q_{\delta}(z_0)}\lambda_1}{\frac{1}{\pi \varphi(\delta)^2}\iint_{Q_{\psi(\delta)}(f(z_0))}\lambda_2} \geq \frac{\varphi(\delta)^2}{\delta^2}~.$$

The result follows taking $\delta \to 0$: by Theorem~\ref{prop:huber 1 aire}  and 
by the fact that as $f$ is derivable by Theorem~\ref{thm:mensh}, 
$\lim_{\delta\to 0} \frac{\varphi(\delta)}{\delta}=\lim_{\delta\to 0} \frac{\psi(\delta)}{\delta}=\vert f'\vert$.
~\end{proof}


\begin{figure}
\begin{center}
\includegraphics{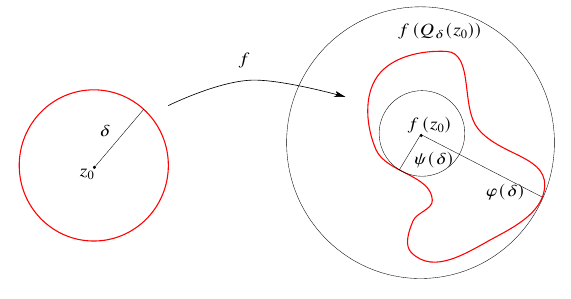}
\caption{To Theorem~\ref{thm uber chgt coord}.}\label{fig:cercles}
\end{center}
\end{figure}

Let us summarize those results in a more classical form. 

\begin{definition}
A surjective mapping $f$ between two domains endowed with finite subharmonic distances $\rho_{\lambda_1}$ 
and $\rho_{\lambda_2}$ is an \emph{isometry}\index{isometry} if $f$ is conformal and 
satisfies \eqref{eq uber chgt coord}.
\end{definition}

\begin{theorem}[Myers–Steenrod Theorem]
A  surjective mapping between  domains endowed with finite subharmonic distances is an isometry if and only if it is distance preserving.
\end{theorem}

One direction is given by Theorem~\ref{thm huber 1} and Theorem~\ref{thm uber chgt coord}.  For the other direction, it is clear that an isometry is length preserving, and hence distance preserving as the distances are intrinsic (we use Lemma~\ref{lem:conf BR} below). 

\begin{lemma}\label{lem:conf BR}
The image of a rectifiable arc (resp. of bounded rotation)  by a conformal mapping is a rectifiable (resp. of bounded rotation) arc. 
\end{lemma}
\begin{proof}
First note that the mapping is Lipschitz over any arc, so the image of a rectifiable arc is rectifiable. 
As by assumption the derivative of the mapping is never zero, the image arc has a right tantrix. Moreover, partial derivatives of  the mapping are also Lipschitz over any arc, and then it is easy to see that the right tantrix of the image of the arc has bounded variation, see Proposition~\ref{prop:tatrix kapap}. We leave the details to the reader. 
\end{proof}

Let us end this section by checking that subharmonicity is invariant by composition by a conformal mapping.
For a $C^2$ function $h:M'\to \R$ and a $C^2$ mapping $f:M\to M'$, a tedious computation shows that
$$\Delta (h\circ f)= \frac{\partial h}{\partial x}\Delta f_1 + \frac{\partial h}{\partial y}\Delta f_2 + 
\frac{\partial^2 h}{\partial x^2}\left(\left(\frac{\partial f_1}{\partial x}\right)^2 +  \left(\frac{\partial f_1}{\partial y}\right)^2\right)$$
$$+ 
\frac{\partial^2 h}{\partial y^2}\left(\left(\frac{\partial f_2}{\partial x}\right)^2 +  \left(\frac{\partial f_2}{\partial y}\right)^2\right)
+2\frac{\partial^2 h}{\partial x\partial y}\left(\frac{\partial f_1}{\partial x}  \frac{\partial f_2}{\partial x} + \frac{\partial f_1}{\partial y}  \frac{\partial f_2}{\partial y}  \right)~.
$$

Now, if $f$ is conformal, its real and imaginary parts are harmonic, i.e., $\Delta f_1=\Delta f_2=0$, and the Cauchy--Riemann equations say that 
$$\frac{\partial f_1}{\partial x}  \frac{\partial f_2}{\partial x} + \frac{\partial f_1}{\partial y}  \frac{\partial f_2}{\partial y}  =0$$
and also that 
$$\left(\frac{\partial f_1}{\partial x}\right)^2 +  \left(\frac{\partial f_1}{\partial y}\right)^2=\left(\frac{\partial f_2}{\partial x}\right)^2 +  \left(\frac{\partial f_2}{\partial y}\right)^2=\det Df$$
so in this case we have
\begin{equation}\label{eq:lapl composition conforme}
\Delta \left( h\circ f\right)= \det Df \Delta h~.
\end{equation}

Let us note the following consequence.

\begin{lemma}\label{lem:comp harm conf}
The right composition of a harmonic function by a conformal mapping is harmonic.
\end{lemma}

With Cauchy--Riemann equations, it is easy to check the following lemma. Actually, this property characterizes conformal mappings, see e.g., \cite{Rudin}, \cite{beardon}, \cite{bak-newman}.

\begin{lemma}\label{lem:harm angle}
A conformal mapping preserves the angle between any two crossing smooth arcs.
\end{lemma}

\begin{lemma}\label{lem:image measure conformal map}
Let $f:M\to \C$ be a conformal mapping and $\mes$ be a positive measure with compact support. Then
$$\Delta \left(p(\mes)\circ f\right) = f^*\mes~.$$
\end{lemma}
\begin{proof}
Let us first consider the case $\mes=\delta_\zeta$, 
with $p:=p(\delta_\zeta)=\frac{1}{2\pi} \ln \vert \cdot-\zeta |$.
As $\ln \vert \cdot-\zeta |$ is harmonic on $\C\setminus \{\zeta\}$, by Lemma~\ref{lem:comp harm conf}, $p\circ f$ is harmonic on $\C\setminus \{f^{-1}(\zeta)\}$.
By Green formula \eqref{eq:green},
$$ \iint_{\C\setminus f^{-1}\left(Q_\epsilon(\zeta)\right)}  p\circ f \Delta \varphi  =\int_{f^{-1}\left(C_\epsilon(\zeta)\right)}
p\circ f\frac{\partial \varphi}{\partial \normal} - \varphi\frac{\partial p\circ f}{\partial \normal}~,$$
where  $\normal$ is the  exterior unit normal to $\partial (\C \setminus f^{-1}\left(Q_\epsilon(\zeta)\right))$.

The function $\frac{\partial \varphi}{\partial \normal}$ is continuous, hence bounded  by some $A$ on $ f^{-1}\left(C_\epsilon(\zeta)\right)$, so  letting $\param(t)=\zeta+\epsilon \E^{\I t}$, 
\begin{equation*}
\begin{split}
 \left|\int_{f^{-1}\left(C_\epsilon(\zeta)\right)} p\circ f\frac{ \partial \varphi}{\partial \normal}\right| & \leq  A \int_{f^{-1}\left(C_\epsilon(\zeta)\right)} p\circ f\\ 
&=  \frac{A}{2\pi} \ln |\epsilon|\epsilon \int_0^{2\pi}\left| (f^{-1})'(\param(t))\right|\D t
   \xrightarrow[\epsilon\to 0]{}  0~. 
     \end{split}
 \end{equation*}
 
The exterior unit normal to $\partial (\C \setminus C_\epsilon(\zeta))$ at a point $z$ is $-(z-\zeta)/\epsilon$, and so by Lemma~\ref{lem:harm angle}, 
$$\nu(z)=X/\vert X \vert~, \, X=Df^{-1}(f(z))\left(-\frac{f(z)-\zeta}{\epsilon}\right)~,$$
and as $f$ is conformal,
$|X|=|\left(f^{-1}\right)'(f(z))|\left\vert \frac{f(z)-\zeta}{\epsilon} \right\vert = |\left(f^{-1}\right)'(f(z))|~. $
With this and  \eqref{eq:grad ln} we compute
$\frac{\partial p\circ f}{\partial \normal}=-\frac{1}{2\pi\epsilon}|\left(f^{-1}\right)'\circ f|^{-1}$, so parameterizing $C_\epsilon(\zeta)$ by 
$=\zeta+\epsilon \E^{\I t}$,
\begin{equation*}
\begin{split}
\Delta \left(p\circ f\right) (\varphi) &=\lim_{\epsilon\to 0}\iint_{\C \setminus f^{-1}\left( Q_\epsilon(\zeta)\right)} p\circ f \Delta \varphi \\
&=\lim_{\epsilon\to 0} -\int_{ f^{-1}\left( C_\epsilon(\zeta)\right)} \varphi\frac{\partial p\circ f}{\partial \normal} \\
& =\lim_{\epsilon\to 0} \frac{1}{2\pi\epsilon}\int_{f^{-1}\left(C_\epsilon(\zeta)\right)} \varphi |\left(f^{-1}\right)'\circ f|^{-1}  \\
& =\lim_{\epsilon\to 0} \frac{1}{2\pi\epsilon}\int_0^{2\pi} \epsilon \varphi (f^{-1}(\zeta+\epsilon \E^{\I t}))\D t = \varphi (f^{-1}(\zeta))\end{split}
\end{equation*}
so
$$\Delta \left(p(\delta_\zeta)\circ f\right)=\delta_{f^{-1}(\zeta)}~. $$

Then, similarly to Remark~\ref{rem:lap convolution}, we have that, for a smooth function $\varphi$ with compact support,
$$\Delta (p(\mes)\circ f)(\varphi)= \iint \varphi \circ f^{-1} \D\mes~.$$
~\end{proof}

\begin{lemma}\label{lem: ssh conf}
The right composition of a subharmonic function by a conformal mapping is subharmonic.
\end{lemma}
\begin{proof}
By Lemma~\ref{lem:comp harm conf}, it suffices to check the result for potentials. 
By Lemma~\ref{lem:image measure conformal map} and Weyl's lemma (Theorem~\ref{thm:Weyl's Lemma}), $p(\mes)\circ f - p(f^*\mes)$ is equal almost everywhere to
a harmonic function. 
For a given $z$ we have 
$$\frac{f^{-1}(\zeta)-z}{\zeta-f(z)}=\frac{f^{-1}(\zeta)-f^{-1}(f(z))}{\zeta - f(z)}\xrightarrow[\zeta\to f(z)]{} \left( f^{-1}\right)' \left(f(z)\right) $$ 
which is non-zero as $f$ is conformal. Hence the function
$$\zeta \mapsto \ln\left\vert \zeta - f(z)\right\vert - \ln\left\vert f^{-1}(\zeta)-z\right\vert$$
is continuous, that implies that the function $p(\mes)\circ f - p(f^*\mes)$ is continuous, hence harmonic.
\end{proof}

\subsection{A glimpse to Riemann surfaces}


Recall that a (connected) (topological) \emph{surface}\index{surface} $S$ is a connected second countable Hausdorff topological space such that each point has a neighborhood  homeomorphic to an open subset of the plane.   If a point $x\in S$ belongs to two such neighborhoods $U_1, U_2$, which are sent respectively by $f_1$ and $f_2$ into the plane, then the homeomorphism
$$f_{21}:f_2\circ f_1^{-1}:   f_1(U_1\cap U_2) \to f_2(U_1\cap U_2)  $$ 
is called a \emph{transition mapping}\index{transition mapping}. A pair $(U_i,f_i)$ is called a \emph{chart}\index{chart}. A collection of charts covering the surface is an \emph{atlas}. The surface is \emph{orientable} if the transition mappings are orientation preserving. Actually, a surface is more precisely an equivalence class of atlases, but as we don't need such consideration, we refer to any textbook about differential geometry for more details.

A \emph{Riemann surface}\index{Riemann surface} is a  connected surface having an atlas whose 
transition mappings are conformal.  Note that as conformal mappings preserve the orientation, Riemann surfaces are orientable. 
(The hypothesis of second countability can be relaxed. It is  obvious in the compact case, and given by Rad\'o Theorem in the general case, see e.g. \cite{jost}, \cite{forster}, \cite{ahlfors-riemann}.)
A chart of such an atlas is called a \emph{conformal chart}\index{conformal chart}.

Let $S$ be a connected orientable surface.
We will say that an intrinsic  distance on  $S$ has  \emph{isothermal coordinates}, \index{isothermal coordinates} if $S$ has an atlas such that 
each chart gives a distance preserving mapping to a 
plane domain with a subharmonic distance. Here, the distance considered on the domain of the chart is the induced intrinsic distance.
As  the transition mappings must be distance preserving,   if follows from  Theorem­~\ref{thm huber 1} that  a distance with such an atlas
turns $S$ into a Riemann surface. 

Note that from Lemma~\ref{lem:haus conf}, the notion of quasi-everywhere (see Section~\ref{sec: hausdorff sub}) holds on a Riemann surface.

\begin{definition}
Let $S$ be a Riemann surface. A \emph{subharmonic metric}\index{subharmonic metric} on $S$ 
is a $(0,2)$-tensor defined quasi-everywhere on $S$, having the form  $\lambda \vert \D z\vert^2$ in any conformal chart, where $\ln \lambda$ is  $\delta$-subharmonic.
\end{definition}

In conclusion, a distance with isothermal coordinates on a surface defines a subharmonic metric by Theorem~\ref{thm uber chgt coord}.

A subharmonic metric will play the role of a ``generalized Riemannian metric''. 
Note that a subharmonic metric is defined only over a Riemann surface, conversely to a Riemannian metric that can be defined over any smooth surface,
but things agree due to the existence of isothermal coordinates on Riemannian surfaces, see Remark~\ref{rem:C1C2curvature}.

\begin{example}{\rm
A classical example of subharmonic metric which is not a Riemannian metric comes from \emph{holomorphic quadratic differentials} on surfaces, that are locally of the form $f\D z^2$, and play a prominent role in the theory of Riemann surfaces \cite{kapovich}, \cite{farb}. They give  subharmonic metrics of the form $|f||\D z|^2$, which are flat except maybe at a finite number of conical singularities, see Lemma~\ref{lem:ln module holo}.
See \cite{troyanov-euclidean1} for more details (the content of this last reference may be found in \cite{nikolaev}).

Another related class of examples is given by Riemannian metrics with simple singularities, see \cite{tro-ouvert}. }
\end{example}


%

\begin{remark}\label{rem:construction sshmet}{\rm
Any Riemann surface supports a subharmonic metric, as  it is easy to construct a Riemannian metric on any compact Riemann surface which is conformal, i.e., that has the form $\lambda|\D z|^2$ in a conformal chart, see e.g., \cite[Lemma 2.3.3]{jost}. }\end{remark}

Now, let us start from a Riemann surface $S$ together with a subharmonic metric $g$.
Let  $U_1, U_2$ be two overlapping conformal charts of $S$,  and let $p\in U_1\cap U_2$. Let $z=f_1(p)$, $f=f_{12}$
and  $X$ be the image of an element $Y$ of $T_pS$ into $U_1$. Then by 
definition, 
\begin{equation*}\label{eq:chgt met conf dif}\lambda_1(z)| X |^2=\lambda_2(f(z))|Df(z)(X)|^2~, \end{equation*}
and then
\begin{equation}\label{eq:chgt met conf dif2}\lambda_1=\vert f'\vert^2 \lambda_2\circ f~. \end{equation}

Let $\arc$ be a rectifiable  arc  between two points of $S$ (that is well-defined by Lemma~\ref{lem:conf BR}). We can define the length for $g$ of $\arc$ by decomposing the arc into subarcs which meet a chart, and then summing. By \eqref{eq:chgt met conf dif2} this is well-defined. We define the distance $d_g$ as the distance induced by this length structure. We say that $d_g$ is a \emph{subharmonic distance}\index{subharmonic distance (on a Riemann surface)}. 
 By construction, $d_g$ is a distance with isothermal coordinates.

Let us consider two overlapping  conformal charts
over $S$, with the subharmonic metric defined respectively by $\lambda(\mes_1,h_1)$
and $\lambda(\mes_2,h_2)$ in those charts, and let $f$ be the conformal change of chart. Over the domain of coordinates change $M$,  \eqref{eq:chgt met conf dif2} gives
\begin{equation}\label{eq:huber develope}p(\mes_1)+h_1=p(\mes_2)\circ f +h_2\circ f + 2\ln \vert f' \vert  \end{equation}
and  as $\ln \vert f'\vert$ is harmonic as $f$ is conformal (Lemma~\ref{lem:ln module holo}), and by Lemma~\ref{lem:comp harm conf},
over $M$
$$\Delta p(\mes_1) = \Delta \left(p(\mes_2)\circ f\right)$$
and by Lemma~\ref{lem:image measure conformal map},
$$\mes_1\vert_U = f^*\mes_2\vert_U~. $$

In turn, the following definition makes sense.

\begin{definition}
The \emph{curvature measure}\index{curvature measure}
of  a subharmonic metric on a Riemann surface is the Borelian measure  $\mes$, 
such that if $E$ is included in a conformal chart $(U,f)$, then
$\mes(E)=\mes_U (f(E))$, where $\mes_U$ is the curvature measure of the subharmonic distance over $f(U)$. 
\end{definition}

Using the maximum principle (Lemma~\ref{princip max}), it is easy to see that a conformal function over a compact Riemann surface must be constant. In turn, 
Weyl Lemma Theorem~\ref{thm:Weyl's Lemma} 
  implies the following, which is the uniqueness  part of Theorem~7.3 in \cite{troyanov-alexandrov}.

 \begin{proposition}\label{prop:unique measure riem}
 Two subharmonic metrics on a compact Riemann surface with the same curvature measure differ by a multiplicative constant.
 \end{proposition}
 
 We would like to know if an existence result holds, i.e., if a measure over a Riemann surface may define a subharmonic metric. We first have classical Gauss--Bonnet necessary condition that we now present. 
 
Let $z$ be the interior point of an arc $\arc$, having semitangents at both directions at this point. Recall from Section~\ref{sec:rot arc pla} that there is a notion of left side of the curve. By considering small circles centered at $z$, and taking a limit when the radii go the zero, we can define the 
\emph{(left) angle}\index{left angle}\index{angle (left)} $\theta_l(z)$ of $\arc$ at $z$, see \cite{R63III} for details. Let $\rho_\lambda$ be a subharmonic distance over a domain $M$ containing $\arc$. The (left) \emph{subharmonic angle}\index{subharmonic angle}\index{angle (subharmonic)} for the distance $\rho_\lambda$ of $\arc$ at $z$ is 
\begin{equation}\label{eq:left angle}\theta_{l,\lambda}(z)=\left(1-\frac{\mes(\{z\})}{2\pi}\right)\theta_l(z)~. \end{equation}
By Lemma~\ref{lem:harm angle}, the subharmonic angle does not depend on the choice of a conformal chart.
This, and the two following results, imply that the
\emph{left turn}\index{turn (on a surface)} of an arc can be defined on a surface with a subharmonic metric. Of course, similar definitions and result hold when replacing ``left'' by ``right'', and $\theta_l+\theta_r=2\pi$.

\begin{theo}
\begin{theorem}[{\cite[Theorem 9]{R63III}}]\label{thm cv angles turn}
Let $\arc$ be a simple arc with a right semitangent at its starting point, and a left semitangent at its endpoint. Let $(L_n)_n$  be a sequence of geodesic broken lines converging on the left to $\arc$, non-overlapping with $\arc$, and with the same extremities. Then
\begin{equation}\label{eq:def limite turn}
\sum_i \left(\pi  - \theta_{l,\lambda}(z_{n,i})\right) +\alpha_n+ \beta_n \xrightarrow[n\to \infty]{}  \kappa_l(K)~,\end{equation}
where $z_{n,i}$ are the vertices of $L_n$, $\alpha_n$ and $\beta_n$ are the interior angles of $K\cup L_n$ at the extremities of both curves. 
\end{theorem}
\end{theo}

 \begin{theo}
 \begin{theorem}[{\cite[Theorem 2]{R63III}}]\label{thm:angle split}
Let $\arc$ be an arc of bounded rotation in the plane, and let $z_0$ be an interior point of $\arc$ that splits it into
subarcs $\arc_1$ and $\arc_2$. Then 
\begin{equation}\label{eq:split turn}\kappa_l(\arc)=\kappa_l(\arc_1)+\kappa_l(\arc_2)+\pi - \theta_{l,\lambda}(z_0)~. \end{equation}
\end{theorem}
\end{theo}


Let us consider a triangulation in the case the surface is compact (see \cite{jost} for triangulations of surfaces). Denoting by 
$\mathcal{T}$,  $\mathcal{E}$, $\mathcal{V}$, respectively, the set of triangles (which are homeomorphic to open discs), edges (without their extremities) and vertices  of the triangulation, we have
$$\omega(S)=\sum_{T\in\mathcal{T}} \omega(T) + \sum_{E\in\mathcal{E}} \omega(E)+\sum_{V\in\mathcal{V}} \omega(V)~. $$
Using \eqref{eq:left+right}, Gauss--Bonnet Formula \eqref{eq:GB} and \eqref{eq:split turn}, a direct calculation gives the 
\emph{Gauss--Bonnet Formula for surfaces}:\index{Gauss--Bonnet formula (surfaces)}
\begin{equation}\label{eq:GB surface}
\mes(S)=2\pi\chi(S)~.
\end{equation}

Now the question is: does any signed measure satisfying Gauss--Bonnet Formula \eqref{eq:GB surface} define a subharmonic metric on a compact Riemann surface? We will need a result about the Laplacian.
Indeed,  the Euclidean Laplacian is not invariant under the action of a conformal mapping, but from
\eqref{eq:lapl composition conforme}, the differential form  $\Delta f \D x\D y$ 
is. We will denote by $\Delta_S$ the form over the Riemann surface $S$, which is such that in a chart,
$\Delta_S f=\Delta f \D x\D y$.  We will use the following result, whose name is taken from \cite{donaldson}. 
We refer to \cite{royden67}, \cite{donaldson}, \cite{chirka} for more details.

\begin{theorem}[Main Theorem for Compact Riemann Surfaces]\label{thm:main thm}
Let $\nu$ be a signed Borelian measure over a compact Riemann surface $S$, such that $\nu(S)=0$. Then there exists $v\in L^1(S)$ such that 
$\Delta_S v= \nu, $
that is, for any $\varphi \in C^\infty(S)$,
$$\int_S v \Delta_S \varphi = \int_S \varphi \D\nu~. $$
\end{theorem}


\begin{theo}
\begin{theorem}[{\cite[Theorem~7.3]{troyanov-alexandrov}}]
Let $S$ be a compact Riemann surface, and let $\mes$ be a signed measure, such that $\mes(S)=2\pi\chi(S)$. Then there exists a subharmonic metric over $S$ with curvature measure $\mes$.

The metric is unique up to multiplication by a positive constant.
\end{theorem}
\end{theo}

\begin{proof}
Let $\lambda_0|\D z|^2$ be any subharmonic metric over $S$ (see Remark~\ref{rem:construction sshmet}), with curvature measure $\mes_0$. By Gauss--Bonnet Formula \eqref{eq:GB surface}, we have $\mes_0(S)=2\pi\chi(S)$, in particular,
for $\nu=\mes-\mes_0$, $\nu(S)=0$. Let $v$ be the function given by the Main Theorem~\ref{thm:main thm}. The wanted metric is $\E^{-2\nu}\lambda_0|\D z|^2$. The uniqueness part is Proposition~\ref{prop:unique measure riem}.
\end{proof}

The argument above can be done using a conformal Riemann metric over the Riemann surface, together with the associated Riemannian Laplacian. The result follows using the potential of this Laplacian, see  \cite{troyanov-alexandrov}. Actually,
in a conformal chart into which the metric is $\lambda |\D z|^2$, the Riemannian Laplacian is $\frac{1}{\lambda}\Delta$. As the volume form of the Riemann metric is then $\lambda \D x\D y$, it follows that $\Delta_S$ is the Riemannian Laplacian times the volume form. 

%
%
%
%
%
%
%

Let us state   the most famous result about Riemann surfaces.
\begin{theorem}[Uniformization Theorem]\label{thm:unif}
Let $M$ be a connected, simply connected, non-compact, Riemann surface. Then there is a conformal bijective mapping from  $M$ onto the unit disc or to $\C$. 
\end{theorem}

The Uniformization Theorem~\ref{thm:unif} can be derived from Theorem~\ref{thm:main thm}, see \cite{donaldson}.   See also e.g., \cite{beardon}, \cite{farkas-kra}. The Uniformization Theorem implies that  
we can take  conformal charts ``as large as possible''. This appears useful for convergence result, because it avoids possible degeneration of domain of charts.

\begin{corollary}\label{cor:disque iso}
On a Riemann surface with a subharmonic metric, in any domain $U$ homeomorphic to a disc, there are isothermal coordinates, i.e., coordinates such that the metric can be written 
$\lambda|\D z|^2$.
\end{corollary}
\begin{proof}
Restricting to $U$ the conformal charts covering it turns $U$ into a Riemann surface. By the Uniformization Theorem, there is a conformal mapping from $U$ onto $X$, where $X$ is the unit disc or $\C$. 
Let $(U_1,f_1)$ be an isothermal chart with $U_1$ contained in $U$. By definition, there exists a conformal mapping from $f_1(U_1)$ into $X$. Let us endow the image of this mapping by the pushforward of the metric.
We do this for all the isothermal charts in $U$. 
\end{proof}

\begin{remark}{\rm \label{remark:angle coincident}
Let us check that the subharmonic angle \eqref{eq:split turn} between two shortest arcs, intersecting at only one point $z_0$, is the same than the angle we  defined previously (see  Remark~\ref{rem:angle}).\footnote{The following argument was suggested to the author by Alexander Lytchak.}  In the following, one of the sectors formed by the shortest arcs is implicitly chosen, and the angle is computed for the induced intrinsic distance on the sector. Also, we 
consider only angles $<\pi$ (see Theorem 12 p. 128 in \cite{AZ}), and also that $\mes(\{z_0\})<2\pi$.
By Theorem~\ref{thm:shortest boundedturn} and Theorem~\ref{thm:BT BR}, $K$ and $L$ are arcs of bounded rotation, in particular they have semitangent lines at extremities. We will use a canonical stretching of $\rho_\lambda$, and refer to Section~\ref{sec can stretch} for the notation. The result of the canonical stretching 
is a flat cone  with curvature measure 
$\mes(\{z_0\})\delta_0$. It is easy to see that  the image of $K$ and $L$ under the stretching are half lines, namely the respective derivative of $K$ and $L$ at their starting points. It follows from its definition \eqref{eq:left angle} that the subharmonic  angle for the distance $\rho_\lambda$ between $K$ and $L$ is equal to the subharmonic angle  between the images of $K$ and $L$ for the flat cone distance.  
 Clearly, on a flat cone, for such curves, both the subharmonic  angle and the angle \eqref{eq upper angle}, denoted by  $\overline{\alpha}$,  coincide. So we are left by proving that the angle $\alpha$ for $\rho_\lambda$ is equal to  $\overline{\alpha}$.

Let $\epsilon>0$. 
Let $X$ and $Y$ be points on $K$ and $L$. For $r$ sufficiently small, 
if $X$ and $Y$ are at Euclidean distance $r$ from $z_0$, then
by  Theorem~\ref{prop:fun estimate} they are close to $z_0$ for the distance $\rho_\lambda$, because $\mes(\{z_0\})<2\pi$. 
Hence, by definition of the angle, 
we may suppose that  $|\alpha-\alpha_0|<\epsilon$,
where $\alpha_0$ is the angle at the vertex corresponding to $z_0$ of the Euclidean comparison triangle of $z_0XY$.
Under the inverse of the mapping $z\mapsto z_0+rz$, $X_r$ and $Y_r$ are sent to points $X'_r$ and $Y'_r$ on the unit Euclidean circle. As the distance $\lambda_{r,z_0}$ is obtained by a similarity from $\rho_\lambda$, it follows that the angle of the comparison triangle $0X'_rY'_r$ is equal to $\alpha_0$. As by Theorem~\ref{thm:canoncial streching}, $\rho_{\lambda_{r,z_0}}$ converge uniformly to the flat cone metric, comparison triangles converge. In turn, $\alpha_0$ is equal to the corresponding angle of the comparison triangle of the flat cone. But for such a metric space, this is $\overline{\alpha}$. At the end of the day, $|\alpha-\overline{\alpha}|<\epsilon$, and $\epsilon$ was arbitrary.
}\end{remark}

\begin{remark}{\rm \label{remark:unicite geod courbure negative}
Let us consider a subharmonic distance over a plane domain homeomorphic to a disc, with non-positive curvature measure.
In this case,   shortest arcs have non-positive turn (Theorem~\ref{thm:shortest boundedturn}), and the topology induced by the distance is the topology of the plane (Theorem~\ref{thm:met coinc}). In turn, from the formula of Theorem~\ref{thm:angle split}, together with Gauss--Bonnet Formula \eqref{eq:GB}, one can see that there is a unique shortest arc between distinct points.  
}\end{remark}\section{Other results}\label{sec-other}
In this last section, we present some more results about subharmonic distances, mainly related to results about two-dimensional manifolds of bounded curvature, see Section~\ref{sec:sh est bic}.

\subsection{Isothermal coordinates on two-dimensional manifolds of bounded curvature}\label{sec:iso coord on bic}

We want a  converse to Theorem~\ref{thm:sh is bic}. We will use the following  converse to Theorem~\ref{thm:approx poly implique bic}.

\begin{theorem}\label{thm:pol approx of bic}
Any point in a domain endowed with a distance  of bounded curvature has a neighborhood that admits a polyhedral approximation.\footnote{Here and in similar statements, the distance that is approximated is the induced intrinsic distance over the neighborhood.}
\end{theorem}

Theorem~\ref{thm:pol approx of bic} was originally proved in Chapter~III of \cite{AZ} (see  Remark 2 p. 88 and Theorem 15 p. 134 and p. 90 there), and the proof is also sketched in \cite{R93} (see Theorem~6.2.1 there). The general idea is to decompose a neighborhood of a point by non-overlapping simple triangles with arbitrary small diameters, to replace each of the triangles by a Euclidean triangles with same edge length (so we get a polyhedral distance), and check that this new distance is close to the original one for the uniform topology. Note that this construction is not the same than the one to obtain polyhedral approximation of subharmonic distances (Theorem~\ref{tm:poly approx}). 

Unfortunately it seems that, beside the fact that the arguments are sometime hard to follow in \cite{AZ} and \cite{R93}, some of them are false. This was pointed out in \cite{creutz2021triangulating}.
Fortunately, a stronger result is proved in this latter reference.

\begin{theorem}[{Creutz--Romney \cite{creutz2021triangulating}}]  \emph{Any} intrinsic finite distance on a plane domain can be decomposed by simple non-overlapping triangles with arbitrary small diameters.
\end{theorem}

Let us emphasize the fact that these results are not about triangulation in the usual sense, as for example a vertex of a triangle may be contained in the interior of the edge of another triangle.

Actually, it is noted in \cite{R93} that a stronger version of Theorem~\ref{thm:pol approx of bic} holds, see Theorem~6.2.1 there.

\begin{theorem}\label{thm:poly approx bic}
Any point in a domain endowed with a distance  of bounded curvature has a neighborhood that admits a polyhedral approximation such that the total variations of the turn of the boundary of the polyhedra are uniformly bounded.
\end{theorem}

Theorem~\ref{thm:poly approx bic} and Theorem~\ref{thm:limit met bc disque} give the following result, which the other half of the main motivation 
 for introducing subharmonic distances at the time, see Theorem~\ref{thm:sh is bic}.

\begin{theo}
\begin{theorem}[{\cite[Theorem II]{R60I} }]\label{thm:bic is sh}
Any point in a metric space of bounded curvature has a neighborhood homeomorphic to a closed disc that is isometric to a subharmonic distance.\footnotemark
\end{theorem}\end{theo}\footnotetext{Here also, the distance over the neighborhood is the induced intrinsic distance. Actually, a stronger result is stated in \cite{R60I}, namely that the interior of \emph{any} neighborhood homeomorphic to a closed disc has isothermal coordinates, see Corollary~\ref{cor:disque iso}. In  \cite{R60I} a stronger (smooth) version of Theorem~\ref{thm:poly approx bic} is stated and used.}

Theorem~\ref{thm:bic is sh} is a generalization of the classical result that Riemannian surfaces 
 admits isothermal coordinates, see Remark~\ref{rem:C1C2curvature}.

Note that in \cite{R60I}, results about Riemannian approximation of distances of bounded curvature are presented. We don't know if their proofs are available somewhere, but we can obtain them \emph{a posteriori}. Indeed, as we now know that, locally, (finite) subharmonic distances and distances of bounded curvature are locally isometric, we can uses the results of Section~\ref{sec cv}.
\begin{theo}
\begin{theorem}[{\cite[Theorem A and Theorem B]{R60I}}]
An intrinsic finite distance over a plane domain inducing the same topology is a distance of bounded curvature if and only if each point has a neighborhood  that can be uniformly approximated by Riemannian distances with uniformly bounded absolute curvature.

Moreover, for a distance of bounded curvature, such an approximating sequence can be chosen so that the total variations of the turn of the boundaries are uniformly bounded.
\end{theorem}
\end{theo}

A more direct way to prove that a distance of bounded curvature can be locally approximated by Riemannian distances with uniform bound of the variation of the curvature measure and turn of the boundary is to ``smooth'' polyhedral distances and use Theorem~\ref{thm:poly approx bic}. This is done at page 116 of \cite{R93}.

\subsection{Turn and curvature in a two-dimensional manifold of bounded curvature}

This section gives some details about $\S 8$ of \cite{R63III}.

The \emph{(left) turn of a curve $\arc$ in a two-dimensional manifold of bounded curvature},\footnote{In \cite{AZ}, what we are calling \emph{turn} was translated by \emph{rotation}, see the Introduction of the present article.}\index{turn  (in a two-dimensional manifold of bounded curvature)} $\tau_l$ is  defined by approximation by geodesic broken lines. Actually, this definition of the turn is formally   the same as \eqref{eq:def limite turn}, with instead of subharmonic angle,
the angle between shortest paths issuing from a same point.
It follows that, as notions of angles coincide in this setting (see Remark~\ref{remark:angle coincident}), we have the following result.\footnote{The fact that the angles coincide is implicit in $\S 8$ of \cite{R63III}.}

\begin{theo}
\begin{theorem}[{\cite[\S 8]{R63III} }]\label{thm turn coincide}
The notions of turn of a curve for a finite subharmonic distance and for a two-dimensional manifold of bounded curvature coincide.
\end{theorem}
\end{theo}

For a two-dimensional manifold of bounded curvature, a positive measure $\mes^+_{\operatorname{AZ}}$ is defined, saying that, for an open set $G$, the quantity $\mes^+_{\operatorname{AZ}}(G)$ is the supremum for all the decomposition of $G$ by a certain kind of non-overlapping triangles of the sum of the positive excesses (see \eqref{eq:upper excess}). 
Here the triangles under consideration must be of a certain kind, and the angles are not exactly the ones given by \eqref{eq:upper excess}, as in the definition of $\mes^+_{\operatorname{AZ}}(G)$, it is the angles between  the ``sectors'' defined by two edges that enter the picture. 
We won't enter details here and refer to \cite{AZ} for precise definitions. Another positive measure $\mes^-_{\operatorname{AZ}}$ is defined, by considering negative excesses instead of positive ones, and finally the 
\emph{curvature (in the Alexandrov sense)}\index{curvature (in the Alexandrov sense)} is the measure $\mes_{\operatorname{AZ}}:=\mes^+_{\operatorname{AZ}}-\mes^-_{\operatorname{AZ}}$. The curvature is related to the turn in the following way.

\begin{theorem}[{Gauss--Bonnet Formula, \cite[Theorem 5 p.192]{AZ} }]
Let $\arc$ be a simple closed curve in a two-dimensional manifold of bounded curvature, bounding a domain $D$ homeomorphic to a disc, then (for a suitable orientation)
$$\mes_{\operatorname{AZ}}(D)+\tau_l(\arc)=2\pi~. $$
\end{theorem}

As the notions of turn coincide for both two-dimensional manifolds of bounded curvature and subharmonic distances, and as 
Gauss--Bonnet Formula holds for both notions of curvature 
(see \eqref{eq:GB}),\footnote{Actually, \eqref{eq:GB} is stated only for regular arcs. It is easy to extend to arcs of the class $\Delta$, i.e., arcs having 
semitangent at each point, see \cite{R93} before Theorem~8.1.7. Anyway, when the two Borelian measures coincide on Euclidean balls, they are equal, see e.g., \cite{measure-ball}, \cite[2.20]{mattila}.} we obtain the following.

\begin{theo}
\begin{theorem}[{\cite[\S 8]{R63III} }]\label{thm measures coincide}
The notions of curvature measure for a finite subharmonic distance and for a two-dimensional manifold of bounded curvature coincide.
\end{theorem}
\end{theo}

Let us mention another argument to prove Theorem~\ref{thm measures coincide}, which is based on the following result.

\begin{theorem}[{Convergence of Curvature, \cite[Theorem 6 p. 240]{AZ}}]\label{thm: cv measure AZ}
Let a sequence of two-dimensional manifolds of bounded curvature converge uniformly to a two-dimensional manifold of bounded curvature, such that the total variation of the curvatures is uniformly bounded. Then the curvatures  converge weakly to the curvature of the limit metric.\footnote{In Theorem~\ref{thm: cv measure AZ} and Theorem~\ref{thm cv area},  weak convergence is for functions with a same compact support, see Remark~\ref{rem:def weak cv mes}.}
\end{theorem}

Indeed, a first observation is that, if a two-dimensional manifold of bounded curvature  (resp. a subharmonic distance) comes from a Riemannian metric, then it is clear that notions of angles coincide, and there is a unique notion of turn of a curve. It follows that both notions of measure coincide. 
Then it suffices to consider a Riemannian approximation of the subharmonic distance with convergence of the curvature measure (Theorem~\ref{thm:approx lisse met}),
and Theorem~\ref{thm: cv measure AZ} implies  Theorem~\ref{thm measures coincide}.

We end this section with the following question.
\begin{question}{Question}
Is it possible to prove the Convergence of Curvature Theorem~\ref{thm: cv measure AZ}  within the framework of subharmonic distances? 
\end{question}

\subsection{Lipschitz approximation and the area measure}
\label{sec:area sh}

The \emph{area}\index{area measure} $\sigma$\index{$\sigma$} of a two-dimensional manifold of bounded curvature  is a positive measure that is defined by polyhedral approximation, see \cite{AZ} for details. We won't say more about the definition, but we will use the following result.

\begin{theorem}[{Convergence of Area, \cite[Theorem 9 p.269]{AZ}}]\label{thm cv area}
Let a sequence of distances of bounded  curvature on a plane domain converge uniformly to a distance of bounded curvature, such that the total variation of the curvature is uniformly bounded. Then the areas converge weakly.
\end{theorem}

\begin{question}{Question}
Is it possible to prove the Convergence of Area Theorem~\ref{thm cv area} within the framework of subharmonic distances? 
\end{question}

We say that a sequence $(\rho_n)_n$ of distances over a same set $M$ \emph{converges proportionally}\index{proportional convergence}  to a distance $\rho$ if, uniformly,
$\frac{\rho_n}{\rho}\to 0$, or, equivalently, for any $\epsilon >0$, for all $x,y\in M$ there is $N$ such that when $n\geq N$,
\begin{equation}\label{eq:def proportional}(1-\epsilon) \rho(x,y) \leq \rho_n(x,y) \leq (1+\epsilon) \rho(x,y)~. \end{equation}

Proportional convergence is nowadays called \emph{Lipschitz convergence}\index{Lipschitz convergence}, and, for compact metric spaces, implies uniform convergence, see e.g., \cite{bbi}.

\begin{theorem}[Burago Lipschitz Approximation]\label{them:burago}
On a compact two-dimensional manifold of bounded curvature, if the curvature is 
$<2\pi$ at any point, then there exists a 
sequence of polyhedral metrics converging proportionally to it. 
\end{theorem}

Burago Lipschitz Approximation Theorem~\ref{them:burago} is a consequence of the statements of Theorem~28 and Theorem~29 in \cite{Res1959}. Unfortunately, proofs of these theorems were never published \cite{res-personal}. A sketch of the proof of Theorem~\ref{them:burago} was then published by Yu.~D.~Burago in \cite{burago-proportional}. Finally, a complete proof of Theorem~\ref{them:burago} was published by Burago in \cite{burago-lip}, see Lemma~6 there.   
\begin{question}{Question}\label{ques:brago}
Is it possible to prove the Burago Lipschitz Approximation Theorem~\ref{them:burago} within the framework of subharmonic distances? 
\end{question}
A proof of Theorem~28 and Theorem~29 in \cite{Res1959}  would provide a positive answer to Question~\ref{ques:brago}.

\begin{remark}{\rm It was noted to the author by Giona Veronelli that there is no  Lipschitz approximation of any  
compact two-dimensional manifold of bounded curvature by Riemannian surfaces.
}\end{remark}

From \eqref{eq:def proportional}, the identity mapping is a bi-Lipschitz mapping from $(M,\rho_n)$ to $(M,\rho)$. Hence, under the assumption of  Burago Lipschitz Approximation Theorem~\ref{them:burago}, it follows from Lemma~\ref{lem:haus lip} that the associated sequences of two-dimensional Hausdorff measures converge weakly. Together with  the Convergence of Area Theorem~\ref{them:burago}, one obtains the following statement.  

\begin{theorem}[Stratilatova]\label{thm: stratilatova}
The area measure of a two-dimensional manifold of bounded curvature is the two-dimensional Hausdorff measure.
\end{theorem}

It is noted in \cite[p. 266]{AZ} that Theorem~\ref{thm: stratilatova} was first proved by M.V. Stratilatova in \cite{Stratilatova}.

In turn, Convergence of Area Theorem~\ref{thm cv area} is a result about weak convergence of Hausdorff measure under an assumption of uniform convergence of the distances. 
We can derive the following result (see Remark~\ref{rem:huber rem aires}). Another proof can be found in \cite[Section 8]{lytchak-wenger}.

\begin{theorem}\label{thm:formule aire}
Let $(M,\rho_\lambda)$ be a subharmonic distance over a bounded plane domain. Then the two-dimensional Hausdorff measure of a Borel subset $E$ is equal to $\iint_E \lambda$.
\end{theorem}
\begin{proof}
We use a Riemannian approximation given by Theorem~\ref{thm:approx lisse met}. By Lemma~\ref{lem:hausdorff lisse} and Theorem~\ref{thm cv area}, it suffices to check that for a Borel set $E$, $\iint_E \lambda_n \to \iint_E \lambda$.
The result follows by the Dominated Convergence Theorem (using the fact that $-p(\mes^+_n)\leq - p(\mes^+)$ by Proposition~\ref{prop:approx potentiel}, and Lemma~\ref{lem: pmu lower} for $p(\mes^-)$).
\end{proof}

\subsection{Non-positive curvature and isoperimetric inequality}

Let us end this section mentioning a last result about subharmonic distances. Note that the original result is more general. 
\begin{theo}
\begin{theorem}[{\cite[Theorem 1]{R61} }]
Let  $\rho_\lambda$ be a subharmonic distance on a plane domain $M$, with area measure $\sigma_\lambda$. Let us suppose that it \emph{satisfies the isoperimetric inequality}, i.e.,  any compact set $F\subset M$ has finite area, and  there is $\delta>0$ such that for any $z_0\in F$ and $r<\delta$,
$$s^2_\lambda(C_r(z_0))  \geq 4\pi \sigma_\lambda(Q_r(z_0))~. $$

Then the curvature measure of $\rho_\lambda$ is non-positive.
\end{theorem}
\end{theo}
Heuristically, the other direction goes at follows. 
Let us consider a subharmonic distance $\rho_\lambda$ in a plane domain. Let us suppose that its curvature measure is non-positive. In other words, let us suppose that 
$-\ln \lambda$ is subharmonic.\footnote{Functions such that their logarithm is subharmonic are called \emph{PL}\index{PL}, see e.g., \cite{rado,B-rado}.}
Suppose that there is a harmonic function $h$  over $M$ such that $-\ln \lambda=h$ over $C_r(z_0)$ (see e.g. \cite[Exercise 4.5.2]{ransford}). Then $-\ln \lambda \leq h$ over $Q_r(z_0)$ (see Remark~\ref{rem:harmonic majorant}). Let us denote 
$\lambda_0=\E^{h}$. Then
$$s_{\lambda}^2(C_r(z_0))=s_{\lambda_0}^2(C_r(z_0)) \geq
4\pi \iint_{Q_r(z_0)}\lambda_0 \geq 4\pi \iint_{Q_r(z_0)}\lambda
=4\pi \sigma_\lambda(Q_r(z_0))~,$$
where  the first $\geq$ is the isoperimetric inequality in the Euclidean plane, as $\lambda_0 |\D z|^2$ is a flat metric, and the last equality is Theorem~\ref{thm:formule aire}.
This is a classical observation, see \cite{ivan}, \cite{lytchak-wenger} for more references.

There is another classical notion of metric spaces with non-positive curvature, which are locally compact intrinsic metric spaces, such that the upper-angle  at the vertex of any triangle in a given neighborhood is not larger that the angle at the corresponding vertex of a comparison triangle \cite{bbi}, \cite{BH}, \cite{alexander-kapovich-petrunin}. 
To avoid confusion, for this latter class of metric space, we will
speak about \emph{non-positive curvature in the sense of Alexandrov}\index{non-positive curvature in the sense of Alexandrov}. 
It follows from the definition that such a distance, say on a plane domain, is a distance of bounded  curvature, and that the curvature measure  is then non-positive.

A theorem of Alexandrov \cite{alex58} says that 
an isoperimetric inequality holds in metric spaces with non-positive curvature in the sense of Alexandrov (see \cite{BH} for a proof). It may also follows from the \emph{Reshetnyak Majorization Theorem}, which is, loosely speaking, an enhancement of the Contraction Onto a Cone Theorem~\ref{thm contraction} for metrics spaces with upper curvature bound. See \cite{resh68}, \cite{alexandrov-book} for a proof.  It is stated (for the case of bounded specific curvature, see below) as Theorem~13 in \cite{Res1959}. 
However, it was then proved in \cite{lytchak-wenger} that, loosely speaking,  isoperimetric inequality (which has to be defined in this general case) characterizes metric spaces with non-positive curvature in the sense of Alexandrov.

Similar definitions and results hold for \emph{metric space with curvature bounded from above by $\curv$ in the sense of Alexandrov}, where $\curv$ is a real number. 
For $\curv<0$, as for $\curv=0$, it is immediate from the definition that such a distance on a plane domain is a distance of bounded curvature. It was considered as folklore that the same fact holds for $\curv>0$, but explicit arguments were recently given in \cite{romney2}.

There is also a symmetric notion of \emph{metric space with curvature bounded from below by $\curv$ in the sense of Alexandrov}. Over a plane domain, they also are distances of bounded curvature, see \cite{richard}.




\printindex

\bibliographystyle{apalike}

\bibliography{resh}

\end{document}